\input ./style/arxiv-general.cfg
\documentclass[aop,MSNbibl,dvips]{arximspdf}
\makeatletter
   \@ifpackageloaded{graphicx}{}{\usepackage{graphicx}}
\makeatother
\usepackage{mathbh}

\doi{10.1214/14-AOP989}
\volume{44}
\issue{2}
\pubyear{2016}
\firstpage{807}
\lastpage{866}
\docsubty{FLA}

\makeatletter
\newcommand{\rrvert}{\vert}
\newcommand{\llvert}{\vert}
\def\xrightarrow{\rightarrow}
\newcommand{\iint}{\int\!\!\!\int}
\newcommand{\eqref}[1]{(\ref{#1})}
\newtheorem{thmm}{Theorem}[section]
\newtheorem{lemma}[thmm]{Lemma}
\newtheorem{corollary}[thmm]{Corollary}
\newtheorem{prop}[thmm]{Proposition}

\newproclaim{defn}[thmm]{Definition}
\newproclaim{example}[thmm]{Example}
\newproclaim{examples}[thmm]{Examples}
\newproclaim{rem}[thmm]{Remark}

\newcommand{\R}{\mathbb{R}}
\newcommand{\N}{\mathbb{N}}
\newcommand{\E}{\mathbb{E}}
\newcommand{\p}{\mathbb{P}}

\newcommand{\calB}{\mathcal{B}}
\newcommand{\calC}{\mathcal{C}}
\newcommand{\calF}{\mathcal{F}}
\newcommand{\calH}{\mathcal{H}}
\newcommand{\calL}{\mathcal{L}}
\newcommand{\calM}{\mathcal{M}}

\newcommand{\1}{\mathbh{1}} 

\newcommand{\ssup}[1] {{({#1})}}
\newcommand{\sse}[1] {{[{#1}]}} 
\newcommand{\eps}{\varepsilon}

\newcommand{\ra}{\rightarrow}
\newcommand{\da}{\downarrow}

\newcommand{\tem}{\mathrm{tem}}
\newcommand{\rap}{\mathrm{rap}}
\renewcommand{\rho}{\varrho}

\newcommand{\blue}{\mathrm{\mathtt{ blue} }}
\newcommand{\red}{\mathrm{\mathtt{ red }}}
\newcommand{\dd}{ d} 

\makeatother

\begin{document}
\begin{frontmatter}

\title{The scaling limit of the interface of the continuous-space
symbiotic branching model\thanksref{T1}}
\runtitle{The scaling limit of the SBM}
\thankstext{T1}{Supported by the DFG research Grant No. BL 1105/2-1 and
the DFG Priority Programme 1590 ``Probabilistic Structures in
Evolution,'' Grant No. BL 1105/4-1.}

\begin{aug}
\author[A]{\fnms{Jochen}~\snm{Blath}\ead[label=e1]{blath@math.tu-berlin.de}\thanksref{m1}},
\author[A]{\fnms{Matthias}~\snm{Hammer}\ead[label=e2]{hammer@math.tu-berlin.de}\thanksref{m1}}
\and
\author[B]{\fnms{Marcel}~\snm{Ortgiese}\corref{}\ead[label=e3]{marcel.ortgiese@uni-muenster.de}\thanksref{m2}}
\runauthor{J. Blath, M. Hammer and M. Ortgiese}
\address[A]{J. Blath\\
M. Hammer\\
Institut f\"ur Mathematik\\
Technische Universit\"at Berlin\\
Stra\ss e des 17. Juni 136\\
10623 Berlin\\
Germany\\
\printead{e1}\\
\phantom{E-mail:\ }\printead*{e2}}
\address[B]{M. Ortgiese\\
Institut f\"ur Mathematische Statistik\\
Westf\"alische Wilhelms-Universit\"at M\"unster\\
Einsteinstra\ss{}e 62\\
48149 M\"unster\\
Germany\\
\printead{e3}}
\affiliation{Technische Universit\"at Berlin\thanksmark{m1} and
Westf\"alische~Wilhelms-Universit\"at M\"unster\thanksmark{m2}}

\end{aug}

%
\received{\smonth{12} \syear{2013}}
%
\revised{\smonth{11} \syear{2014}}

%
\begin{abstract}
The continuous-space symbiotic branching model describes
the evolution of two interacting populations that can reproduce locally
only in the
simultaneous presence
of each other.
If started with complementary Heaviside initial conditions, the
interface where both populations coexist remains
compact.
Together with a diffusive scaling property, this suggests the presence
of an interesting
scaling limit.
Indeed, in the present paper, we show weak convergence of the diffusively
rescaled populations as measure-valued processes in the Skorokhod,
respectively the Meyer--Zheng,
topology (for suitable parameter ranges).
The limit can be characterized as the unique solution to a martingale
problem and
satisfies a ``separation of types'' property.
This provides an important step toward
an understanding of the scaling limit for the interface.
As a corollary, we obtain an estimate
on the moments of the width of an approximate interface.
\end{abstract}

%

%
\end{frontmatter}

\section{Introduction}\label{sec1}

\subsection{The symbiotic branching model and its interface}

The symbiotic branching model of Etheridge and Fleischmann \cite{EF04} is
a spatial stochastic model of two interacting populations
described by the nonnegative solutions of
the stochastic partial differential equations
%
\begin{equation}
\label{eqn:spde} {\operatorname{cSBM}(\varrho,
\gamma)}_{u_0, v_0}\dvtx%
\cases{\displaystyle \frac{\partial}{\partial t}
u_t(x) = \frac{\Delta}{2} u_t(x) + \sqrt{ \gamma
u_t(x) v_t(x)} \dot{W}^\ssup{1}_t(x),
\vspace*{2pt}
\cr
\displaystyle\frac{\partial}{\partial t}v_t(x) = \frac{\Delta}{2}
v_t(x) + \sqrt{ \gamma u_t(x) v_t(x)}
\dot{W}^\ssup{2}_t(x),
}
\end{equation}
with suitable nonnegative initial conditions $u_0(x) \ge0, v_0(x) \ge
0, x \in\R$.
Here, $\gamma> 0$ is the branching rate, and $ (\dot{W}^\ssup{1},
\dot{W}^\ssup{2})$ is a pair of correlated standard Gaussian white
noises on $\mathbb{R}_+ \times\mathbb{R}$ with correlation $\varrho
\in[-1,1]$, that is, for $t_1,t_2 \geq0$,
%
\begin{equation}
\label{correlation} \mathbb{E} \bigl[ W^\ssup{i}_{t_1}(A_1)W^\ssup{j}_{t_2}(A_2)
\bigr] = %
\cases{ (t_1\wedge t_2)
\ell(A_1\cap A_2), &\quad $i=j$,
\cr
\varrho(t_1
\wedge t_2) \ell(A_1\cap A_2), &\quad $i\neq j$,}
\end{equation}
where $\ell$ denotes the Lebesgue measure and $A_1,A_2$ are Borel sets.
Solutions of this model have been considered rigorously in the
framework of the corresponding martingale problem in Theorem~4 of \cite
{EF04}, which states that, under natural conditions on the initial
conditions $u_0(\cdot), v_0(\cdot)$,
a solution exists for all $\varrho\in[-1, 1]$. Further, the
martingale problem is well-posed for all $\varrho\in[-1,1)$, which
implies the strong Markov property except in the boundary case $\varrho=1$.
The model interpolates between several well-known examples of spatial
population models.
Indeed, for $\varrho=-1$ and $u_0 = 1- v_0$,
the system reduces to the continuous-space \emph{stepping stone model},
discussed, for example, by Tribe in \cite{T95}. For $\varrho=0$, the
system is the so-called \emph{mutually catalytic model} of Dawson and
Perkins \cite{DP98}, and for $\rho= 1$ and
$u_0= v_0$, an instance of the \emph{parabolic Anderson model}; see,
for example, \cite{MuellerSupport91}.

Natural questions about such (systems of) SPDEs are related to their
long-term behavior,
for example, the limiting shape of the interface for suitable initial
conditions.
For us, of particular interest are
``complementary Heaviside initial conditions,'' that is,
\[
u_0(x) = \mathbf{ 1}_{\R^-}(x) \quad\mbox{and}\quad
v_0(x) = \mathbf{ 1}_{\R
^+}(x), \qquad x \in\R.
\]

\begin{defn}\label{def:ifc}
The interface at time $t$ of a solution $(u_t,v_t)_{t\ge0}$ of the
symbiotic branching model $\operatorname{cSBM}(\varrho,\gamma)_{u_0, v_0}$
with $\varrho\in[-1,1]$, $\gamma>0$ is defined as
\[
\operatorname{ Ifc}_t = \operatorname{cl} \bigl\{x\in\R\dvtx u_t(x)
v_t(x) > 0 \bigr\},
\]
where $\operatorname{cl}(A)$ denotes the closure of the set $A$ in $\R$.
\end{defn}

The main question addressed by Etheridge and Fleischmann \cite{EF04} is
whether for the above initial conditions, the ``compact interface
property'' holds, that is, whether the interface is compact at each
time almost surely. This is answered affirmatively in their Theorem~6,
together with the assertion that the interface propagates with at most
linear speed, that is,
there exists a constant $c = c(\gamma)$ such that
for each $\varrho\in[-1,1]$, there is a (almost-surely) finite random
time $T_0$ such that, almost surely, for all $T \ge T_0$,
%
\begin{equation}
\label{eq:linspeed} \bigcup_{t \le T} \operatorname{Ifc}_t
\subseteq [-cT, cT ].
\end{equation}
However, due to the scaling property of the symbiotic branching model
[see \eqref{scaling property} below], one might expect that the
fluctuations of the position of the interface should be of order
$t^{1/2}$. Indeed, Blath, D\"oring and Etheridge \cite{BDE11},
Theorem~2.11, strengthen the linear propagation bounds~\eqref{eq:linspeed}
for a (rather small) parameter range:

\begin{thmm}[(\cite{BDE11})]
\label{thmm:wavespeed}
There exists $\rho_0>-1$ such that the following holds: Suppose
$(u_t,v_t)_{t\ge0}$ is a solution to $\operatorname{cSBM}(\varrho,\gamma
)_{1_{\R^-}, 1_{\R^+}}$ with $-1 < \varrho< \varrho_0$. Then there is
a constant $C(\gamma, \varrho)>0$ and a finite random time $T_0$ such
that almost surely
\[
\bigcup_{t\leq T} \operatorname{Ifc}_t
\subseteq
 \bigl[-C\sqrt{T\log(T)},C\sqrt{T\log(T)} \bigr],
\]
for all $T>T_0$.
\end{thmm}

The restriction to $\varrho<\varrho_0$ seems artificial and comes from
the technique of the proof. Although the value of $\varrho_0\approx
{-0.9958}$ is {rather} close to $-1$, the result is remarkable, since
it shows that {sub-linear speed of propagation} is not restricted to
situations in which solutions are uniformly bounded as, for instance,
for $\varrho=-1$. The proof is based on the ``dyadic grid technique''
in~\cite{T95}
together with improved bounds on the
moments of the symbiotic branching model (see the ``critical curve'' in
Theorem~\ref{thmm:mc} below),
circumventing the lack of uniform boundedness of the population sizes.

The symbiotic branching model exhibits the following fundamental \emph
{scaling property} (see Lemma~8 of \cite{EF04}):
If $(u_t, v_t)_{t\ge0}$ is a solution to $\operatorname{cSBM}(\rho, \gamma
)_{u_0, v_0}$,
then
%
\begin{equation}
\label{scaling property} \bigl(u^\ssup{K}_t(x), v^\ssup{K}_t(x)
\bigr):= \bigl(u_{K^2t}(K x), v_{K^2t}(K x) \bigr), \qquad x \in\R, K
>0
\end{equation}
is a solution to $\operatorname{cSBM}(\rho, K \gamma)_{u^\ssup{K}_0,
v^\ssup
{K}_0}$ with initial states $(u^\ssup{K}_0, v^\ssup{K}_0)$
transformed accordingly.
Note that complementary Heaviside initial conditions $(u_0,v_0) = ({\1
}_{\R^-}, {\1}_{\R^+})$ are invariant under this rescaling. Thus in
this case letting $K\to\infty$ in~\eqref{scaling property} is
equivalent to increasing the branching rate $\gamma\to\infty$.

In light of the
scaling property \eqref{scaling property}, one might hope
that (at least for a suitable range of parameters) a diffusive
rescaling could lead to an interesting scaling limit.
In fact, the program of letting the branching rate tend to infinity has
been carried out for the \emph{discrete space} version of~(\ref{eqn:spde}).
For the mutually catalytic model (the case $\rho= 0$), Klenke and
Mytnik construct in a series of papers~\cite{KM10a,KM11b,KM11c}
a nontrivial limiting process for $\gamma\to\infty$ (on the lattice)
and study its
long-term properties.
This limit is called the ``infinite rate mutually catalytic branching process.''
Moreover, Klenke and Oeler \cite{KO10} give a Trotter-type approximation
and \emph{conjecture} that, under suitable assumptions, a nontrivial
interface for the limiting process exists; see page~485, before Corollary~1.2.
Recently, analogous results have been derived by D\"oring and Mytnik in
the case
$\varrho\in(-1, 1)$ in~\cite{DM11a,DM12}.

Returning to the continuous-space set-up, for $\rho=-1$ (the stepping
stone model) Tribe \cite{T95} proves a ``functional limit theorem'': For
a pair of (continuous) functions $(u,v)$,
define
%
\begin{equation}
\label{defn:R_L} R(u,v) := \sup\bigl\{x \dvtx u(x) > 0\bigr\},\qquad
 L(u,v) = \inf\bigl\{x
\dvtx v(x) > 0\bigr\}.
\end{equation}
Note that for a solution $(u_t, v_t)_{t\ge0}$ of the symbiotic
branching model, the interface at time $t$ is contained in the interval
$[L(u_t, v_t), R(u_t, v_t)]$.
It is proved in~\cite{T95} for $\varrho=-1$ and for continuous initial
conditions $u_0 = 1- v_0$ which
satisfy $-\infty< L(u_0,v_0) \leq R(u_0,v_0) < \infty$ that under
Brownian rescaling, the motion of the position of the right endpoint of
the interface
$t \mapsto\frac{1}{n}R(u_{n^2 t}, 1-u_{n^2 t}), t\ge0$, converges
to a Brownian motion as $n \rightarrow\infty$.

The above results suggest the existence of an interesting diffusive
scaling limit for the continuous-space symbiotic branching model (and
its interface) for $\rho>-1$. This is
the starting point of our investigation. However, compared to the case
$\rho=-1$, the situation is more involved here: For example, the total
mass of the solution is not necessarily bounded, and in particular,
moments of the solution may diverge as $t\to\infty$, depending on
$\rho$.
For instance, second moments diverge for $\rho\geq0$.
In order to state this result, which was
obtained in \cite{BDE11}, we define the \emph{critical curve} $p\dvtx
(-1,1) \to(1, \infty)$ of the symbiotic branching model by
%
\begin{equation}
p(\varrho)=\frac{\pi}{ \arccos(-\varrho)} ,\label{criticalcurve}
\end{equation}
and denote its inverse by $\varrho(p) = - \cos(\frac{\pi}{p})$ (for
$p>1$). This curve separates the upper right quadrant into
two areas: below the critical curve, where moments remain bounded, and
above the critical curve, where moments
increase to infinity as $t \to\infty$:

\begin{thmm}[({\cite{BDE11}, Theorem~2.5})]\label{thmm:mc}
Suppose $(u_t,v_t)_{t\ge0}$ is a solution to the symbiotic branching
model with initial conditions $u_0=v_0\equiv1$. Let $\varrho\in(-1,1)$
and $\gamma>0$. Then, for every $x \in\R$,
\[
\varrho<\varrho(p) \quad\mbox{iff}\quad \E_{1,1} \bigl[u_t(x)^p
\bigr] \qquad\mbox {is bounded uniformly in all }t\ge0.
\]
In particular, if $\varrho<\varrho(p)$, there exists a constant
$C(\varrho)$ so that, uniformly for all $x \in\R$ and $t \ge0$,
\[
\E_{1,1} \bigl[u_t(x)^p \bigr] \le C(\varrho),\qquad
t \ge0.
\]
\end{thmm}

\begin{rem}\label{rem:starting_point}
(i) Of course, due to symmetry, the same result holds for the $v$ population.
The existence of a finite bound which is independent of $x$ follows
from the fact that the
system is translation invariant under the $(\mathbf{1},\mathbf{1})$
starting condition.

(ii)
In particular, for $\rho< \rho(4)= - \frac{1}{\sqrt{2}}$ and any
$x_1,\ldots, x_4$ we have by the generalized
H\"older inequality that
\[
\E_{1,1}\bigl[ u_t(x_1) u_t(x_2)
v_t(x_3) v_t(x_4) \bigr] \leq\max
_{i = 1, \ldots, 4} \E_{1,1}\bigl[ u_t(x_i)^4
\bigr] \leq C(\varrho),
\]
and similarly if some of the $v$'s are replaced by $u$ (and vice versa).
\end{rem}

The main tool of our approach is the use of several \emph{dual
processes} for the symbiotic branching model.
For the case $\rho=-1$ (heat equation with Wright--Fisher noise), Tribe
\cite{T95} uses the duality with coalescing Brownian motions.
In our case, we have to use instead a duality due to \cite{EF04} with a
system of colored Brownian particles with an exponential correction
term, involving collision local times.
Moreover, we will rely on an exponential \emph{self-duality} for
uniqueness. These dualities will be explained in detail below.

\subsection{Main results and open problems}\label{ssn:main}

We define the measure-valued processes
%
\begin{equation}
\label{defn:mu_nu} \mu_t^\ssup{n} (dx) := u_{n^2
t}(nx)
\,dx ,\qquad  \nu _t^\ssup {n} (dx) := v_{n^2 t}(nx) \,dx ,
\end{equation}
obtained by taking the diffusively rescaled solutions of $\operatorname
{cSBM}(\rho,\gamma)$ as densities.
We consider the pair $(\mu^\ssup{n}_t, \nu^\ssup{n}_t)_{t \geq0}$ as
random elements of $\calC_{[0,\infty)}(\calM_\tem^2)$, the space of
continuous
processes taking values in the space of (pairs of) tempered measures
endowed with the Skorokhod topology.
Loosely speaking, $\calM_\tem$ contains all the measures
whose integral against any nonnegative function that is decaying
exponentially fast at $\pm\infty$
is finite.
We recall the precise definition of $\calM_\tem$ and all other
necessary spaces of functions
and measures in
Appendix~\ref{appendix0}.
Our first main result reads as follows:

\begin{thmm}
\label{thmm:main1}
Assume $\varrho< \varrho(4) = -\frac{1}{\sqrt{2}}$. Let $(u_t,
v_t)_{t\ge0}$ be a solution to $\operatorname{cSBM}(\rho,\gamma)_{u_0,v_0}$
with complementary Heaviside initial conditions $(u_0, v_0) = (\1_{\R
^-}, \1_{\R^+})$. Then the
processes $(\mu_t^\ssup{n}, \nu_t^\ssup{n})_{t \geq0}$ converge weakly
in $\calC_{[0,\infty)}( \calM_\tem^2)$
to a limit $(\mu_t,\nu_t)_{t\ge0}$ which has the following properties:
\begin{itemize}
\item Absolute continuity: For each fixed $t>0$, $\mu_t$ and $\nu_t$
are absolutely continuous w.r.t. the Lebesgue measure $\ell$, $\p$-a.s.,
\[
\mu_t(dx)=\mu_t(x) \,dx, \qquad\nu_t(dx)=
\nu_t(x) \,dx,\qquad \p\mbox {-a.s.}\setcounter{footnote}{1}\footnote {For an absolutely continuous
measure, we will usually use the same symbol to denote the measure and its
density.}
\]
\item Separation of types: For each fixed $t>0$
the (absolutely continuous) measures $\mu_t$ and $\nu_t$ are mutually
singular: We have
%
\begin{equation}
\label{singularity1} \mu_t(\cdot) \nu_t(\cdot)=0, \qquad\p\otimes\ell
\mbox{-a.s.}
\end{equation}
\end{itemize}
\end{thmm}
%

\begin{rem}\label{re:2011-1}
%
%
(a) \emph{Identification of the limit}. For $\rho= -1$,
Tribe~\cite{T95} shows that the process $(\mu_t^\ssup{n}, \nu
_t^\ssup
{n})_{t \geq0}$
converges weakly to
\[
(\1_{\{x\le B_t\}} \,dx, \1_{\{x\ge B_t\}} \,dx)_{t \geq0} ,
\]
for $(B_t)_{t \geq0}$ a standard\vspace*{1pt} Brownian motion.
In our case, however, that is, for $\rho\in(-1,-\frac{1}{\sqrt{2}})$,
one can show that the limit $(\mu_t,\nu_t)_{t \geq0}$
cannot be of the form
\[
(\1_{\{x\le I_t\}} \,dx,\1_{\{x\ge I_t\}} \,dx)_{t \geq0}
\]
for a semimartingale $(I_t)_{t \geq0}$; see Remark~\ref
{rem:not_a_limit} below for more details.
Moreover, we remark that the limiting process in Theorem~\ref
{thmm:main1} is also not trivial, that is,
nondeterministic: If it were, then by the Green function
representation of the limit (see Corollary~\ref{cor:Green function
representation} below) it would have to be given by
\[
(\mu_t,\nu_t)=(S_tu_0,S_tv_0),
\]
which, however, violates the ``separation of types'' condition~\eqref
{singularity1}.

(b) \emph{Restrictions on $\rho$ and initial
conditions.}
Note that the restrictions on the range of parameters only comes from
our proof of
tightness for the rescaled solutions.
The decisive step is an estimate on the second moment of the integral
$\int u_t(x) v_t(x) \,dx$ that is uniform in time.
It is here that both assumptions $ \rho<-\frac{1}{\sqrt{2}}$ and
complementary
Heaviside initial conditions are essential.
The restriction on $\rho$ comes from the fact that
the second moment of the product $u_t v_t$ is really a
fourth moment, and we recall from Theorem~\ref{thmm:mc} that only
for $\rho< \rho(4) = -\frac{1}{\sqrt{2}}$ fourth moments (at a
single location)
remain bounded in time.
In fact, our technique would work in principle if we could control
$p$th moments for $p > 2$,
but our integer-moment particle system duality in combination with the
Burkholder--Davis--Gundy inequality
requires mixed fourth moments; see Lemma~\ref{tst_fn_tight}.

Similarly, the restriction to Heaviside initial conditions is due to
the technique of proof.
Only in this case can we control the expression obtained via the
particle dual.
Roughly speaking, we need the ``simple'' shape of the initial conditions
to be able to reduce the (spatial) integrals
to pointwise estimates that can be controlled via Theorem~\ref{thmm:mc}.

More generally, it
seems conceivable (although probably technically much more involved)
that one can deal with initial conditions of the type
$u_{0,n} = \1_{(-\infty, an]} + \1_{[bn, cn]}, v_{0,n} =
\1_{[an,bn]} + \1_{[cn, \infty)}$ (and its obvious generalizations
to several blocks; cf. also~\cite{T95} for $\rho= -1$).
It seems difficult to
go beyond this class, and
it is clear that the ``overlap'' of the support of the initial
conditions needs to vanish sufficiently quickly
(for $n \ra\infty$) for the moment bound to hold.

Note that we can relax both assumptions to any $\rho< 0$ and
general initial conditions if we allow a weaker topology than the
Skorokhod topology on $\calC_{[0,\infty)}$;
see Theorem~\ref{thmm:MPinf} below.
%
\end{rem}

Unfortunately, we do not yet have a fully explicit representation of
the limiting process of Theorem~\ref{thmm:main1} as in \cite{T95}. We
can, however, characterize it as the unique solution to a certain
martingale problem.
For the (standard) notation we again refer the reader to Appendix~\ref
{appendix0}.

\begin{defn}[{[Martingale problem $(\mathbf{MP})_{\mu_0,\nu_0}^\rho
$]}]\label{defn:MP}
Fix $\rho\in[-1,1]$ and (possibly random) initial conditions $(\mu
_0,\nu
_0)\in\calM_\tem^2$ (resp., $\calM_\rap^2$).
A continuous $\calM_\tem^2$-valued (resp., $\calM_\rap^2$-valued)
stochastic process $(\mu_t, \nu_t)_{t\ge0}$
is called a solution to the martingale problem $(\mathbf{MP})_{\mu
_0,\nu_0}^\rho$ if
there exists a continuous $\calM_\tem$-valued (resp., $\calM_\rap$-valued)
process $(\Lambda_t)_{t\ge0}$
such that
for each test function $\phi\in\calC_{\rap}^{(2)}$ (resp., $\phi
\in
\calC^{(2)}_{\tem}$),
the process $(M(\phi),N(\phi))$ defined by
%
\begin{eqnarray}
\label{MP1} M(\phi)_t&:=&\langle\mu_t,\phi\rangle-
\langle\mu_0, \phi\rangle - \int_0^t
\biggl\langle\mu_s, \frac{1}{2}\Delta\phi\biggr\rangle \,ds,
\nonumber
\\[-8pt]
\\[-8pt]
\nonumber
N(\phi)_t&:=&\langle\nu_t,\phi\rangle- \langle
\nu_0, \phi\rangle - \int_0^t
\biggl\langle\nu_s, \frac{1}{2}\Delta\phi\biggr\rangle \,ds
\end{eqnarray}
is a pair of continuous square-integrable martingales null at zero with
covariance structure
%
\begin{eqnarray}
\label{Cov1} \bigl[ M(\phi),M(\phi)\bigr]_t & =& \bigl[ N(\phi),N(
\phi)\bigr]_t = \bigl\langle\Lambda_t, \phi ^2
\bigr\rangle ,
\nonumber
\\[-8pt]
\\[-8pt]
\nonumber
{}\bigl[ M(\phi),N(\phi)\bigr]_t &=&\rho\bigl\langle
\Lambda_t, \phi^2\bigr\rangle.
\end{eqnarray}
\end{defn}

Observe that if there exists a process $\Lambda$ controlling the
correlation as in Definition~\ref{defn:MP}, then
it is uniquely determined by $(\mu,\nu)$ via the martingales in
\eqref{MP1}.
Obviously, $\Lambda_0=0$ and $\Lambda$ has to be an increasing process
in the sense that $ (\langle\Lambda_t,\phi\rangle
)_{t\ge0}$
is increasing for all $\phi\ge0$. Also, condition \eqref{Cov1} implies
that the martingale measure $M$ (and similarly $N$) is orthogonal in
the sense of~\cite{Wal86}; that is, for all test functions $\phi,\psi$
with $\phi\psi\equiv0$ we have $[M(\phi),M(\psi)]_t=\langle
\Lambda
_t,\phi\psi\rangle=0$.

It is important to note that in the definition of the martingale
problem $(\mathbf{MP})_{\mu_0,\nu_0}^\rho$,
we do not specify the measure-valued process $\Lambda$ more explicitly.
As a consequence, it is not surprising that the martingale problem
$(\mathbf{MP})_{\mu_0,\nu_0}^\rho$
is not well posed without any further conditions.

Indeed, assume that the initial conditions are absolutely continuous
with densities $(u_0,v_0)\in(\calB_\tem^+)^2$, and let $\gamma>0$
be arbitrary.
If we denote by $(u_t^\sse{\gamma},v_t^\sse{\gamma})$ the symbiotic
branching process with finite branching rate $\gamma$,
then $(u_t^\sse{\gamma},v_t^\sse{\gamma})$, considered as
measure-valued processes, is a solution to $(\mathbf{MP})_{\mu_0,\nu
_0}^\rho$ with $\Lambda
=\Lambda
^\sse{\gamma}$ given by
%
\begin{equation}
\label{defn:Lambda} \Lambda_t^\sse{\gamma}(dx):=\gamma\int
_0^t\,ds u_s^\sse{\gamma
}(x)v_s^\sse {\gamma}(x) \,dx,
\end{equation}
for \emph{every} $\gamma>0$; see Theorem~4 in \cite{EF04}.

Certainly uniqueness in the martingale problem $(\mathbf{MP})_{\mu
_0,\nu_0}^\rho$ can thus be
achieved by \emph{prescribing} an explicit correlation structure as
in~\eqref{defn:Lambda}.
However, in order to characterize the limiting object in Theorem~\ref
{thmm:main1}, we
proceed differently, and
we only require (a slightly stronger version of) the ``separation of
types'' property \eqref{singularity1}
for uniqueness.

Our uniqueness argument relies on a variant of the self-duality \`a la
Mytnik~\cite{Mytnik98}.
Also, instead of requiring that the dual process lives in the space of
continuous measure-valued
processes, we can relax this condition, and we will construct the dual
for a large
class of initial conditions by
approximations in the (less restrictive) Meyer--Zheng topology.

Recall the self-duality function
employed in \cite{EF04}:
let $\rho\in(-1,1)$ and if either $(\mu,\nu,\phi,\psi)\in\calM
_\tem
^2\times\calB_\rap^2$ or $(\mu,\nu,\phi,\psi)\in\calM_\rap
^2\times\calB
_\tem^2$, denote
%
\begin{equation}
\label{self-duality product} \langle\!\langle\mu, \nu, \phi,
 \psi\rangle\!\rangle_\rho:= -\sqrt{1-\rho} \langle\mu+ \nu, \phi+ \psi\rangle + i
\sqrt{1+\rho} \langle\mu- \nu, \phi- \psi\rangle.
\end{equation}
Then we define the \emph{self-duality function $F$} as
%
\begin{equation}
\label{self-duality function} F(\mu,\nu,\phi,\psi) := \exp\langle\!\langle\mu, \nu, \phi,
\psi \rangle\!\rangle_\rho.
\end{equation}
With this notation, we define another (somewhat weaker) martingale
problem, which is tailored for an application of the self-duality.

\begin{defn}[{[Martingale problem $(\mathbf{MP'})_{\mu_0,\nu_0}^\rho
$]}]\label{defn:MP'}
Fix $\rho\in(-1,1)$ and (possibly random) initial conditions $(\mu
_0,\nu
_0)\in\calM_\tem^2$ (resp., $\calM_\rap^2$).
A c\`adl\`ag $\calM_\tem^2$-valued (resp., $\calM_\rap^2$-valued)
stochastic process $(\mu_t, \nu_t)_{t\ge0}$
is called a solution to the martingale problem $(\mathbf{MP'})_{\mu
_0,\nu_0}^\rho$
if the following holds: There exists an increasing c\`adl\`ag $\calM
_\tem$-valued (resp., $\calM_\rap$-valued) process $(\Lambda
_t)_{t\ge
0}$ with $\Lambda_0=0$ and
%
\begin{equation}
\label{finiteness Lambda 1} \E_{\mu_0,\nu_0} \bigl[\Lambda_t(dx) \bigr]\in
\calM_\tem\qquad \bigl(\mbox {resp., } \E _{\mu_0,\nu_0} \bigl[
\Lambda_t(dx) \bigr]\in\calM_\rap\bigr)
\end{equation}
for all $t>0$,
such that for all test functions $\phi,\psi\in (\calC_{\rap
}^{(2)} )^+$ [resp., $\phi,\psi\in (\calC^{(2)}_{\tem} )^+$],
the process
%
\begin{eqnarray}
\label{MP9} &&F(\mu_t, \nu_t,\phi,\psi) - F(
\mu_0,\nu_0,\phi,\psi)\nonumber
\\
&&\qquad{}- \frac
{1}{2}\int_0^t F(
\mu_s, \nu_s,\phi,\psi) \langle\!\langle
\mu_s, \nu_s, \Delta\phi, \Delta\psi\rangle\!
\rangle_\rho \,ds
\\
&&\qquad{} - 4\bigl(1-\rho^2\bigr)\int_{[0,t]\times\R} F(
\mu_s,\nu_s,\phi,\psi) \phi(x)\psi (x) \Lambda(ds,dx)\nonumber
\end{eqnarray}
is a martingale.
\end{defn}

In \eqref{MP9} we have interpreted the right-continuous and increasing
process $t\mapsto\Lambda_t(dx)$ as a (locally finite) measure
$\Lambda
(ds,dx)$ on $\R^+\times\R$, via
\[
\Lambda\bigl([0,t]\times B\bigr):=\Lambda_t(B).
\]

\begin{rem}
Note that in contrast to Definition~\ref{defn:MP}, we do not require a
solution of $(\mathbf{MP'})_{\mu_0,\nu_0}^\rho$ to be continuous,\vspace*{4pt} but
only c\`adl\`ag. Hence we can construct solutions to $(\mathbf
{MP'})_{\mu_0,\nu_0}^\rho$ via approximations in the weaker
Meyer--Zheng ``pseudo-path'' topology (see \cite{MZ84} and \cite
{Kurtz91} and cf.~Appendix~\ref{appendix0}), which allows us to work
with second instead of fourth moment bounds and more general initial
conditions.

We do not include the boundary cases $\rho= \pm1$, since either the
real or imaginary
part in~\eqref{self-duality product} vanishes, and we cannot use the
resulting $F$ for our
approach, showing uniqueness via self-duality.

As in Definition~\ref{defn:MP}, the martingale problem of Definition~\ref{defn:MP'} is not well posed: In fact, by Corollary~\ref{cor MPinf}
any solution to $(\mathbf{MP})_{\mu_0,\nu_0}^\rho$ is also a
solution to $(\mathbf{MP'})_{\mu
_0,\nu
_0}^\rho$; this is a simple application of It\^o's formula. In
particular, for any $\gamma> 0$ the solution to the finite rate
symbiotic branching model ${\operatorname{cSBM}(\varrho,\gamma)}_{u_0, v_0}$
also solves the martingale problem $(\mathbf{MP'})_{\mu_0,\nu
_0}^\rho$.
\end{rem}

Somewhat surprisingly, even without prescribing $\Lambda$
we can (at least for $\rho<0$) still prove self-duality and thus uniqueness,
as long we require a certain ``separation of types'' property.
We denote by $(S_t)_{t\ge0}$ the usual heat semigroup.

\begin{thmm}\label{thmm:MPinf}
Fix absolutely continuous initial conditions with densities which are
tempered or rapidly decreasing functions, that is, $(\mu_0,\nu_0)\in
(\calB_\tem^+)^2$,
respectively, $(\mu_0,\nu_0)\in(\calB_\rap^+)^2$.
Assume that $\rho\in(-1,0)$.
\begin{longlist}[(ii)]
\item[(i)]
There exists a unique solution $(\mu_t,\nu_t)_{t\ge0}$ to the
martingale problem $(\mathbf{MP'})_{\mu_0,\nu_0}^\rho$ that is characterized
by the following ``separation of types'' property:
For all $t\in(0,\infty)$, $x\in\R$ and $\eps>0$ we have
%
\begin{equation}
\label{singularity 1} S_{t+\eps} \mu_0(x) S_{t+\eps}
\nu_0(x)\ge\E_{\mu_0,\nu_0} \bigl[S_\eps
\mu_t(x) S_\eps\nu_t(x) \bigr]\mathop{\rightarrow}^{\eps
\downarrow0}0.
\end{equation}
\item[(ii)] Moreover, for each $\gamma>0$ denote by $(\mu_t^\sse
{\gamma
},\nu_t^\sse{\gamma})_{t\ge0}$ the solution to\break $\operatorname{cSBM}(\rho
,\gamma)_{\mu_0,\nu_0}$, considered as measure-valued processes. Then,
as \mbox{$\gamma\uparrow\infty$}, the sequence of processes $(\mu
^\sse
{\gamma}_t,\nu^\sse{\gamma}_t)_{t\ge0}$ converges in law in
$D_{[0,\infty)}(\calM_\tem^2)$, respectively, in $D_{[0,\infty
)}(\calM
_\rap^2)$, equipped with the Meyer--Zheng ``pseudo-path'' topology
to the unique solution of the martingale problem $(\mathbf{MP'})_{\mu
_0,\nu_0}^\rho$
satisfying~\eqref{singularity 1}.
\end{longlist}
\end{thmm}

We call the unique solution to the martingale problem $(\mathbf
{MP'})_{\mu_0,\nu_0}^\rho$ satisfying~\eqref{singularity 1}
the \emph{continuous-space infinite rate symbiotic branching process}.

Note that \emph{if} the measures $\mu_t$ and $\nu_t$ are absolutely
continuous for some $t>0$, then by a simple application of Fatou's
lemma, condition \eqref{singularity 1} implies mutual singularity of
the measures, that is, the separation of types in the more intuitive
sense \eqref{singularity1}; see also the proof of Corollary~\ref
{cor:singularity}.

\begin{rem} \emph{Comparison to the discrete-space infinite rate model}.
The martingale problem $(\mathbf{MP'})_{\mu_0,\nu_0}^\rho$ may be
regarded as a continuous-space analogue of the martingale problem
employed in \cite{KM11b}, Theorem~1.1, to characterize the
discrete-space infinite rate mutually catalytic branching model.
In the discrete case, uniqueness is achieved by prescribing the
condition that $u_t(k) v_t(k) = 0$ for all $k$ in the
state space,
and it suffices to consider
test functions $\phi,\psi$ with disjoint support (i.e., $\phi\psi
\equiv0$).
Consequently,
the last term in~\eqref{MP9} vanishes, and $\Lambda$ does not appear.
Also, it is not possible to copy the self-duality proof from \cite{KM11b},
Proposition~4.7, since unlike in the discrete-space context in
continuous space, we cannot apply the Laplacian directly to the
solutions. We have to ``smooth out'' the solutions (see the proof of
Proposition~\ref{prop:self-dual} below), which, however, destroys the
disjoint support property, giving the additional term in \eqref{MP9}
involving the correlation structure $\Lambda$.

A similarity to the discrete model is that we formulate the convergence
in the weaker Meyer--Zheng
topology. This allows us to work with very general initial conditions
and to relax the condition on the correlation
to $\rho< 0$. Unfortunately, we cannot show convergence for all $\rho
\in(-1,1)$ as for the discrete model,
as we do need bounded second moments.

Finally, in the discrete model the limiting object can be described by
a system of stochastic differential equations
with jumps. We do not yet have such an explicit description of the
limit and we will describe possible approaches
to this problem in Remark~\ref{re:open} below.
\end{rem}

We return to the symbiotic branching model with complementary Heaviside
initial conditions for some fixed branching rate $\gamma>0$, and to the
corresponding diffusively rescaled solutions, considered
as measure-valued processes $(\mu^\ssup{n}_t,\nu^\ssup{n}_t)_{t
\geq
0}$ as in \eqref{defn:mu_nu}.
From the scaling property~\eqref{scaling property} it follows that
$(\mu^\ssup{n}_t, \nu^\ssup{n}_t)_{t \geq0}$ are in law equal
to the nonrescaled system $(\mu^\sse{n\gamma}_t, \nu^\sse{n\gamma
}_t)_{t \geq0}$
with branching rate $\gamma n$. In particular, Theorem~\ref
{thmm:MPinf}(ii) shows that
$(\mu^\ssup{n}_t, \nu^\ssup{n}_t)_{t \geq0}$ converges in law in the
Meyer--Zheng ``pseudo-path'' topology for any $\rho< 0$.

However, in Theorem~\ref{thmm:main1} we have stated convergence in the
stronger Skorokhod topology
on $C_{[0,\infty)}( \calM_\tem^2)$ (albeit for a smaller range of the
parameter $\rho$).
As we have explained in Remark~\ref{re:2011-1},
this indicates that some extra input is needed. We now state the
full version of our main result, which generalizes Theorem~\ref{thmm:main1}
by characterizing the limit for $\varrho>-1$ as the continuous-space
infinite rate symbiotic branching process; recall that for $\varrho
=-1$, the limit is characterized by the results in \cite{T95}.

\begin{thmm}
\label{thmm:main2}
Assume $\varrho\in(-1, -\frac{1}{\sqrt{2}})$. Let $(u_t, v_t)_{t\ge
0}$ be a solution to $\operatorname{cSBM}(\rho,\gamma)_{\mu_0,\nu_0}$
with complementary Heaviside initial conditions $(\mu_0, \nu_0) = (\1
_{\R^-}, \1_{\R^+})$. Then the sequence of
processes $(\mu_t^\ssup{n}, \nu_t^\ssup{n})_{t \geq0}$ converges in\break 
$\calC_ {[0,\infty)}( \calM_\tem^2)$ w.r.t. the Skorokhod topology to
the unique solution $(\mu_t,\nu_t)_{t\ge0}$ of the martingale problem
$(\mathbf{MP'})_{\mu_0,\nu_0}^\rho$ satisfying \eqref{singularity 1}
from Theorem~\ref{thmm:MPinf}.
Moreover, the limit
has the following properties:
\begin{itemize}
\item It is also the unique solution to the martingale problem
$(\mathbf{MP})_{\mu_0,\nu_0}^\rho$ of
Definition~\ref{defn:MP} with the property \eqref{singularity 1}.
\item Absolute continuity: For each fixed $t>0$, $\mu_t$ and $\nu_t$
are absolutely continuous w.r.t. the Lebesgue measure $\ell$,
\[
\mu_t(dx)=\mu_t(x)\,dx, \qquad \nu_t(dx)=
\nu_t(x)\,dx, \qquad \p\mbox{-a.s.}
\]
\item
The ``separation of types'' property holds also in the sense \eqref
{singularity1}, that is,
for each $t>0$, the (absolutely continuous) measures $\mu_t$ and $\nu
_t$ are mutually singular: We have
\[
\mu_t(\cdot) \nu_t(\cdot)=0,\qquad \p\otimes\ell\mbox{-a.s.}
\]
\end{itemize}
\end{thmm}

\begin{rem}
Note that our result state convergence in the Skorokhod topology, which
is stronger
than the Meyer--Zheng topology employed in Theorem~\ref{thmm:MPinf} and
also in the discrete-space model.
For the continuous model, we believe that the stronger result should
also be true
for a larger range of parameters. In contrast, in the discrete model,
the limit is
given by a system of stochastic differential equations with jumps that
is not continuous and
so cannot be the Skorokhod limit of continuous processes.
\end{rem}

\begin{rem}\label{rem:not_a_limit}
With the help of the characterization in Theorem~\ref{thmm:main2}, one
can now show that unlike in the stepping stone case considered in \cite
{T95}, for $\rho> -1$ the limit
cannot be of the form
%
\begin{equation}
\label{eq:1807-1} (\mu_t,\nu_t)_{t \geq0 } = (
\1_{\{x
\leq I_t\}} \,dx, \1_{\{x \geq I_t\}} \,d x)_{t \geq0} ,
\end{equation}
for a semimartingale $(I_t)_{t\ge0}$. Indeed, suppose that $(\mu
_t,\nu
_t)$ is of this form and satisfies
the martingale problem $(\mathbf{MP})_{\mu_0,\nu_0}^\rho$. First of
all, since the limiting
measure-valued processes are continuous,
this forces $(I_t)_{t \geq0}$ to be a continuous semimartingale.
Moreover, the initial conditons
tell us that $I_0 = 0$. Therefore, we can write $I_s = M_s + A_s$
for a continuous local martingale $M_t$ (with $M_0 = 0$) and a
continuous, adapted process $A_t$ that
is of locally finite variation (and $A_0 = 0$). Now, let $\phi\in
\calC
_{\rap}^{(2)}$. Then, by It\^o's formula, we have that
\begin{eqnarray*}
\langle\mu_t, \phi\rangle& =& \int_{-\infty
}^{I_t}
\phi(x) \,dx = \langle\1_{\R^-}, \phi\rangle+ \int_0^t
\phi(I_s) \,d I_s + \frac{1}2 \int
_0^t \phi'(I_s) \,d
[I]_s
\\
& =& \langle\1_{\R^-}, \phi\rangle+ \int_0^t
\phi(I_s) \,d M _s + \int_0^t
\phi(I_s) \,d A_s+ \frac{1}2 \int
_0^t \langle \mu_s, \Delta\phi
\rangle \,d [I]_s.
\end{eqnarray*}
Thus, by the first condition~\eqref{MP1} of $(\mathbf{MP})_{\mu
_0,\nu_0}^\rho$,
we can deduce that
\[
\int_0^t \phi(I_s) \,d
A_s+ \frac{1}2 \int_0^t
\langle\mu_s, \Delta \phi \rangle \,d [I]_s -
\frac{1}2 \int_0^t \langle
\mu_s, \Delta\phi\rangle \,ds
\]
is a local martingale. Since it is continuous and of locally finite
variation, the expression has to be constant equal to $0$.
Moreover, since $\phi$ was arbitrary, this allows us to conclude that
$A_t$ is identically $0$ and $[I]_t = t$.
Hence, $I_t$
is a Brownian motion by L\'evy's characterization and thus
\[
\bigl[ \langle\mu_{\cdot}, \phi\rangle, \langle\mu_{\cdot}, \phi
\rangle \bigr]_t = \int_0^t
\phi(I_s)^2 \,ds.
\]
Finally, we note that
\[
\langle\nu_t , \phi\rangle= \int_{I_t}^\infty
\phi(x) \,dx = \int \phi (x) \,dx - \langle\mu_t, \phi\rangle.
\]
In particular, we find that
\[
\bigl[ \langle\mu_{\cdot}, \phi\rangle, \langle\nu_{\cdot}, \phi
\rangle \bigr]_t = - \bigl[ \langle\mu_{\cdot}, \phi\rangle,
\langle\mu_{\cdot}, \phi \rangle\bigr]_t.
\]
This contradicts the second condition~\eqref{Cov1} of $(\mathbf
{MP})_{\mu_0,\nu_0}^\rho$ (since we
assume $\rho\neq-1$), so that the limit cannot be
of the form given in~\eqref{eq:1807-1}.
\end{rem}

In the case $\rho=-1$, the authors in~\cite{MT97} exploit the
corresponding fourth moment bound
to get an estimate on the moments of the width of the interface
$|R(u_t,v_t) - L(u_t,v_t)|$, \emph{without} any rescaling [here we use
the notation \eqref{defn:R_L}].
However, this estimate heavily relies on the fact that there are
``no holes'' in the system where both $u$ and $v$ are zero.
In our case, we can imitate the reasoning to get an estimate
for the approximate interface defined in the following way.
For any $\eps> 0$, define an approximate left endpoint of the
interface as
\[
L_t(\eps) = \inf \biggl\{ x\in\R\dvtx\int_{-\infty}^x
u_t(y) v_t(y) \,dy\geq \eps \biggr\} \wedge
R(u_t,v_t)
\]
and similarly for the right endpoint
\[
R_t(\eps) = \sup \biggl\{ x\in\R\dvtx\int_x^\infty
u_t(y) v_t(y) \,dy\geq \eps \biggr\} \vee
L(u_t,v_t).
\]
Since $|R(u_t,v_t)|,|L(u_t,v_t)|$ are almost surely finite, $R_t(\eps),
L_t(\eps)$
are well defined. Our final result states that this width of the approximate
interface remains small uniformly in $t$
in the following way.

\begin{thmm}\label{thmm:width} Suppose $(u_0, v_0) = (\1_{\R^-}, \1
_{\R
^+})$, $(u_t,v_t)$ is a solution of
$\operatorname{cSBM}(\rho,\gamma)_{u_0,v_0}$ and $\eps> 0$. Then,
for any $\varrho< \varrho(4)= -\frac{1}{\sqrt{2}}$, $p \in(0,1)$ and
any $\delta\in(0, 2(1-p))$, there exists a constant
$C = C(\rho, \delta, p)$ such that for all $t > 0$,
\[
\E_{\1_{\R^-},\1_{\R^+}} \bigl(\bigl(R_t(\eps) - L_t(\eps)
\bigr)^+\bigr)^p \leq C \eps^{-2
+ \delta} \gamma^{-(2 + p -\delta)}.
\]
\end{thmm}
%

\begin{rem}\label{re:open} \emph{Open problems.}
Ideally, one would like to characterize the limiting process $(\mu,\nu
)$ in Theorems~\ref{thmm:MPinf} and~\ref{thmm:main2} in an explicit way.
A first approach toward a better understanding of the limit would be
the identification of the quadratic (co-)variation of the limit
martingales. Indeed, using the same method as in \cite
{Dawsonetal2002}, Lemma~41, it should be
in principle possible to ``compute'' the limit of the processes
$\Lambda
^\sse{\gamma}$ from \eqref{defn:Lambda} as $\gamma\uparrow\infty
$, for
general initial conditions and all $\rho<0$.
The resulting expression will (as the ``collision local time'' in \cite
{Dawsonetal2002}) involve a spatial smoothing of the limit densities
and can then be used to specify the process $\Lambda$ in the martingale
problems $(\mathbf{MP})_{\mu_0,\nu_0}^\rho$ and $(\mathbf
{MP'})_{\mu_0,\nu_0}^\rho$.
Remarkably, it seems that proving the self-duality (Proposition~\ref
{prop:self-dual} below), and hence uniqueness, using this specification
of $\Lambda$ turns out to be technically substantially more involved
than using the ``separation of types'' approach.
In fact, the strength of our approach is that we can show uniqueness
while leaving the process $\Lambda$ largely unspecified.

Nevertheless, it is promising
to again specialize to the case of complementary Heaviside initial conditions.
In this case, we first note that
the constant on the right-hand side of Theorem~\ref{thmm:width} tends to
$0$ as $\gamma\uparrow\infty$.
This strongly suggests that the interface of the diffusively rescaled
processes shrinks to a single point in the limit.
That is, we expect the limit densities to be of the form
\[
\mu_t(x) = \mu_t(x) \1_{\{ x < I_t \}},\qquad
\nu_t(x) = \nu_t(x) \1_{\{
x >
I_t\}},\qquad x\in\R,
\]
with $I_t:=\sup\{x\in\R\dvtx\mu_t(x)>0\}=\inf\{x\in\R\dvtx\nu
_t(x)>0\}$
denoting the position of the interface.
In fact, if we assume this and moreover that the densities are
sufficiently regular at the point of the interface,
then the expression for $\Lambda$ given in terms of the above spatial smoothing
(analogous to~\cite{Dawsonetal2002}) simplifies considerably.
Indeed, preliminary calculations suggest that under these assumptions,
\[
\Lambda_t(dx)=\frac{1}{|\rho|}\int_0^t\,ds
\mu_s(I_s-)\nu _s(I_s+)\delta
_{I_s}(dx).
\]
Note that this is in line with the stepping stone case $\rho=-1$
considered in \cite{T95}, and that the expression blows up as $\rho
\uparrow0$.
However, especially the assumption that the densities are sufficiently
regular at the interface
is rather strong. In particular, it seems likely that for a proof
one would have to go beyond the measure-valued approach of the present
paper. At the moment, we do not even know whether the densities are
locally bounded or not; recall, for example, that the densities of the
\emph{two-dimensional} finite rate continuous mutually catalytic
branching model considered in \cite{Dawsonetal2002} are locally unbounded.

In order to prove that the interface shrinks to one point, a possible
line of attack would be to establish stationarity of the
interface without any rescaling as in~\cite{MT97}.
Another approach, which might also shed some light on the question of
an explicit equation for the limit, could be to diffusively rescale the
discrete-space infinite rate model and to investigate whether it
converges to our limit process.
This is also supported by the conjecture of Klenke and Oeler~\cite
{KO10}, that for the discrete-space infinite rate
model, the interface is essentially a single point.
However, carrying out this rather ambitious program is clearly beyond
the scope of the present paper, and
will be taken up in future research.
\end{rem}

\subsection{Strategy of proof and organization of the paper}\label
{ssn:strategy}

The proof of our main result, Theorem~\ref{thmm:main2}, splits into
two parts:
The first step is to show
tightness, while the second step is to find a property that uniquely
identifies the limit
points. In our case, we can show that any limit point satisfies the
martingale problem
$(\mathbf{MP'})_{\mu_0,\nu_0}^\rho$ and also the additional
``separation of types'' property~\eqref{singularity 1}
which by Theorem~\ref{thmm:MPinf} gives uniqueness.
More concretely
the \emph{proof of Theorem~\ref{thmm:main2}} is obtained by combining
the following results:
\begin{itemize}
\item Tightness in $\calC_{[0,\infty)} ( \calM_\tem^2)$ is proved in
Proposition~\ref{tightness of processes}.
\item In Proposition~\ref{MPinf 1}, we show that any limit point
satisfies $(\mathbf{MP})_{\1_{\R^-},\1_{\R^+}}^\rho$
and therefore by Corollary~\ref{cor MPinf} also $(\mathbf{MP'})_{\1
_{\R
^-},\1_{\R^+}}^\rho$.
To guarantee uniqueness, we also check in Lemma~\ref{lemma:sep} that
the ``separation of types'' condition~\eqref{singularity 1} is satisfied.
\item Finally, we note that the absolute continuity of the limit is proved
in Proposition~\ref{prop:AC}, from which together with Lemma~\ref
{lemma:sep} we obtain also the separation of types in the form of
\eqref
{singularity1}; see Corollary~\ref{cor:singularity}.
\end{itemize}

The \emph{proof of Theorem~\ref{thmm:MPinf}} relies on a strong
interplay between
parts (i) and (ii). More precisely, we will proceed as follows:
\begin{itemize}
\item We show tightness (in the Meyer--Zheng sense) for $(\mu^\sse
{\gamma}, \nu^\sse{\gamma})$
as $\gamma\ra\infty$ starting with general initial conditions
in $\calB_\tem^+$ or $\calB^+_\rap$ for any $\rho< 0$; see
Proposition~\ref{tightness_MZ dual}.
\item Next, we show in Proposition~\ref{MPinf 2} and Lemma~\ref{lemma:sep}
that any limit point satisfies the martingale problem
$(\mathbf{MP'})_{\mu_0,\nu_0}^\rho$ and also property~\eqref
{singularity 1}.
\item These first two steps cover
the existence statement
of part (i). Moreover, they are also essential for the uniqueness as
stated in Proposition~\ref{prop:unique}.
Indeed, the uniqueness proof relies on a self-duality argument,
where we need the existence of the dual process, which in our case is
the infinite rate symbiotic branching
model with rapidly decreasing initial conditions.
\item Part (ii) of Theorem~\ref{thmm:MPinf} is now a corollary of what
we have already shown. Indeed, we have covered the tightness
and proved that any limit point satisfies $(\mathbf{MP'})_{\mu_0,\nu
_0}^\rho$
including~\eqref{singularity 1}
so that uniqueness follows immediately.
\end{itemize}

The structure of the remaining paper is as follows: In Section~\ref
{sn:tightness}, we show tightness for complementary
Heaviside initial conditions and $\rho<-\frac{1}{\sqrt{2}}$ on
$\calC
_{[0,\infty)}$ in the Skorokhod sense, and for general initial
conditions and all $\rho<0$ on $D_{[0,\infty)}$ in the Meyer--Zheng
sense. Next, we consider in Section~\ref{sn:properties} the properties
of limit points
in both topologies.
Furthermore, we prove uniqueness of the martingale problems in
Section~\ref{ssn:martconv}.
In Section~\ref{ssn:moments}, we provide a missing ingredient for the
proof of tightness in the strong sense, namely an estimate on integrated
fourth mixed moments. Finally, in Section~\ref{ssn:width} we prove
Theorem~\ref{thmm:width}
as a corollary to the fourth moment bound.

Many of the basic techniques, such as using duality to show uniqueness
and deducing tightness from moments estimates,
are standard in the literature for measure-valued processes.
Also, the Meyer--Zheng topology has
been used for the \emph{discrete} infinite rate symbiotic branching
model, because in this topology, tightness
relies only on relatively weak moment bounds.
However, we would like to highlight two novelties in our approach:
In our Theorem~\ref{thmm:main1} we claim convergence in the Skorokhod
topology, which is stronger
than convergence in the Meyer--Zheng sense. For our result,
we use the Meyer--Zheng topology only to construct the dual process
that then yields uniqueness; cf. also Theorem~\ref{thmm:MPinf}.
This approach allows us to construct the dual process for a large class
of initial conditions, which is essential, since only then
duality can be used to identify the law of the original process.

The second novelty is to show uniqueness without specifying the
correlation $(\Lambda_t)_{t \geq0}$ in the
martingale problem. This should be compared to a similar situation
in~\cite{Dawsonetal2003}, where the authors
show uniqueness for the two-dimensional equivalent of the mutually
catalytic branching model, which satisfies a similar ``separation of types''
property. In their case, they identify the correlation as an
intersection local time and only then deduce uniqueness.

One further important contribution is the integrated fourth moment
bound of Proposition~\ref{mixed_moments} below which is essential for
tightness in the $\calC_{[0,\infty)}$-sense.
Its derivation relies on careful estimates of intersection local times
together with (uniformly) bounded fourth moments,
which explains the restriction on~$\rho$.

\textit{Notation}: We have collected some of the standard facts and notation
about measure-valued processes in Appendix~\ref{appendix0}. In
Appendix~\ref{appendix1} we recall the martingale problem formulation
of the finite rate symbiotic branching model ${\operatorname{cSBM}(\varrho
,\gamma)}_{u_0, v_0}$ and deduce some consequences of the martingale
problem $(\mathbf{MP})_{\mu_0,\nu_0}^\rho$ of Definition~\ref
{defn:MP}. Finally,
Appendix~\ref{appendix2} is a collection of estimates for Brownian motion
and its local time.
Throughout this paper, we will denote by $c,C$ generic constants whose
value may change from line to line. If the dependence on parameters
is essential, we will indicate this correspondingly.

\section{A bound on integrated fourth mixed moments}
\label{ssn:moments}

The first step is a bound on integrated fourth mixed moments that will
allow us to prove tightness of the sequence \eqref{defn:mu_nu} of
rescaled processes along the lines of~\cite{T95}; see the next section.
For this estimate, we heavily use that the symbiotic branching model is
dual to a system of colored particles via a moment duality due to \cite
{EF04} that we explain now.

We aim to describe the asymptotic behavior of mixed moments of the form
\[
\E_{u_0,v_0} \bigl[u_t(x_1)\cdots
u_t(x_{n})v_t(x_{n+1})\cdots
v_t(x_{n+m}) \bigr].
\]
For $\varrho\in[-1,1]$, the dual works as follows: Consider $n+m$
particles in $\R$ which can take on two colors, say $\red$ and $\blue$.
Each particle moves like a
Brownian motion independently of all other particles. At time $0$, we
place $n$ $\red$ particles at positions $x_1,\ldots, x_n$, respectively,
and $m$ $\blue$ particles at positions $x_{n+1},\ldots,x_{n+m}$.
As soon as two particles meet, they start collecting collision local time.
If both particles are of the same color, one of them changes color when
their collision local time
exceeds an (independent) exponential time with parameter $\gamma$.
Denote by $L_t^=$ the total collision local time collected by all pairs
of the same color up to time $t$, and let $L_t^{\neq}$ be the collected
local time of all pairs of different color up to time $t$.
Finally, let $l_t:=(l_t^\red, l_t^\blue), t\ge0$, be the corresponding
particle process, that is, $l^\red_t(x)$ denotes the number of $\red$
particles
at $x $ at time $t$, and $l^\blue_t(x)$ is defined accordingly for
$\blue$ particles.
Our mixed moment duality function will then be given, up to an
exponential correction involving both $L_t^=$ and $L_t^{\neq}$, by
a moment duality function
\[
(u,v)^{l_t} := \mathop{\prod_{x\in\R\dvtx}}_{ l^\red_t(x)\ \mathrm{or}\ l^\blue
_t(x)\ne0}
u(x)^{l_t^\red(x)}v(x)^{l_t^\blue(x)}.
\]
Note that since there are only $n+m$ particles, the potentially
uncountably infinite product is actually a finite product and hence
well defined. The following lemma is taken from \cite{EF04},
Proposition~12.

\begin{lemma}\label{la:mdual}
Let $(u_t,v_t)_{t\ge0}$ be a solution to $\operatorname{cSBM}(\varrho
,\gamma
)_{u_0,v_0}$ with $\varrho\in[-1,1]$. Then, for any $x \in\R$ and
$t\geq0$,
%
\begin{equation}\qquad
\label{eq:dual} \E_{u_0,v_0} \bigl[u_t(x_1)\cdots
u_t(x_{n})v_t(x_{n+1})\cdots
v_t(x_{n+m}) \bigr]=\E \bigl[(u_0,v_0)^{l_t}e^{\gamma(L_t^=+\varrho
L_t^{\neq})}
\bigr],
\end{equation}
where the dual process $(l_t)_{t\ge0}$ behaves as explained above,
starting in $l_0=(l_0^\red, l_0^\blue)$ with $\red$ particles
located in $(x_1, \ldots, x_n)$ and $\blue$ particles in $(x_{n+1},
\ldots, x_{n+m})$, respectively.
\end{lemma}

Note that if $u_0=v_0\equiv1$, the first factor in the expectation of
the right-hand side equals $1$. Also note that for second mixed
moments, the duality simplifies considerably: In this case, the dual
process is started from two particles of different color, which by the
definition of the process will retain their respective color for all
time (color changes can only occur if two particles of the same color
meet). Introducing two independent Brownian motions $(B^i_t)_{t\ge0}$,
$i=1,2$, with intersection local time $(L^{1,2}_t)_{t\ge0}$, equation
\eqref{eq:dual} can thus be written as
%
\begin{equation}
\label{eq:dual second moments} \E_{u_0,v_0} \bigl[u_t(x)v_t(y)
\bigr]=\E_{x,y} \bigl[u_0\bigl(B_t^1
\bigr)v_0\bigl(B_t^2\bigr)e^{\gamma\rho L_t^{1,2}}
\bigr],
\end{equation}
where here and in the following we will label the Brownian motions
according to their starting positions from left to right.

We now state the fourth (mixed) moment estimate announced above:

\begin{prop}[(Mixed moments)]
\label{mixed_moments}
Let $(u_t, v_t)_{t\ge0}$ be a solution to ${\operatorname{cSBM}(\varrho
,\gamma
)}_{u_0, v_0}$
with initial values $(u_0, v_0) = ( \1_{\R^-},\1_{\R^+})$.
Then,
for $\varrho<\varrho(4)= - \frac{1}{\sqrt{2}}$,
\[
\E_{u_0, v_0} \biggl[ \int\int u_t(x)u_t(y)v_t(x)v_t(y)
\,\dd x \,\dd y \biggr] \leq C(u_0,v_0; \gamma,\rho)
\]
uniformly for all $t\ge0$.
\end{prop}

Note that by Fubini's theorem and a simple substitution, it is
sufficient to prove that for $z > 0$,
\[
\E_{u_0, v_0} \biggl[ \int u_t(x) u_t(x-z)
v_t(x)v_t(x-z) \,\dd x \biggr]
\]
is integrable in $z$. Our Ansatz is to use
the moment duality from Lemma~\ref{la:mdual} and combine it with the
moment bounds of
Theorem~\ref{thmm:mc}. However, Theorem~\ref{thmm:mc} requires constant
initial conditions, which simplifies the moment duality considerably.

In our case,
the duality in~(\ref{eq:dual}) reads
\begin{eqnarray*}
&&\E_{1_{\R^-},1_{\R^+}} \bigl[u_t(x)  u_t(x-z)
v_t(x) v_t(x-z) \bigr]
\\
&&\qquad=\E_{l_0^\red=(x, x-z), l_0^\blue=(x, x-z)} \bigl[(u_0,v_0)^{l_t}e^{\gamma(L_t^=+\varrho L_t^{\neq})}
\bigr].
\end{eqnarray*}
To describe the dynamics of $(l_t)_{t\ge0}$, we introduce a system of
four independent Brownian motions $\{B^{i}_t$, $i=1,\ldots,4\}$
with respective colors $c_i(t)\in\{ \red, \blue\}$ at time $t$.
We label the Brownian motions according to their starting positions
$B_0^1=0, B_0^2=0, B_0^3=z, B_0^4=z$
in increasing order, and
we set their initial colors to be $c_1(0) = c_3(0) = \red$, while
$c_2(0) = c_4(0) = \blue$. Defining
\[
f^\red:= u_0=\1_{\R^-}, \qquad f^\blue:=
v_0 =\1_{\R^+},
\]
we can rewrite the duality as
\begin{eqnarray*}
&&\E_{u_0, v_0} \bigl[ u_t(x)
v_t(x) u_t(x-z) v_t(x-z) \bigr]
\\
&&\qquad = \E_{l_0^\red=(0,z), l_0^\blue=(0,z)} \Biggl[ \prod_{i=1}^4
f^{c_i(t)} \bigl(x - B_t^i\bigr) e^{\gamma(L_t^= + \rho L_t^{\neq}) }
\Biggr].
\end{eqnarray*}
We now integrate over $x$ and estimate the integral. Note that the
exponential term does not depend on $x$. Hence, we may restrict our
attention to
%
\begin{equation}
\label{eq:prod} \int\prod_{i=1}^4
f^{c_i(t)} \bigl(x - B_t^i\bigr) \,\dd x ,
\end{equation}
for different color configurations.
First observe that
%
\begin{equation}
\label{eq:int} f^\red(x-B_t) = \1_{\{ x < B_t\}} \quad\mbox{and} \quad f^\blue(x-B_t) = \1 _{\{
x > B_t \}},
\end{equation}
so that one should think of the integral in~(\ref{eq:prod}) as
an integral over a product of Heaviside functions centered at $B_t^i$, where
the color determines the shape.

Now denote by $r(t)$ the index of the left-most $\red$ Brownian motion
at time $t$, that is, $c^{r(t)}(t) = \red$ and
\[
B^{r(t)}_t \leq B_t^i\qquad \mbox{for all
} i \mbox{ such that } c^i(t) = \red,
\]
where we choose the smaller index to resolve ties.
Similarly, we denote by $\ell(t)$ the index of the right-most $\blue$ Brownian
motion, that is, $c^{\ell(t)}(t) = \blue$ and
\[
B^{\ell(t)}_t \geq B_t^i \qquad\mbox{for all
} i \mbox{ such that } c^i(t) = \blue
\]
(with the smaller index to resolve ties).

Observe that, due to the definition of our dual particle system
$(l_t)_{t\ge0}$, if we start with four particles and two colors,
there will always be at least one $\red$ particle and at least one
$\blue$ particle around at any time, no matter
what the actual color changes were (color changes can only occur if two
particles of the same color meet).
Moreover, with the above notation, the integral in~(\ref{eq:prod}) is
$0$ unless $B^{r(t)}_t > B^{\ell(t)}_t$
(see Figure~\ref{fig:indicators-1}),
and since the product is either $0$ or $1$, we obtain
\[
\int\prod_{i=1}^4 f^{c_i(t)}
\bigl(x - B_t^i\bigr) \,\dd x = \bigl(B^{r(t)}_t
- B^{\ell(t)}_t\bigr)^+;
\]
see also Figure~\ref{fig:indicators-2}.

\begin{figure}

\includegraphics{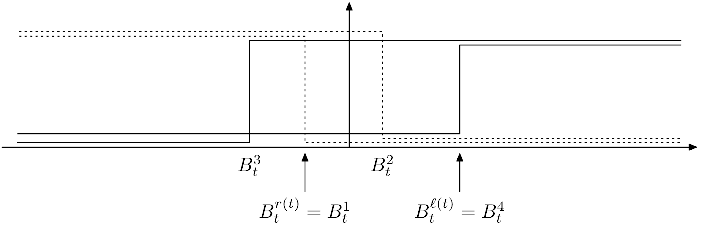}

\caption{An illustration of the four factors in the product
in~\protect\eqref{eq:prod}
(drawn slightly shifted for illustration). The Heaviside functions
are centred at the positions of the Brownian motions, and the color
determines the shape
($\red$ is dotted, and $\blue$ is drawn in black). Here, $c_1(t) =
c_2(t) = \red$ and
$c_3(t) = c_4(t) = \blue$. In this case, the product of all four
factors is zero, since $B_t^{r(t)} < B_t^{\ell(t)}$.}
\label{fig:indicators-1}
\end{figure}

\begin{figure}[b]

\includegraphics{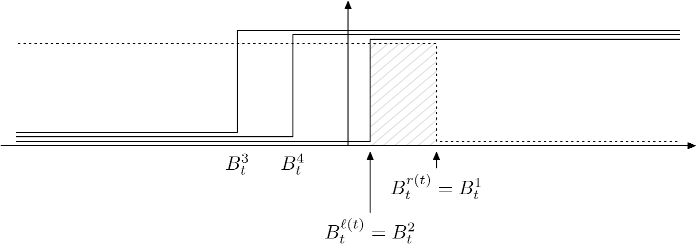}

\caption{In this scenario, $c_1(t) = \red$, $c_2(t) = c_3(t) = c_4(t)=
\blue$. Since $B_t^{\ell(t)} < B_t^{r(t)}$, the
integral gives a nonzero contribution corresponding to the shaded
area.}
\label{fig:indicators-2}
\end{figure}

Altogether, we arrive at
%
\begin{eqnarray}
\label{eq:reforml} %
&&\E_{u_0, v_0} \int
u_t(x)  u_t(x-z) v_t(x)v_t(x-z)
\,\dd x
\nonumber
\\[-8pt]
\\[-8pt]
\nonumber
&&\qquad = \E _{(0,z), (0,z)} \bigl[ \bigl(B^{r(t)} - B^{\ell(t)}\bigr)^+
e^{\gamma(L_t^= +
\rho
L_t^{\neq})} \bigr]
\end{eqnarray}
and need to show that, for $z>0$, this expression is integrable in $z$.
We prepare this with a lemma
which covers the important case where the two particles that are
initially in the middle start
in the same location.

\begin{lemma}\label{rough_moment_bound} Assume that $\varrho< \varrho
(4)$, and let $-\infty< x<y<z <\infty$ and $\delta\in(0,\frac
{1}{2})$. Then, for any initial configuration $l_0=\underline{x}$ that
contains four particles in positions $x,y,y,z$ and two of each color,
that is,
\[
\underline{x} \in \bigl\{%
(x,y), (y,z); (y,z), (x,y); (x,z), (y,y);
(y,y), (x,z) \bigr\},
\]
we have
\begin{eqnarray*}
&&\E_{\underline{x}} \bigl[ \bigl(B_t^{r(t)} -
B_t^{\ell(t)}\bigr)^+ e^{\gamma
(L_t^= + \rho L_t^{\neq}) } \bigr]
\\
& &\qquad\leq C(\varrho, \gamma, \delta) \min \biggl\{ \frac{(z-y+1)(y-x+1)}{t^{{1}/{2} - \delta}}, 1 \vee
t^{\delta} \biggr\}.
\end{eqnarray*}
\end{lemma}

\begin{pf}
Pick $\varrho'$ so that $\varrho< \varrho'< \varrho(4)$,
and let $\delta\in(0,\frac{1}{2})$.
Using the (generalized) H\"older inequality twice for $p_1,p_2, p_3
\geq1$ with $p_3 = (1-\frac{\delta}{2})^{-1}$ and $p_1 = p_2$ such that
$\frac{1}{p_1}+\frac{1}{p_2} + \frac{1}{p_3}= 1$, we obtain
%
\begin{eqnarray}
\label{eq:0202-1} &&\E_{\underline{x}} \bigl[ \bigl(B^{r(t)} -
B^{\ell(t)}\bigr)^+ e^{\gamma
(L_t^= + \rho L_t^{\neq})} \bigr]\nonumber
\\
&&\qquad \leq\E_{\underline{x}} \bigl[ \bigl(\bigl(B^{r(t)} - B^{\ell(t)}
\bigr)^+ \bigr)^{p_1} \bigr]^{{1}/{p_1}} \E_{\underline{x}} \bigl[
e^{p_2\gamma(L_t^= + \rho' L_t^{\neq
})} \bigr]^{{1}/{p_2}}\\
&&\qquad\quad{}\times \E_{\underline{x}} \bigl[ e^{-p_3\gamma(\rho
'-\rho)
L_t^{\neq}}
\bigr]^{{1}/{p_3}}.\nonumber
\end{eqnarray}
By the moment duality \eqref{eq:dual}, the second expectation in (\ref
{eq:0202-1})
corresponds to the fourth mixed moment of
a system with branching rate $p_2\gamma$, correlation parameter $\rho'$
and constant initial conditions.
Since $\rho' < \rho(4)$, this expression is bounded by a constant
(depending only on $\rho'$) uniformly in $t\ge0$;
see Theorem~\ref{thmm:mc} and also Remark~\ref{rem:starting_point}.

For the first expectation on the right-hand side in~(\ref{eq:0202-1}),
we claim that
%
\begin{equation}
\label{eq:0203-1} \E_{\underline{x}} \bigl[ \bigl(\bigl(B^{r(t)}_t
- B^{\ell(t)}_t\bigr)^+ \bigr)^{p_1}
\bigr]^{{1}/{p_1}} \leq C(p_1)t^{{1}/{2}}.
\end{equation}
The claim follows if we can show that the expectation on
the left-hand side does not depend on the distances of the starting
points $z-y, y-x$. We recall that the particles are labeled from
left to right according to the initial positions. In particular $2,3$
are the labels of the particles
started in $y$. Also, we can always assume that $B_t^{\ell(t)} < B_t^{r(t)}$
since this is the only scenario when we observe a positive contribution
to the expectation.

Denote by $\tau_{i,j}$ the first collision time of particles $i,j$.
We claim that if $B_t^{\ell(t)} < B_t^{r(t)}$, then there exist
$i, j \in\{1, \ldots, 4\}, i \neq j$, such that $c_i (t) \neq c_j(t)$
and $\tau_{ij} \leq t$.
Indeed, suppose first that no color change occurs up to time $t$. Then, if
particles~$2$ and~$3$ (both started in $y$) have different colors,
$\tau
_{2,3} = 0$,
and the claim holds. Conversely, if $2$ and $3$ have the same color,
there has to be a
collision between particles of different colors before time $t$, since
the condition $B_t^{\ell(t)} < B_t^{r(t)}$
implies that both $\blue$ particles are to the left of the $\red$
particles at time $t$; see Figure~\ref{fig:scenario-no-change}
for an illustration.

Moreover, if there is a color change before time $t$, we can consider
particle $i$ that has changed its color last before time $t$ (out of all
particles), say at time $\sigma^i$.
Then by construction of the particle process, the color change happened
through the interaction
with particle $j$, which just before the change had the same color, but
now satisfies $c_i(\sigma_i) \neq c_j(\sigma_i)$
and also $\tau_{ij} \leq\sigma_i\leq t$. However, since $i$ was the
last particle to change color, it follows that
$c_j(t) = c_j(\sigma_i) \neq c_i(\sigma_i) = c_i(t)$; see also
Figure~\ref{fig:scenario-change}.
%
\begin{figure}

\includegraphics{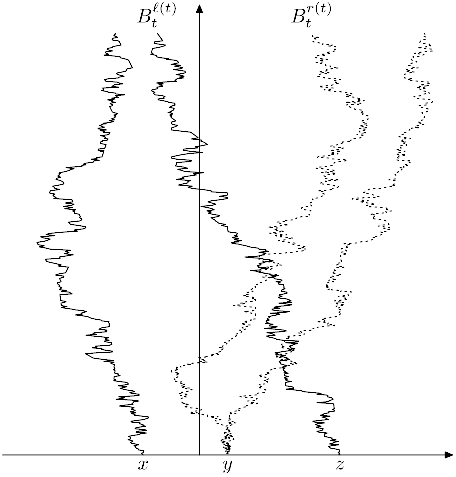}

\caption{No color change occurs up to time $t$ and two $\red$ particles
(dotted line)
start in $y$, while two $\blue$ particles (black line) start in $x$ and
$z$, respectively.
Moreover at time $t$, $B_t^{r(t)} > B_t^{\ell(t)}$ so that particles
of distinct
colors must have crossed.}\label{fig:scenario-no-change}
\end{figure}

\begin{figure}

\includegraphics{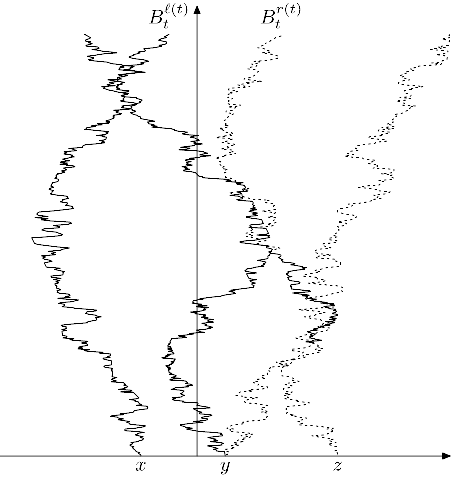}

\caption{If a color change occurs, at least two particles of distinct
colors at time $t$
must have met before.}\label{fig:scenario-change}
\end{figure}

Consequently, in order to show~\eqref{eq:0203-1} we can assume that $i,
j$ are such that $\tau_{ij} \leq t$ and
$c_i(t) \neq c_i(t)$.
Again note that if $B_t^{r(t)} - B_t^{\ell(t)} > 0$, all $\blue$
particles are to
the left of $\red$ particles, so that since particles $i$ and $j$ have
different colors, we find that
$ (B_t^{r(t)} - B_t^{\ell(t)})^+ \leq|B_t^i - B_t^j|$.
Therefore,
by the strong Markov property,
%
\begin{eqnarray}
\label{eq:0203-3} &&\E_{\underline x} \bigl[ \1_{\{\tau_{i,j}\leq t, c_i(t) \neq c_j(t)\}
} \bigl(
\bigl(B_t^{r(t)} - B_t^{\ell(t)}\bigr)^+
\bigr)^{p_1} \bigr]^{{1}/{p_1}}\nonumber
\\
&&\qquad \leq\E_{\underline x} \bigl[ \1_{\{\tau_{i,j}\leq t\}} \bigl|B_t^i
- B_t^j\bigr|^{p_1} \bigr]^{{1}/{p_1}}
\nonumber
\\[-8pt]
\\[-8pt]
\nonumber
&&\qquad \leq\E_{\underline x} \Bigl[ \1_{\{\tau_{i,j} \leq t\}} \E \Bigl[ \sup
_{ 0 \leq s \leq t- \tau_{i,j}} \bigl|B_{\tau_{i,j} + s}^i - B_{\tau_{i,j} + s}^j
 \bigr|^{p_1}| \calF_{\tau_{i,j}} \Bigr] \Bigr]^{
{1}/{p_1}}
\\
&&\qquad \leq\E_{0,0} \Bigl[ \sup_{ 0 \leq s \leq t}
\bigl|B_s^1 - B_s^2 \bigr|^{p_1}
\Bigr]^{{1}/{p_1}} \leq C(p_1) t^{{1}/ {2}}.\nonumber
\end{eqnarray}
By summing over all distinct pairs $i,j$, we thus obtain~\eqref{eq:0203-1}
(where we again make use of the convention that the value of
unspecified constants may change from line to line).

Thus we can conclude from~(\ref{eq:0202-1}) that
\[
\E_{\underline{x}} \bigl[ \bigl(B^{r(t)} - B^{\ell(t)}\bigr)^+
e^{\gamma
(L_t^= +
\rho L_t^{\neq})} \bigr] \leq C(p_1,p_2,\gamma, \rho)
t^{{1}/{2}} \E_{\underline
{x}} \bigl[ e^{-\gamma p_3(\rho'-\rho) L_t^{\neq}} \bigr]^{{1}/{p_3}}.
\]
Recalling that $\frac{1}{p_3} = 1-\frac{\delta}{2}$, we see that
in order to complete the proof
it suffices to show that for any $s > 0$, there
is a constant
$C=C(s)$ such that for all $t \geq0$,
%
\begin{eqnarray}\quad
\label{unequal_local_times}&& \E_{\underline{x}}\bigl[ e^{-s L_t^{\neq}}\bigr]
\nonumber
\\[-8pt]
\\[-8pt]
\nonumber
&&\qquad \leq C \min \biggl
\{ \frac{(z-y+ \log(t \vee e))(y-x + \log(t\vee e))}{t} , \bigl(\log(t\vee e)\bigr)t^{-{1}/{2}} \biggr\},
\end{eqnarray}
where we note that the term $\log(t \vee e)$ can be bounded by
$t^{\delta'} \vee1$ for any $\delta'>0$. Also note that \eqref
{unequal_local_times} holds
trivially
for $t\le1$. Thus we will assume $t\ge1$ throughout the rest of the proof.

First, recall that for the collision local time $L_t^{1,2}$ up to time
$t$ of two independent Brownian motions, started in positions $x\leq
y$, we
have the classical bound that for all $t \geq1$,
%
\begin{equation}
\label{eq:classic} \p_{x,y}\bigl\{ L_t^{1,2} \le\alpha
\log t \bigr\} \le\frac{1}{\sqrt{\pi}} (2\alpha\log t + y-x) t^{-{1}/2},\qquad
\alpha>0;
\end{equation}
see, for example, Corollary~\ref{le:asymp_collision}.
Now fix $s >0$, and let $c = \frac{2}{s}$. We distinguish the three cases:
\begin{longlist}[(iii)]
\item[(i)] $L_t^{\neq} \geq c \log t$,
\item[(ii)] $L_t^{\neq} < c \log t$, but $L_t^\mathrm{tot} := L_t^= +
L_t^{\neq} \geq2c\log t$,
\item[(iii)] $L_t^{\neq} < c \log t$ and $L^\mathrm{tot}_t < 2c\log t$.
\end{longlist}
Regarding (i), we can estimate
\[
\E_{\underline{x}} \bigl[ e^{-s L_t^{\neq}}\1_{\{L_t^{\neq} \geq c
\log
t\}} \bigr] \leq
t^{-sc} = t^{-2},
\]
by our choice of $c = \frac{2} s$.

For (ii), we have in particular that $L_t^= \geq c \log t$.
Now, from the fourth moment bounds (Theorem~\ref{thmm:mc} and
Remark~\ref{rem:starting_point} for the system
with branching rate $\frac{s}{ |\varrho|}$) together with the moment
duality \eqref{eq:dual} for constant initial conditions, we can deduce
that
\begin{eqnarray*}
&&\E_{\underline{x}} \bigl[ e^{-s L_t^{\neq}}  \1_{\{ L_t^{\neq} < c
\log
t, L^\mathrm{tot} \geq2c\log t\}} \bigr]
\\
&&\qquad\le t^{-{cs}/{|\varrho|}}\E_{\underline{x}} \bigl[ e^{
{s}/{|\varrho|} (L_t^= + \varrho L_t^{\neq}) }
\1_{\{ L_t^{\neq} < c
\log t, L^\mathrm{tot} \geq2c\log t\}} \bigr]
\\
&&\qquad\le t^{-{cs}/{|\varrho|}}\E_{\underline{x}} \bigl[ e^{
{s}/{|\varrho|} (L_t^= + \varrho L_t^{\neq}) } \bigr]
\\
&&\qquad\le C(\varrho) t^{-{cs}/{|\varrho|}} \leq C( \varrho) t^{-cs} = C(\varrho)
t^{-2}.
\end{eqnarray*}
Finally, consider case (iii).
Here, note that if the total collision local time is small, then
in particular the collision local time between the two Brownian motions
started at $y$ is small. That is, using \eqref{eq:classic},
\[
\E_{\underline{x}} \bigl[ e^{-s L_t^{\neq}} \1_{\{L_t^{\neq} < c
\log t,
L_t^\mathrm{tot} < 2c \log t\}} \bigr] \leq
\p_{y,y} \bigl\{ L_t^{1,2} \le2c \log t\bigr\} \leq
\frac{4c}{\sqrt{\pi}} (\log t) t^{-{1}/{2}}.
\]
A different bound can be reached by considering the collision local times
between each pair of Brownian motions started in $y,z$ and $y,x$
respectively, leading to [again using \eqref{eq:classic}]
\begin{eqnarray*}
\E_{\underline{x}} \bigl[ e^{-s L_t^{\neq}} \1_{\{
L_t^{\neq} < c
\log t, L_t^\mathrm{tot} < 2c \log t\}} \bigr] & \le&
\p_{x,y}\bigl\{ L_t^{1,2} \leq2 c \log t \bigr\}
\p_{y,z}\bigl\{ L_t^{1,2} \le 2 c \log t \bigr\}
\\
& \leq&\frac{1}{\pi} (4c \log t + y-x) (4c \log t + z-y)t^{-1}.
\end{eqnarray*}
Hence, we can take the minimum of the two bounds for (iii). Then,
we notice that since we are assuming that $t\geq1$,
cases (i) and (ii) are dominated by the contribution of (iii),
so that we obtain~\eqref{unequal_local_times}.
\end{pf}

\begin{pf*}{Proof of Proposition~\ref{mixed_moments}}
Fix $0< \eps< \frac{1}{2}$. By \eqref{eq:reforml}, it suffices to
show that
there exists a constant $C=C(\gamma,\rho,\eps)$ such that for all $z
> 0$,
%
\begin{equation}
\label{eq:0202-2} \E_{(0,z), (0,z)} \bigl[ \bigl(B^{r(t)}_t -
B^{\ell(t)}_t\bigr)^+ e^{\gamma(L_t^=
+ \rho L_t^{\neq})} \bigr] \leq C \bigl(1
\wedge z^{-2(1-\eps)}\bigr) ,
\end{equation}
which is clearly integrable in $z$.

We condition on the time of the first collision of certain pairs of the
four Brownian motions.
Indeed, let $\tau_{i,j}$ denote the first hitting time of the Brownian
motions with index $i$ and
$j$, and consider the stopping time
\[
\tau:= \tau_{1,3} \wedge\tau_{1,4} \wedge
\tau_{2,3}\wedge\tau_{2,4},
\]
which is the first time that a motion started in $0$ meets with a
motion started in $z$.

Note that we can always assume that $\tau\leq t$, for otherwise
the expectation in~(\ref{eq:0202-2}) is zero.
Then, if $(\calF_t)_{t \geq0}$ denotes the filtration of the dual process,
we can apply the strong Markov property and use that up to
time $\tau$ there are no particles of the same color that accumulate
local time. In particular, none of the particles have switched color up
to time $\tau$, so
the positions of $B_\tau^i$ at time $\tau$ and the color configuration
at time $\tau$
satisfy the assumptions of
Lemma~\ref{rough_moment_bound}.
Thus choosing $\delta:= \frac{\eps}{8}$ in Lemma~\ref
{rough_moment_bound}, we obtain that
there exists a constant $C(\rho, \gamma,\eps)$ such that
%
\begin{eqnarray}
\label{eq:1002-1} &&\E_{(0,z), (0,z)} \bigl[ \bigl(B^{r(t)}_t -
B^{\ell(t)}_t\bigr)^+ e^{\gamma
(L_t^= + \rho L_t^{\neq})} \bigr]
\nonumber
\\
&&\qquad = \E_{(0,z), (0,z)} \bigl[ \E \bigl[ \bigl(B^{r(t)}_t -
B^{\ell(t)}_t\bigr)^+ e^{\gamma(L_t^=-L_\tau^= + \rho(L_t^{\neq}- L_\tau^{\neq}) )} | \calF_\tau
\bigr]e^{\gamma(L_\tau^= + \rho L_\tau^{\neq})} \bigr]
\nonumber
\\
&&\qquad \leq4 C(\rho, \gamma, \eps) \E_{(0,z), (0,z)}\\
&&\qquad\quad{}\times \biggl[ \1_{\{ \tau= \tau_{2,3} \leq t\}}
\min \biggl\{ \frac{(B_\tau^4 - B_\tau^3+1)(B_\tau
^2- B_\tau^1 + 1)}{(t-\tau)^{{1}/{2}-\delta}},\nonumber \\
&&\hspace*{205pt}{}(t-\tau)^\delta\vee1 \biggr\}
e^{\rho\gamma(L_\tau^{1,2} + L_\tau
^{3,4})} \biggr].
\nonumber
\end{eqnarray}
Here, we also used that the four possible cases $\tau= \tau_{1,3},
\tau
_{1,4} , \tau_{2,3}, \tau_{2,4}$ are
all equally likely, and in all cases we obtain the same bound from
Lemma~\ref{rough_moment_bound}.
Moreover, in this scenario $L_\tau^{\neq} = L_\tau^{1,2} + L_\tau^{3,4}$.

In the following, we will use repeatedly the fact that for a standard
Brownian motion $(B_t)_{t\ge0}$ with maximum process $(M_t)_{t\ge0}$
and local time $(L^0_t)_{t\ge0}$ at zero, by L\'evy's equivalence (see,
e.g., Lemma~\ref{le:localtime}) we have $L_t^0\stackrel
{d}=M_t\stackrel
{d}=|B_t|$ for all $t>0$, implying that for any $s>0$ there exists a
constant $C=C(s)$ such that
for $t > 0$,
%
\begin{equation}\quad
\label{levy} \E_0 \bigl[e^{-sL_t^0} \bigr]=\E_0
\bigl[e^{-s|B_t|} \bigr] =\frac{1}{\sqrt{2\pi t}}\int_\R
e^{-{x^2}/{(2t)}} e^{-s|x|} \,dx\le C \bigl(1\wedge t^{-1/2}\bigr).
\end{equation}

In the analysis of the right-hand side of~(\ref{eq:1002-1}), we
distinguish four cases
(where we always assume $\tau\leq t$):
\begin{longlist}[(iii)]
\item[(i)] $\tau\leq z^{2-\eps}$;
\item[(ii)] $\tau> z^{2-\eps}$ and ($z^{2-\eps} > t^{{1}/{4}}$
or $t
\leq2$);
\item[(iii)] $\tau> z^{2-\eps}$, but $z^{2-\eps} \leq t^{{1}/{4}}$
and $\tau\leq t^{1/2 - \delta}$, $t \geq2$;
\item[(iv)] $\tau> z^{2-\eps}$, $z^{2-\eps} \leq t^{{1}/{4}}$, but
$\tau> t^{1/2 - \delta}$, $t \geq2$.
\end{longlist}

\textit{Case} (i). On the event that $\tau\leq z^{2-\eps}\wedge t$, we
obtain
\begin{eqnarray*}
&&\E_{(0,z),(0,z)} \biggl[ \1_{\{ \tau= \tau_{2,3} \leq z^{2-\eps}
\wedge
t\}} \min \biggl\{ \frac{(B_\tau^4 - B_\tau^3+1)(B_\tau^2- B_\tau^1 +
1)}{(t-\tau)^{{1}/{2}-\delta}},\\
&&\hspace*{225pt}(t-\tau)^\delta\vee1 \biggr\} e^{\rho\gamma L_\tau^{\neq}} \biggr]
\\
&&\qquad\le\E_{(0,z),(0,z)} \bigl[ \1_{\{\tau= \tau_{2,3}\leq z^{2-\eps}\}} \bigl(B^4_\tau-
B^3_\tau+1\bigr) \bigl(B^2_\tau-
B^1_\tau+1\bigr) \bigr]
\\
&&\qquad \leq\E_{0,0} \Bigl[ \max_{s \leq z^{2-\eps}}
\bigl(B_s^2 - B_s^1 +1
\bigr)^2 \Bigr] \p_{0,z} \bigl\{ \tau_{1,2} \leq
z^{2-\eps} \bigr\}^{{1}/ {2}}
\\
&&\qquad \leq C \bigl(1 \vee z^{2-\eps}\bigr) \p_{0,z} \bigl\{
\tau_{1,2} \leq z^{2-\eps
} \bigr\}^{{1}/ {2} },
\end{eqnarray*}
where we used the Cauchy--Schwarz inequality in the penultimate step.
In order to estimate the first collision time, denoting by $\tau(0)$
the first hitting time of $0$ for a single Brownian
motion $B$ started at~$z$, we observe that
\begin{eqnarray*}
\p_{0,z} \bigl\{ \tau_{1,2} \leq z^{2-\eps}\bigr\} & =&
\p_z \bigl\{ \tau(0) \leq 2 z^{2-\eps}\bigr\}\\
& =&
\p_0 \Bigl\{ \max_{s\leq2 z^{2-\eps}} B_s \geq z
\Bigr\}
\\
& =& 2 \p_0 \{ B_{2 z^{2-\eps}} \geq z \} \\
&\leq&1\wedge \biggl(
\frac{2}{\sqrt{\pi}} z^{-({1}/{2})\eps} e^{-{
z^{\eps}}/{4}} \biggr) ,
\end{eqnarray*}
where we used the reflection principle and a standard Gaussian
estimate; see, for example, \cite{MP10}, Remark~2.22.
Combining the previous two displays
shows that in case (i)
we obtain an upper bound
\[
C \bigl(1 \vee z^{2-\eps-({1}/{4})\eps}\bigr) e^{-({1}/{8})z^{\eps}}
\]
on the right-hand side~of (\ref{eq:1002-1}), which in turn can be estimated by the
right-hand side of (\ref{eq:0202-2}).

\textit{Case} (ii).
In this scenario,
we can find an upper bound on the expectation on the right-hand side
in~(\ref{eq:1002-1}) by
\begin{eqnarray*}
&&\E_{(0,z), (0,z)} \biggl[ \1_{\{z^{2 - \eps} < \tau= \tau_{2,3}
\leq t\}
}
\\
&&\hspace*{38pt}\quad{} \times\min \biggl\{ \frac{(B_\tau^4 - B_\tau^3+1)(B_\tau^2-
B_\tau^1 +
1)}{(t-\tau)^{{1}/{2}-\delta}}, (t-\tau)^\delta\vee1 \biggr\}
e^{\rho\gamma(L_\tau^{1,2} + L_\tau
^{3,4})} \biggr]
\\
&&\qquad \leq\E_{(0,z), (0,z)} \bigl[ \1_{\{ z^{2-\eps} < \tau_{2,3} =
\tau
\leq t \}} \bigl(1 \vee
t^\delta\bigr) e^{\gamma\rho(L^{1,2}_\tau+ L^{3,4}_{\tau})} \bigr]
\\
&&\qquad \leq\bigl(1\vee t^\delta\bigr) \E_{0,0} \bigl[ \exp \bigl(
\gamma\rho L^{1,2}_{z^{2-\eps}} \bigr) \bigr]^2 \leq C
\bigl(1 \vee t^{\delta}\bigr) \bigl(1 \wedge z^{-2 +\eps}\bigr) ,
\end{eqnarray*}
where we used the independence of the two pairs of Brownian motions and then~\eqref{levy}.
Since we assume $t \leq2$
or $z^{2-\eps} > t^{{1}/{4}}$,
this latter expression can be bounded
by $C (1 \wedge z^{-2 +\eps+ 4\delta(2-\eps)})$, which
by our choice of $\delta= \frac{\eps}{8}$ is of the required form.

\textit{Case} (iii). In this case,
we assume in particular that $t \geq2$ and $z^{2-\eps}<\tau\leq
t^{{1}/{2}-\delta}$,
so that we can estimate
\[
(t - \tau)^{-({1}/{2} - \delta)} \leq\bigl(t - t^{{1}/{2}- \delta}\bigr)^{-({1}/{2}-\delta)}
\leq C t^{-{1}/{2} + \delta}.
\]
Hence, we can
deduce from~(\ref{eq:1002-1}) that
\begin{eqnarray*}
&&\E_{(0,z), (0,z)} \biggl[ \1_{\{ z^{2-\eps} < \tau= \tau_{2,3} \leq
t^{1/2 - \delta}\} }
\\
&&\hspace*{37pt}\quad{} \times\min \biggl\{ \frac{(B_\tau^4 - B_\tau^3+1)(B_\tau
^2- B_\tau^1 + 1)}{(t-\tau)^{{1}/{2}-\delta}}, (t-\tau)^\delta\vee1 \biggr\}
e^{\rho\gamma(L_\tau^{1,2} + L_\tau
^{3,4})} \biggr]
\\
&&\qquad \leq\E_{(0,z), (0,z)} \biggl[ \1_{\{z^{2-\eps} < \tau= \tau_{2,3}
\leq t^{1/2 - \delta}\}}\\
&&\hspace*{80pt}{}\times\frac{(B_\tau^4 - B_\tau^3+1)(B_\tau^2-
B_\tau
^1 + 1)}{(t-\tau)^{{1}/{2}-\delta}}
e^{\gamma\rho(L_\tau
^{1,2} +
L_\tau^{3,4})} \biggr]
\\
&& \qquad\leq C t^{-{1}/{2} + \delta} \E_{(0,z), (0,z)} \Bigl[ \max_{s \leq t^{1/2-\delta}}
\bigl(\bigl|B_s^4 - B_s^3\bigr| + 1\bigr)
\bigl(\bigl|B_s^2- B_s^1\bigr| + 1\bigr)
\\
&&\hspace*{147pt}\qquad\quad{} \times\exp \bigl(\gamma\rho\bigl(L^{1,2}_{z^{2-\eps}} +
L^{3,4}_{z^{2-\eps}}\bigr) \bigr) \Bigr].
\end{eqnarray*}
Now, applying H\"older's inequality
with $p = \frac{1}{1-{\eps}/{2}}$ and $q$ its conjugate, and then
using the independence of the two pairs of Brownian motions,
we obtain an upper bound
\begin{eqnarray*}
&&C  t^{-{1}/{2} + \delta} \E_{(0,z), (0,z)} \Bigl[ \max_{s \leq t^{1/2-\delta}}
\bigl(\bigl|B_s^4 - B_s^3\bigr| + 1\bigr)
\bigl(\bigl|B_s^2- B_s^1\bigr| + 1\bigr)\\
&&\hspace*{148pt}{}\times\exp \bigl(\gamma\rho\bigl(L^{1,2}_{z^{2-\eps}} +
L^{3,4}_{z^{2-\eps}}\bigr) \bigr) \Bigr]
\\
&&\qquad \leq C t^{-{1}/{2} + \delta} \E_{(0,0)} \Bigl[ \max_{s \leq
t^{1/2-\delta}}
\bigl(\bigl|B_s^2 - B_s^1\bigr| + 1
\bigr)^q \Bigr]^{{2}/ q} \E_{0,0} \bigl[ \exp \bigl(
\gamma\rho p L^{1,2}_{z^{2-\eps}} \bigr) \bigr]^{
{2}/{p}}
\\
&&\qquad \leq C t^{-{1}/{2} + \delta} \E_{(0,0)} \Bigl[ \max_{s \leq1}
\bigl(t^{({1}/{2})({1}/{2}- \delta)}\bigl|B_s^2 - B_s^1\bigr|
+ 1 \bigr)^q \Bigr]^{{2}/ q}\\
&&\qquad\quad{}\times \E_{0,0} \bigl[ \exp
\bigl(\gamma\rho p L^{1,2}_{z^{2-\eps}} \bigr) \bigr]^{{2}/{p}}
\\
&&\qquad \leq C \bigl(1 \wedge z^{-(2-\eps)({1}/{p})} \bigr),
\end{eqnarray*}
where we used Brownian scaling (and $t\geq2$) to estimate the first
term and \eqref{levy}
for the second term. In particular,
we obtain that the latter expression is bounded
by $C(1 \wedge z^{ - (2- \eps)({1}/{p})})
\leq C(1 \wedge z^{-2(1-\eps)})$, by our choice
of $p$.

\textit{Case} (iv). For the remaining case
(where we can assume $t \geq2$),
we use~(\ref{eq:1002-1}) and
the independence of the Brownian motions to get an upper bound
\begin{eqnarray*}
&&\E_{(0,z), (0,z)} \biggl[ \1_{\{t^{1/2 - \delta} < \tau= \tau_{2,3}
\leq t \}}
\min \biggl\{ \frac{(B_\tau^4 - B_\tau^3+1)(B_\tau^2- B_\tau
^1 + 1)}{(t-\tau)^{{1}/{2}-\delta}},\\
&&\hspace*{230pt}{} (t-\tau)^\delta\vee1 \biggr\}
e^{\rho\gamma(L_\tau^{1,2} + L_\tau
^{3,4})} \biggr]
\\
&&\qquad \leq\bigl( 1 \vee t^{\delta}\bigr) \E_{(0,z),(0,z)} \bigl[ \exp
\bigl({\gamma\rho \bigl(L^{1,2}_{t^{1/2-\delta}} + L^{3,4}_{t^{1/2-\delta}}
\bigr)} \bigr) \bigr]
\\
&&\qquad \leq C \bigl(1 \wedge t^{-{1}/{2}+2\delta}\bigr) \leq C \bigl(1 \wedge
z^{4(2-\eps)
(-{1}/{2} + 2\delta)}\bigr)
\end{eqnarray*}
on~(\ref{eq:1002-1}),
where we used again \eqref{levy}
and finally that $z^{2-\eps} \leq t^{{1}/{4}}$.
Since $2\delta= \frac{1}{4}\eps< \frac{1}{4}$,
the resulting expression is of the form~(\ref{eq:0202-2}).

These cases exhaust all possibilities so that the Lemma is proved
via~(\ref{eq:1002-1}).
\end{pf*}

\section{Tightness}
\label{sn:tightness}

Recall that for initial conditions $( u_0, v_0 ) \in(\calB_\rap^+)^2$,
respectively, $(\calB_\tem^+)^2$, we denote by $( u_t^\sse{\gamma},
v_t^\sse
{\gamma})_{t\ge0}\in\calC_{(0,\infty)}(\calC_\rap^+)^2$,
respectively,\break $\calC
_{(0,\infty)}(\calC_\tem^+)^2$
the solution to
$\operatorname{cSBM}(\varrho,\gamma)_{ u_0, v_0}$
with these initial conditions and finite branching rate $\gamma>0$.
Also recall
that by the scaling property \eqref{scaling property}, this includes
the framework of diffusively rescaled solutions with
complementary Heaviside initial conditions as considered in \eqref{defn:mu_nu}.
We consider the measure-valued processes
%
\begin{eqnarray}
\label{defn:measure-valued processes3} \mu_t^\sse{\gamma}(dx)&:=&
u_t^\sse{\gamma}(x) \,dx,\qquad \nu_t^\sse
{\gamma }(dx):= v_t^\sse{\gamma}(x) \,dx,
\\
\label{defn:measure-valued processes4} \Lambda_t^\sse{\gamma}(dx)&:=&\gamma\int
_0^t\, ds u_s^\sse{
\gamma}(x) v_s^\sse{\gamma}(x) \,dx.
\end{eqnarray}

In this section, we will prove tightness of the above
processes
on the
space of
paths taking values in the space of rapidly decreasing, respectively
tempered, measures.
For $\rho<-\frac{1}{\sqrt{2}}$ and complementary Heaviside initial
conditions, we obtain tightness with respect to the Skorokhod topology
on the space of continuous paths.
For $\rho<0$ and general initial conditions, we can still obtain
tightness in the weaker Meyer--Zheng ``pseudopath'' topology on the
space of c\`adl\`ag paths introduced by \cite{MZ84}; see also the end
of Appendix~\ref{appendix0} for a brief description of this topology.

For tightness w.r.t. the Skorokhod topology, a nice exposition of the
general strategy in the same setting of measure-valued processes can be
found in~\cite{Dawsonetal2002}, Section~4.1. We refer the reader to
Appendix~\ref{appendix0} for a discussion of the spaces of functions
and measures that are employed in the following.

\subsection{Some preliminary estimates}\label{ssn:estimates}
In this subsection, we derive some estimates which are essential for
establishing tightness in both the Skorokhod and the Meyer--Zheng sense.
Let $( u_0, v_0 ) \in(\calB_\rap^+)^2$ [resp., $(\calB_\tem^+)^2$].
Recall that
by the Green function representation for $\operatorname{cSBM}(\varrho
,\gamma
)_{ u_0, v_0}$ (see~\cite{EF04}, Corollary~19, or Corollary~\ref
{cor:Green function representation} in the \hyperref[app]{Appendix}),
we have for every $\gamma>0$ and $\phi\in\bigcup_{\lambda>0}\calC
_{-\lambda}$ (resp., $\phi\in\bigcup_{\lambda>0}\calC_\lambda$) that
%
\begin{equation}
\label{MP1a dual} M^\sse{\gamma}_t(\phi) := \bigl\langle
u^\sse{\gamma}_t, \phi \bigr\rangle - \langle
u_0, S_t\phi \rangle,\qquad N^\sse {
\gamma}_t(\phi) := \bigl\langle v^\sse{
\gamma}_t, \phi \bigr\rangle - \langle v_0,
S_t\phi \rangle\vadjust{\goodbreak}
\end{equation}
are martingales with quadratic (co-)variation
%
\begin{eqnarray}
\label{Cov1a dual} 
&&\bigl[ M^\sse{\gamma}(\phi),
M^\sse{\gamma}(\phi) \bigr]_t
\nonumber
\\
& &\qquad= \bigl[ N^\sse{\gamma}(\phi), N^\sse{\gamma}(\phi)
\bigr]_t 
\nonumber\\
&&\qquad= \gamma\int_0^t
\int_\R S_{t-r}\phi(x)^2
u^\sse{\gamma}_r(x) v^{\sse
{\gamma}}_r(x)
\,dx \,dr,
\\
&& \bigl[ M^\sse{\gamma}(\phi) , N^\sse{\gamma}(\psi)
\bigr]_t \nonumber\\
&&\qquad= \rho\gamma \int_0^t \int
_\R S_{t-r}\phi(x) S_{t-r}\psi(x)
u^\sse{\gamma}_r(x) v^{\sse
{\gamma}}_r(x)
\,dx \,dr.
\nonumber
\end{eqnarray}

We start with the following lemma which shows in particular that the
expectation of the previous display is bounded uniformly in $\gamma>0$:

\begin{lemma}\label{lemma boundedness quadratic variation}
Suppose $\rho< 0$ and $( u_0, v_0 ) \in(\calB_\rap^+)^2$ [resp.,
$(\calB_\tem^+)^2$]. Then for all $t>0$, $\gamma>0$ and $\phi,\psi
\in
\bigcup_{\lambda>0}\calC_{-\lambda}^+$ (resp., $\bigcup_{\lambda
>0}\calC
_\lambda^+$),
we have
%
\begin{eqnarray}
\label{Cov2 dual} &&\gamma\E_{ u_0, v_0} \biggl[ \int_0^t
\int_\R S_{t-s}\phi(x) S_{t-s}\psi (x)
u^\sse{\gamma}_s(x) v^{\sse{\gamma}}_s(x) \,dx
\,ds \biggr]\nonumber
\\
&&\qquad=\frac{1}{|\rho|} \iint\phi(x)\psi(y) \E_{x,y} \bigl[
u_0\bigl(B^\ssup {1}_t\bigr) v_0
\bigl(B^\ssup{2}_t\bigr) \bigl(1-e^{\gamma\rho L_t^{1,2}} \bigr)
\bigr] \,dx\,dy
\\
&&\qquad\uparrow\frac{1}{|\rho|} \iint\phi(x)\psi(y) \E_{x,y} \bigl[
u_0\bigl(B^\ssup{1}_t\bigr) v_0
\bigl(B^\ssup{2}_t\bigr) \1_{\{ L_t^{1,2}>0\}} \bigr] \,dx\,dy<
\infty,\nonumber
\end{eqnarray}
as $\gamma\uparrow\infty$,
where $B^\ssup{1},B^\ssup{2}$ are independent Brownian motions with
intersection local time $L^{1,2}$.
\end{lemma}

\begin{pf}
First, note that the limit on the right-hand side of \eqref{Cov2 dual} holds by
monotone convergence since $(1-e^{\gamma\rho L_t^{1,2}})\uparrow\1
_{L_t^{1,2}>0}$ as $\gamma\uparrow\infty$. Also observe that the right-hand side is
finite under our assumptions since it is bounded by
%
\begin{eqnarray}
\label{upper bound Cov2a dual}&& \frac{1}{|\rho|} \iint\phi(x)\psi(y) \E_{x,y} \bigl[
u_0\bigl(B^\ssup{1}_t\bigr) v_0
\bigl(B^\ssup{2}_t\bigr) \bigr]\,dx\,dy
\nonumber
\\[-8pt]
\\[-8pt]
\nonumber
&&\qquad=\frac{1}{|\rho|}
\langle\phi,S_t u_0\rangle \langle\psi,S_t
v_0\rangle<\infty;
\end{eqnarray}
see, for example, Lemma~\ref{lemma estimates}(a).

In order to show the first equality in \eqref{Cov2 dual}, we adapt and
elaborate an argument from the proof of \cite{T95}, Lemma~4.4: For a
suitable process $X$, denote by $(L_t^{x,X})_{t\ge0}$ the local time of
$X$ at $x\in\R$. Let $B^\ssup{1}$, $B^\ssup{2}$ be independent Brownian
motions. Then by a change of variables $s\mapsto t-s$, Fubini's theorem
and the colored particle moment duality, we have
%
\begin{eqnarray}
\label{proof lbqv1} &&\gamma \E_{ u_0, v_0} \biggl[ \int_0^t
\,ds \int_\R \,dx S_{t-s}\phi (x)S_{t-s}
\psi(x) u^\sse{\gamma}_s(x) v^{\sse{\gamma}}_s(x)
\biggr]\nonumber
\\
&&\qquad = \gamma\int_0^t\,ds\int_\R
\,dx S_{s}\phi(x)S_{s}\psi(x)
\\
&&\qquad\quad{} \times\E_{0,0} \bigl[ u_0\bigl(B^{(1)}_{t-s}+x
\bigr) v_0\bigl(B_{t-s}^{(2)}+x\bigr)\exp \bigl(
\gamma\rho L_{t-s}^{0,B^\ssup{2}-B^\ssup{1}}\bigr) \bigr].\nonumber
\end{eqnarray}
Writing $B:=(B^\ssup{1}, B^\ssup{2})$ and denoting by $(\calF
_s)_{s\ge
0}$ the natural filtration of $B$, we use that (by the independence and
stationarity of the increments) for functionals $f(B_\cdot)$ of the
two-dimensional Brownian path, we have
\[
\E_{0,0} \bigl[f(B_{\cdot+ s} - B_s)|
\calF_s \bigr]\equiv\E _{0,0} \bigl[f(B_{\cdotp})
\bigr]
\]
for each fixed time $s\ge0$. Applying this with the functional
\[
f(B_\cdot):= u_0\bigl(B^\ssup{1}_{t-s}+x
\bigr) v_0\bigl(B^\ssup{2}_{t-s}+x\bigr)\exp
\bigl(\gamma \rho L_{t-s}^{0,B^\ssup{2}-B^\ssup{1}}\bigr)
\]
for $s\in[0,t]$
and then shifting the $\,dx$-integral (change of variables $y:=-B_s^\ssup
{2}+x$), we
see that \eqref{proof lbqv1} is equal to
\begin{eqnarray*}
&&\gamma\int_0^t \,ds\int_{\mathbb{R}}
\,dx \E_{0,0} \bigl[ \phi \bigl(-B_s^\ssup {1} + x
\bigr) \psi\bigl(-B_s^\ssup{2} + x\bigr)
\\
&&\hspace*{75pt}\quad{} \times\mathbb{E}_{(0,0)} \bigl[ u_0\bigl(B^\ssup{1}_t
- B^\ssup{1}_s+x\bigr) v_0
\bigl(B^\ssup{2}_t - B^\ssup{2}_s+x
\bigr)
\\
&&\hspace*{114pt}\quad{} \times\exp (\gamma\rho\bigl(L_{t}^{0,B^\ssup{2}- B_s^\ssup{2} -
(B^\ssup{1}-B^\ssup{1}_s)}\\
&&\hspace*{153pt}{}\qquad-L_s^{0, B^\ssup{2}- B_s^\ssup{2}
 -
(B^\ssup
{1}-B^\ssup{1}_s)}
\bigr) | \calF_s \bigr] \bigr]
\\
&&\qquad = \gamma\int_0^t\,ds \int
_{\mathbb{R}}\,dy \E_{0,0} \bigl[ \phi \bigl(B_s^\ssup
{2} - B_s^\ssup{1} + y\bigr) \psi(y)
\\
&&\hspace*{85pt}\qquad\quad{} \times u_0\bigl(B^\ssup{1}_t -
B^\ssup{1}_s+B^\ssup{2}_s+y\bigr)
v_0\bigl(B^\ssup {2}_t +y\bigr)
\\
&&\hspace*{85pt}\qquad\quad{} \times\exp \bigl(\gamma\rho
\bigl(L_{t}^{0,B^\ssup{2}- B_s^\ssup{2}
- (B^\ssup{1}-B^\ssup{1}_s)}\\
&&\hspace*{195pt}{}-L_s^{0,B^\ssup{2}- B_s^\ssup{2} -
(B^\ssup
{1}-B^\ssup{1}_s)}
\bigr) \bigr) \bigr]
\\
&&\qquad = \gamma\int_\R \,dy \psi(y) \E_{0,0} \biggl[
v_0\bigl(B^\ssup{2}_t + y\bigr)\int
_0^t \,ds \phi\bigl(B_s^\ssup{2}
- B_s^\ssup{1} + y\bigr)
\\
&&\hspace*{81pt}\qquad\quad{} \times u_0\bigl(B^\ssup {1}_t +
B^\ssup{2}_s - B^\ssup{1}_s+y\bigr)
\\
&&\hspace*{81pt}\qquad\quad{} \times\exp \bigl(\gamma\rho
 \bigl(L_{t}^{B^\ssup{2}_s-B_s^\ssup{1},
B^\ssup{2}-B^\ssup{1}}\\
&&\hspace*{180pt}{}-L_s^{B^\ssup{2}_s - B_s^\ssup{1}, B^\ssup
{2} -
B^\ssup{1}}
\bigr) \bigr) \biggr].
\end{eqnarray*}
Now for the inner integral $\int_0^t\cdots \,ds$, we apply Lemma~\ref
{lem:occ_times} in the \hyperref[app]{Appendix}
and then another change of variables $x:=y+z$
to see that the above equals
\begin{eqnarray*}
&&\hspace*{-4pt} \int_\R \,dy \psi(y) \E_{0,0} \biggl[
v_0\bigl(B^\ssup{2}_t + y\bigr) \int
_\R \,dz \phi(z + y) u_0\bigl(B^\ssup{1}_t
+ z + y\bigr)
\\
&&\hspace*{-6pt}\hspace*{62pt}\quad{} \times\int_0^t \,dL_s^{z,B^\ssup{2}-B^\ssup{1}}
\gamma\exp \bigl(\gamma \rho \bigl(L_t^{z,B^\ssup{2} - B^\ssup{1}} -
L_s^{z,B^\ssup{2} -
B^\ssup
{1}} \bigr) \bigr) \biggr]
\\
&&\hspace*{-6pt}\qquad = \iint \,dx\,dy \phi(x)\psi(y) \\
&&\hspace*{-6pt}\qquad\quad{}\times\E_{0,0} \biggl[ u_0
\bigl(B^\ssup{1}_t + x\bigr) v_0
\bigl(B^\ssup{2}_t + y\bigr)
\\
&&\hspace*{-6pt}\hspace*{30pt}\qquad\quad{} \times\int_0^t \,dL_s^{x-y,B^\ssup{2}-B^\ssup{1}}
\gamma\exp \bigl(\gamma\rho \bigl(L_t^{x-y,B^\ssup{2} - B^\ssup{1}} -
L_s^{x-y,B^\ssup
{2} - B^\ssup{1}} \bigr) \bigr) \biggr]
\\
&&\hspace*{-6pt}\qquad = \iint \,dx\,dy \phi(x)\psi(y) \E_{0,0} \biggl[ u_0
\bigl(B^\ssup{1}_t + x\bigr) v_0
\bigl(B^\ssup{2}_t + y\bigr)
\\
&&\hspace*{-6pt}\hspace*{113pt}\qquad\quad{} \times\frac{1}{|\rho|} \bigl(1-\exp \bigl(\gamma\rho L_t^{x-y,B^\ssup
{2}-B^\ssup{1}}
\bigr) \bigr) \biggr],
\end{eqnarray*}
which gives the first equality in \eqref{Cov2 dual}.
\end{pf}

From the above estimate, we obtain a uniform bound on the first moment
of $( u^\sse{\gamma}, v^\sse{\gamma},\Lambda^\sse{\gamma})$ integrated
against suitable test functions:

\begin{lemma}\label{lemma uniform first moments dual process}
Suppose $\rho< 0$ and $( u_0, v_0 ) \in(\calB_\rap^+)^2$ [resp.,
$(\calB_\tem^+)^2$]. Then for all $T > 0$ and $\phi\in\bigcup_{\lambda
>0}\calC_{-\lambda}^+$ (resp., $\bigcup_{\lambda>0}\calC_\lambda^+$),
we have
%
\begin{equation}
\label{uniform first moments dual process 1} \sup_{\gamma>0} \E_{ u_0, v_0} \Bigl[ \sup
_{0\leq t \leq T} \bigl\langle u^\sse {\gamma}_t,
\phi \bigr\rangle \Bigr] < \infty, \qquad \sup_{\gamma>0}
\E_{ u_0, v_0} \Bigl[ \sup_{0\leq t \leq T} \bigl\langle
v^\sse{\gamma}_t, \phi \bigr\rangle \Bigr] < \infty
\end{equation}
and
%
\begin{equation}
\label{uniform first moments dual process 2} \sup_{\gamma>0} \E_{ u_0, v_0} \Bigl[ \sup
_{0\leq t \leq T} \bigl\langle \Lambda^\sse{
\gamma}_t, \phi \bigr\rangle \Bigr] < \infty.
\end{equation}
\end{lemma}

\begin{pf}
Suppose $( u_0, v_0 ) \in(\calB_\rap^+)^2$.
By \eqref{MP1a dual} and \eqref{Cov1a dual}, using the\break 
Burkholder--Davis--Gundy and Jensen inequalities as well as Lemma~\ref
{lemma boundedness quadratic variation} [recall the upper bound \eqref
{upper bound Cov2a dual}], we have
\begin{eqnarray*}
\E_{ u_0, v_0} \Bigl[ \sup_{0\leq t \leq T}  \bigl\langle
u^\sse {\gamma}_t, \phi \bigr\rangle \Bigr] &\le&
\E_{ u_0, v_0} \Bigl[ \sup_{t\in[0,T]} \bigl| M_t^\sse
{\gamma}(\phi )\bigr| \Bigr] + \sup_{0\leq t \leq T} \langle u_0,
S_t\phi \rangle
\\
&\le& C \bigl(\E_{ u_0, v_0} \bigl[\bigl[ M^\sse{\gamma}(\phi),
M^\sse {\gamma }(\phi)\bigr]_T \bigr]
\bigr)^{1/2} + \sup_{0\leq t \leq T} \bigl| \langle u_0,
S_t\phi \rangle \bigr|
\\
&\le& C \biggl( \frac{1}{|\rho|}\langle\phi, S_T u_0
\rangle \langle \phi ,S_T v_0\rangle
\biggr)^{1/2} + \sup_{0\leq t \leq T} \bigl| \langle u_0,
S_t\phi \rangle \bigr|<\infty,
\end{eqnarray*}
and analogously for $\tilde v^\sse{\gamma}$. Since this bound is
independent of $\gamma$, \eqref{uniform first moments dual process 1}
follows. In order to show \eqref{uniform first moments dual process 2},
assume without loss of generality that $\phi=\phi_{\lambda}$ with
$\lambda<0$. Since
\[
\phi_{\lambda}(x)\le C(\lambda,T)\inf_{t\in[0,T]}S_t
\phi_{\lambda}(x),\qquad  x\in\R
\]
[see Lemma~\ref{lemma estimates}, estimate \eqref{estimate 1a} in the
\hyperref[app]{Appendix}],
and again applying bound \eqref{upper bound Cov2a dual}, we get
\begin{eqnarray*}
\E_{ u_0, v_0} \Bigl[ \sup_{0\leq t \leq T} \bigl\langle\Lambda
^\sse {\gamma}_t, \phi_\lambda \bigr\rangle
\Bigr]&=&\E_{ u_0, v_0} \bigl[ \bigl\langle\Lambda^\sse{
\gamma}_T, \phi_{\lambda/2}^2 \bigr\rangle \bigr]
\\
&=&\gamma\E_{ u_0, v_0} \biggl[ \int_{0}^T
\int_\R\phi_{\lambda/2}(x)^2
u^\sse{\gamma}_s(x) v^\sse{
\gamma}_s(x) \,dx\,ds \biggr]
\\
&\le& C\gamma\E_{ u_0, v_0} \biggl[ \int_0^T
\int_\R\bigl(S_{T-s}\phi _{\lambda
/2}(x)
\bigr)^2 u^\sse{\gamma}_s(x)
v^\sse{\gamma}_s(x) \,dx\,ds \biggr]
\\
&\le&\frac{C}{|\rho|}\langle\phi_{\lambda/2},S_T
u_0 \rangle \langle \phi_{\lambda/2},S_T
v_0 \rangle<\infty
\end{eqnarray*}
uniformly in $\gamma>0$.

The proof for initial conditions in $(\calB_\tem^+)^2$ is completely
analogous.
\end{pf}

\begin{corollary}[(Compact containment)]\label{cor:comp_cont}
Suppose $\rho< 0$ and $( u_0, v_0 ) \in(\calB_\tem^+)^2$ [resp.,
$(\calB_\rap^+)^2$].
Then the \emph{compact containment condition} holds for the family $(
u^\sse{\gamma}_t, v^\sse{\gamma}_t,\Lambda^\sse{\gamma}_t)_{t\ge0}$;
that is, for every $\eps>0$ and $T>0$, there exists a compact subset
$K=K_{\eps,T}\subseteq\calM_\tem$ (resp., $\calM_\rap$) such that
\[
\inf_{\gamma>0}\p \bigl\{ u^\sse{\gamma}_t
\in K_{\eps,T} \mbox{ for all } t \in[0,T] \bigr\} \ge1-\eps,
\]
and similarly for $ v_t^\sse{\gamma}$ and $ \Lambda_t^\sse{\gamma}$.
\end{corollary}

\begin{pf}
Let $( u_0, v_0 ) \in(\calB_\tem^+)^2$. To check the compact
containment condition, as in the proof of \cite{Dawsonetal2002}, Proposition~37, use compact subsets of $\calM_\tem$ of the form
\[
K = K\bigl((c_m)_{m\in\N}\bigr) := \bigl\{ \nu\in
\calM_\tem\dvtx\langle\nu, \phi _{1/m} \rangle\leq
c_m \mbox{ for all } m \in\N\bigr\} ,
\]
where $(c_m)_{m\in\N}$ is a sequence of positive numbers: Given $\eps>
0$ and $T>0$, for any $m \in\N$ we can find
by Lemma~\ref{lemma uniform first moments dual process} a number $c_m =
c_m(\eps,T) > 0$ such that for all $\gamma>0$,
\[
\p \Bigl\{ \sup_{0 \leq t \leq T} \bigl\langle u_t^\sse{
\gamma}, \phi_{1/m} \bigr\rangle\geq c_m \Bigr\} \leq
\frac{\eps}{2^m}.
\]
In particular, it follows that for all $\gamma>0$,
%
\begin{equation}
\label{compcontain} \p \bigl\{ u^\sse{\gamma}_t \in K
\bigl((c_m)_{m \in\N
}\bigr) \mbox { for all } t \in[0,T] \bigr
\} \geq1 - \eps.
\end{equation}
The same reasoning shows that the compact condition also holds for $
v^\sse{\gamma}$ and~$\Lambda^\sse{\gamma}$.

The proof for rapidly decreasing initial conditions $(\tilde u_0,
\tilde v_0 ) \in(\calB_\rap^+)^2$ is completely analogous, using
compact subsets of $\calM_\rap$ of the form
\[
K = K\bigl((c_m)_{m\in\N}\bigr) := \bigl\{ \nu\in
\calM_\rap\dvtx\langle\nu, \phi_{-m} \rangle\leq
c_m \mbox{ for all } m \in\N\bigr\}
\]
together with Lemma~\ref{lemma uniform first moments dual process}.
\end{pf}

\subsection{Tightness in $\mathcal{C}$}

In this subsection we will prove tightness of the family of processes
\eqref{defn:measure-valued processes3}--\eqref{defn:measure-valued processes4}
with respect to the Skorokhod topology on the
space of continuous paths.
The proof relies on the fourth moment bound of Proposition~\ref
{mixed_moments} and thus requires complementary Heaviside initial
conditions and the condition $\rho<-\frac{1}{\sqrt{2}}$.

In the first step, we establish tightness of the above measures
integrated against suitable test functions:

\begin{lemma}\label{tst_fn_tight}
Suppose\vspace*{-1.5pt} $\rho< - \frac{1}{\sqrt{2}}$
and $(u_0, v_0 ) = (\1_{\R^-}, \1_{\R^+})$.
Then for all $\phi\in\bigcup_{\lambda>0} \calC_\lambda$ the
family of
coordinate processes $(\langle\phi, u^\sse{\gamma}_t \rangle
,\langle
\phi, v^\sse{\gamma}_t \rangle,\langle\phi, \Lambda^\sse{\gamma}_t
\rangle)_{ t \geq0}$, considered as a family indexed over $\gamma>0$,
is tight
in the space $\calC_{[0,\infty)}(\R^3)$.
\end{lemma}
Having established the fourth moment bound in Proposition~\ref{mixed_moments},
the proof of tightness follows closely the proof of~\cite{T95}, Lemma~4.1.
\begin{pf*}{Proof of Lemma \ref{tst_fn_tight}}
The Green function representation for $\operatorname{cSBM}(\rho,\break  \gamma
)_{u_0,v_0}$ (see, e.g., Corollary~\ref{cor:Green function
representation}) yields for $\phi\in\bigcup_{\lambda>0} \calC
_\lambda$ that
%
\begin{equation}
\label{rgfr} \bigl\langle\phi, u^\sse{\gamma}_t \bigr
\rangle= \langle\phi, S_t u_0 \rangle + \int
_{[0,t] \times\mathbb{R}} S_{t-r} \phi(x) M^\sse{
\gamma}(dr,dx) ,
\end{equation}
where $(S_t)_{t \geq0}$ denotes the heat semigroup, and $M^\sse
{\gamma}(dr,dx)$
is a zero-mean martingale measure with
quadratic variation given by
\[
\gamma\int_0^t\int_\R
u^\sse{\gamma}_r(x)v^\sse{\gamma
}_r(x) \bigl(S_{t-r}\phi (x)\bigr)^2 \,dx\,dr.
\]

We check Kolmogorov's tightness criterion for the stochastic integral
in (\ref{rgfr}).
For $0 < s < t$,
we have
%
\begin{eqnarray}
\label{eq:2702-1} &&\int_{[0,t] \times\mathbb{R}} S_{t-r} \phi(x)
M^\sse{\gamma }(dr,dx) - \int_{[0,s] \times\mathbb{R}}
S_{s-r} \phi(x) M^\sse{\gamma }(dr,dx)\nonumber
\\
& &\qquad= \int_{[s,t] \times\mathbb{R}} S_{t-r} \phi(x) M^\sse{
\gamma }(dr,dx)
\\
& &\qquad\quad{}+ \int_{[0,s] \times\mathbb{R}} \bigl(S_{t-r} \phi(x) -
S_{s-r} \phi(x)\bigr) M^\sse{\gamma}(dr,dx).\nonumber
\end{eqnarray}
Consider the fourth moment of the first term on the right-hand side
in~(\ref{eq:2702-1}): Using first the Burkholder--Davis--Gundy
inequality, then Jensen's inequality, the scaling property \eqref
{scaling property}
and finally the fourth moment bound of Proposition~\ref{mixed_moments}
for $\rho< -\frac{1}{\sqrt{2}}$, we obtain
%
\begin{eqnarray}
\label{eq:2702-2} %
&&\E_{u_0,v_0} \biggl[ \biggl( \int
_{[s,t] \times\mathbb{R}} S_{t-r} \phi (x) M^\sse{
\gamma}(dr,dx) \biggr)^4 \biggr]\nonumber
\\
&&\qquad \le C \E_{u_0,v_0} \biggl[ \biggl( \gamma\int_s^t
\int_\R u^\sse {\gamma }_r(x)
v^\sse{\gamma}_r(x) \bigl(S_{t-r} \phi(x)
\bigr)^2 \,dx\,dr \biggr)^2 \biggr]\nonumber
\\
&&\qquad \leq C \Vert\phi\Vert^4_{\infty} (t-s)^2
\E_{u_0,v_0} \biggl[ \biggl(\frac
{1}{t-s}\int_s^t
\int_{\mathbb{R}} \gamma u^\sse{\gamma}_r(x)
v^\sse{\gamma}_r(x) \,dx\,dr \biggr)^{ 2} \biggr]
\\
&&\qquad \leq C(\phi) (t-s) \E_{u_0,v_0} \biggl[ \int_s^t
\biggl( \int_{\mathbb{R}} u^\sse{1}_{\gamma^2 r}(x)
v^\sse{1}_{\gamma^2 r}(x) \,dx \biggr)^{ 2} \,dr \biggr]\nonumber
\\
&&\qquad \leq C(u_0,v_0,\phi, \rho) (t-s)^2.\nonumber
\end{eqnarray}
Now consider the expectation of the fourth power of the second term on
the right-hand side in~(\ref{eq:2702-1}):
Again using the Burkholder--Davis--Gundy inequality and the elementary bound
\[
\Vert S_t \phi- S_s \phi\Vert_{\infty} \leq2
\Vert\phi\Vert _{\infty} \bigl((t-s) s^{-1} \wedge1 \bigr),
\]
which follows from the estimate $\Vert\partial_r S_r\phi\Vert
_\infty\leq
\Vert\phi\Vert_\infty\frac{1}{r}$
together with $\|S_r \phi\|_\infty\le\| \phi\|_\infty$ for any $r >
0$, we have
\begin{eqnarray}
\label{eq:0403-2} &&\E_{u_0,v_0} \biggl[ \biggl(\int_{[0,s] \times\mathbb{R}}
\bigl(S_{t-r} \phi (x) - S_{s-r} \phi(x)\bigr)
M^\sse{\gamma}(dr,dx) \biggr)^4 \biggr]
\nonumber
\\
& &\qquad\leq C \| \phi\|_\infty^4 \E_{u_0,v_0} \biggl[
\biggl( \int_0^s \int_{\R
}
\gamma u^\sse{\gamma}_{ r}(x)v^\sse{
\gamma}_{ r}(x)\\
&&\hspace*{134pt}{}\times \bigl((t-s)^2(s-r)^{-2} \wedge1
\bigr) \,dx \,dr \biggr)^{ 2} \biggr].
\nonumber
\end{eqnarray}
Now defining
\[
f(r) := 1 \wedge(t-s)^2(s-r)^{-2},\qquad  r\in[0,s],
\]
we can rewrite the right-hand side of~(\ref{eq:0403-2}) and then apply
Jensen's inequality,
the scaling property
and finally the fourth moment bound to obtain
%
\begin{eqnarray}
\label{eq:0403-3}&& C \| \phi\|_\infty^4 \biggl( \int
_0^s f(r)\,dr \biggr)^2
\E_{u_0,v_0} \biggl[ \biggl( \frac{1}{\int_0^s f(r) \,dr } \int_0^s
\int_{\R} \gamma u^\sse {\gamma}_{
r}(x)v^\sse{
\gamma}_{ r}(x) \,dx f(r) \,dr \biggr)^{ 2} \biggr]
\nonumber
\\
&&\qquad\leq C \|\phi\|_\infty^4 \int_0^s
f(r) \,dr \int_0^s \E_{u_0,v_0} \biggl[
\biggl( \int_{\mathbb{R}} u^\sse{1}_{\gamma^2 r}(x)
v^\sse {1}_{\gamma^2
r}(x) \,dx \biggr)^{ 2} \biggr] f(r)
\,dr
\\
&&\qquad \leq C(u_0,v_0,\phi, \rho) \biggl(\int
_0^s f(r) \,dr \biggr)^2.
\nonumber
\end{eqnarray}
Now note that if $s \in[\frac{t}{2},t)$, we have by an explicit calculation
\[
\int_0^s f(r) \,dr = \int_0^{2s-t}
\biggl(\frac{t-s}{s-r} \biggr)^{2} \,dr + \int_{2s-t}^s
1 \,dr = 2(t-s) - \frac{(t-s)^2}{s} \leq2(t-s).
\]
On the other hand, if $s \in[0,\frac{t}{2}]$,
we find that $f(r)=1$ for all $r\in[0,s]$ and thus
\[
\int_0^s f(r) \,dr = s\le t-s.
\]
Thus in both cases we obtain from \eqref{eq:0403-3}
that \eqref{eq:0403-2}
is bounded by $4 C (t-s)^2$.

Combining the fourth moment estimates of the two terms
in~(\ref{eq:2702-1}), one can deduce that
\begin{eqnarray*}&& \E_{u_0,v_0} \biggl[ \biggl(\int
_{[0,t] \times\mathbb{R}}  S_{t-r} \phi(x) M^\sse{
\gamma}(dr,dx) - \int_{[0,s] \times
\mathbb
{R}} S_{s-r} \phi(x)
M^\sse{\gamma}(dr,dx) \biggr)^4 \biggr]
\\
&&\qquad \leq C (t-s)^2 ,
\end{eqnarray*}
confirming that the stochastic integral satisfies Kolmogorov's
tightness criterion. The proof for tightness of $\langle\phi, v^\sse
{\gamma}_t \rangle$ is analogous. Finally, noting that
\[
\E_{u_0,v_0} \bigl[\bigl\langle\Lambda_t^\sse{
\gamma}-\Lambda_s^\sse {\gamma },\phi\bigr
\rangle^2 \bigr]\le\|\phi\|^2_\infty
\E_{u_0,v_0} \biggl[ \biggl(\gamma\int_s^t
\int_\R u_r^\sse{
\gamma}(x)v_r^\sse{\gamma}(x) \,dx\,dr \biggr)^2
\biggr],
\]
tightness of $\langle\phi, \Lambda^\sse{\gamma}_t \rangle$ follows
by the same argument as in \eqref{eq:2702-2}.
\end{pf*}
%

\begin{rem}
Note that for the application of Kolmogorov's tightness criterion, it
would suffice to control $p$th moments for any $p>2$ instead of $p=4$
in the above proof. However, the duality technique only allows us to
estimate integer (mixed) moments.
This is the reason for the restriction $\rho<\rho(4)=-\frac{1}{\sqrt
{2}}$ in the above approach. We believe this restriction to be due to
the technique of the proof (duality), however, and expect the above
results to hold for \emph{all} $\rho<\rho(2)=0$.
Since our approach allows us to control second moments, we can at least
show tightness in the weaker Meyer--Zheng topology for all $\rho<0$;
see Proposition~\ref{tightness_MZ dual} below.
\end{rem}

\begin{prop}\label{tightness of processes}
Let $\rho<-\frac{1}{\sqrt{2}}$ and $(u_0, v_0 ) = (\1_{\R^-}, \1
_{\R
^+})$. Then the family $(\mu^\sse{\gamma}_t,\nu^\sse{\gamma}_t,
\Lambda
^\sse{\gamma}_t)_{ t \geq0 }$ of measure-valued processes is tight
with respect to the Skorokhod topology on the space
$\calC_{[0,\infty)}( \calM_\tem^3)$.
\end{prop}

\begin{pf}
By a standard argument known as Jakubowski's criterion
(see \cite{Jak86}, Theorem~3.1 or \cite{Daw91}, Theorem~3.6.4; see
also~\cite{EK86}, Theorem~3.9.1),
tightness of the measure-valued processes
follows from tightness of the coordinate
processes together with the compact containment condition.
We have already checked the latter in Corollary~\ref{cor:comp_cont}.
Moreover, for each test function $\phi\in\bigcup_{\lambda>0}\calC
_\lambda^+$, the coordinate processes $\langle\phi, u^\sse{\gamma}_t
\rangle$, $\langle\phi, v^\sse{\gamma}_t \rangle$ and $\langle
\phi,
\Lambda^\sse{\gamma}_t \rangle$ are tight
in $\calC_{[0,\infty)}(\R)$ by Lemma~\ref{tst_fn_tight}. Since the
family of functions
$ \{\langle\phi,\cdot\rangle\dvtx\phi\in\bigcup_{\lambda
>0}\calC
_\lambda^+  \}$
is separating for $\calM_\tem$ [recall the definition of the topology
of $\calM_\tem$ in \eqref{defn:top_tem}],
an application of~\cite{Jak86}, Theorem~3.1, completes the proof.
\end{pf}

\begin{rem}
Note that the restriction to $\rho<-\frac{1}{\sqrt{2}}$ and
complementary Heaviside initial conditions in the previous proposition
comes only from Lem\-ma~\ref{tst_fn_tight} (tightness of coordinate
processes). The compact containment condition,
on the other hand, holds for \emph{all} $\rho<0$ and general initial
conditions by Corollary~\ref{cor:comp_cont}.
As a consequence,
any generalization of Lemma~\ref{tst_fn_tight} to other values of
$\rho
<0$ or to more general initial conditions would immediately result in a
corresponding strengthening of the conclusion in Proposition~\ref
{tightness of processes}.
\end{rem}

\subsection{Meyer--Zheng tightness}
The approach of the previous subsection relies heavily on the
assumption of complementary Heaviside initial conditions and
that $\rho<-\frac{1}{\sqrt{2}}$. In particular, only under those
conditions we are able to establish the fourth moment bound of
Proposition~\ref{mixed_moments}, which in turn is essential for proving
tightness in the space of continuous paths w.r.t. the Skorokhod
topology. We will see now that both assumptions can be weakened if we
consider tightness w.r.t. the weaker Meyer--Zheng ``pseudopath''
topology on the space of c\`adl\`ag paths.
This extension will be of crucial importance in the uniqueness proof in
Section~\ref{ssn:martconv} below. Indeed, in order to show uniqueness
of the limit point, we use self-duality for solutions of the martingale
problem $(\mathbf{MP'})_{\mu_0,\nu_0}^\rho$. Therefore, we have to
construct a dual process, that is, solutions to the martingale problem,
for a sufficiently rich class of rapidly decreasing initial conditions.
In particular, we will need solutions for initial conditions with
nondisjoint support.

We now show that tightness of the family $(\mu_t^\sse{\gamma},\nu
_t^\sse
{\gamma},\Lambda_t^\sse{\gamma})_{t\ge0}$ of measure-valued processes
in the Meyer--Zheng topology is a simple consequence of the estimates
already derived in Section~\ref{ssn:estimates}:

\begin{prop}\label{tightness_MZ dual}
Suppose $\rho< 0$ and $( u_0, v_0 ) \in(\calB_\rap^+)^2$ [resp.,
$(\calB_\tem^+)^2$]. Then the family of processes $(\mu_t^\sse
{\gamma
},\nu_t^\sse{\gamma},\Lambda_t^\sse{\gamma})_{t\ge0}$ from
\eqref
{defn:measure-valued processes3}--\eqref{defn:measure-valued
processes4} is tight with respect to the Meyer--Zheng topology on
$D_{[0,\infty)}(\calM_\rap^3)$ [resp., $D_{[0,\infty)}(\calM_\tem^3)$].
\end{prop}

\begin{pf}
Suppose $( u_0, v_0 ) \in(\calB_\rap^+)^2$.
We aim at applying \cite{Kurtz91}, Corollary~1.4, which requires us to
check the Meyer--Zheng tightness condition [see, e.g., \eqref{MZ
condition} in the \hyperref[app]{Appendix}] for the coordinate
processes plus a compact
containment condition. Let $\phi\in\calC_\tem^+$, and fix $T>0$.

For $(\langle\phi, u^\sse{\gamma}_t \rangle)_{t\ge0}$, in view of
\eqref{MP1a dual}
and since $t\mapsto\langle u_0,S_t\phi\rangle$ has finite variation
on $[0,T]$,
checking the Meyer--Zheng condition \eqref{MZ condition} amounts to
showing that
\[
\sup_{\gamma>0}\sup_{t\in[0,T]}\E_{ u_0, v_0}
\bigl[\bigl\langle\phi, u^\sse {\gamma}_t\bigr\rangle
\bigr]<\infty,
\]
which is, however, implied immediately by Lemma~\ref{lemma uniform
first moments dual process}.
The same argument works for $(\langle\phi, v^\sse{\gamma}_t\rangle
)_{t\ge0}$. For the increasing process $t\mapsto\langle\phi,
\Lambda
^\sse{\gamma}_t \rangle$, condition \eqref{MZ condition} reduces to
\[
\sup_{\gamma>0}\E_{ u_0, v_0} \bigl[\bigl\langle\phi,
\Lambda^\sse {\gamma }_T\bigr\rangle \bigr]<\infty,
\]
which is also ensured by Lemma~\ref{lemma uniform first moments dual process}.

This shows that the Meyer--Zheng tightness criterion is satisfied for
the coordinate processes.

The compact containment condition has already been checked in Corollary~\ref{cor:comp_cont}. Applying \cite{Kurtz91}, Corollary~1.4, we are done.
The proof for initial conditions in $(\calB_\tem^+)^2$ is completely
analogous.
\end{pf}
%

\section{Properties of limit points}
\label{sn:properties}
Having established tightness of our family \eqref{defn:measure-valued
processes3}--\eqref{defn:measure-valued processes4} of measure-valued
processes on path space,
we turn to the investigation of the properties of limit points in the
respective topologies. Our starting point is the observation that each
limit point w.r.t. the Skorokhod topology on $\calC$ satisfies the
martingale problem $(\mathbf{MP})_{\mu_0,\nu_0}^\rho$ from
Definition~\ref{defn:MP}. This implies in
particular
the absolute continuity of the limit measures which is part of our main
result Theorem~\ref{thmm:main2}.
We will also see that limit points w.r.t. the (weaker) Meyer--Zheng
topology still satisfy the (weaker) martingale problem $(\mathbf
{MP'})_{\mu_0,\nu_0}^\rho$ from Definition~\ref{defn:MP'}, which will
be used in the proof of self-duality and uniqueness later on.

The second fundamental observation is the fact that each Meyer--Zheng
limit point has the ``separation of types'' property \eqref{singularity
1} (see Lemma~\ref{lemma:sep} below), which will allow us to prove
self-duality and uniqueness without having to specify the quadratic
variation of the limit martingales in the martingale problem $(\mathbf
{MP})_{\mu_0,\nu_0}^\rho$. For
Skorokhod limit points, this will also imply the separation of types in
the more intuitive sense \eqref{singularity1}.

Suppose $(\mu_t,\nu_t,\Lambda_t)_{t\ge0}$ is a limit point of the
family \eqref{defn:measure-valued processes3}--\eqref
{defn:measure-valued processes4} of measure-valued processes.
[Recall that by Proposition~\ref{tightness of processes}, such a limit
point exists under complementary Heaviside initial conditions
$(u_0,v_0)=(\1_{\R^-},\1_{\R^+})$ whenever $\rho<-\frac{1}{\sqrt{2}}$.]
By the definition of the finite rate symbiotic branching model
${\operatorname
{cSBM}(\varrho,\gamma)}_{u_0, v_0}$, we know that for every $\gamma>0$,
$(\mu_t^\sse{\gamma},\nu_t^\sse{\gamma})_{t\ge0}$ is a solution
to the
martingale problem $(\mathbf{MP})_{\mu_0,\nu_0}^\rho$, with the
covariation structure in \eqref
{Cov1} given by the measure $\Lambda^\sse{\gamma}$ from \eqref
{defn:measure-valued processes4}. Thus it comes as no surprise that the
limit point $(\mu,\nu)$ of $(\mu^\sse{\gamma},\nu^\sse{\gamma})$
satisfies the same martingale problem,
with the covariation
now controlled by the limit point $\Lambda$ of $\Lambda^\sse{\gamma}$:

\begin{prop}\label{MPinf 1}
Let $\rho<0$ and $( u_0, v_0)\in(\calB_\tem^+)^2$ [resp., $(\calB
_\rap^+)^2$].
If $(\mu_t,\nu_t,\Lambda_t)_{t\ge0}\in\calC_{[0,\infty)}(\calM
_{\tem
}^3)$ [resp., $\calC_{[0,\infty)}(\calM_{\rap}^3)$] is a limit point
with respect to the Skorokhod topology\vspace*{1pt} of the family $(\mu^\sse
{\gamma
}_t,\nu^\sse{\gamma}_t,\Lambda^\sse{\gamma}_t)_{t\ge0}$, $\gamma>0$,
then $(\mu_t,\nu_t)_{t\ge0}$ satisfies the martingale problem
$(\mathbf
{MP})_{ u_0, v_0}^\rho$ with the covariation structure in \eqref{Cov1}
being given by the process $(\Lambda_t)_{t\ge0}$.
\end{prop}

\begin{pf}
We give the proof for $( u_0, v_0)\in(\calB_\tem^+)^2$, the proof for
initial conditions in $\calB_\rap^+$ being completely analogous.

Consider a sequence $\gamma_k\uparrow\infty$ such that
\[
\bigl(\mu^\sse{\gamma_k}_t,
\nu_t^\sse{\gamma_k},\Lambda_t^\sse
{\gamma _k}\bigr)_{t\ge0}\mathop{\xrightarrow}_{k\to\infty}^{\calL}(
\mu_t,\nu_t,\Lambda _t)_{t\ge0}
\]
in $\calC_{[0,\infty)}(\calM_{\tem}^3)$. Let $\phi\in\calC_{\rap
}^{(2)}$.
Then we have also
\[
\bigl(M_t^\sse{\gamma_k}(
\phi),N_t^\sse{\gamma_k}(\phi),\Lambda
_t^\sse {\gamma_k}\bigl(\phi^2
\bigr) \bigr)_{t\ge0}\mathop{\xrightarrow}_{k\to\infty}^{\calL }
\bigl(M_t(\phi),N_t(\phi),\Lambda_t\bigl(
\phi^2\bigr) \bigr)_{t\ge0}
\]
in $\calC_{[0,\infty)}(\R^3)$, where $(M^\sse{\gamma_k}(\phi
),N^\sse
{\gamma_k}(\phi))$ and $(M(\phi),N(\phi))$ denote the pairs of
processes from \eqref{MP1} corresponding to $(\mu^\sse{\gamma
_k},\nu
^\sse{\gamma_k})$ and $(\mu,\nu)$, respectively. We already know that
$M^\sse{\gamma_k}(\phi)$ and $N^\sse{\gamma_k}(\phi)$ are martingales.
In order for the weak limit $(M(\phi)$, $N(\phi))$ to be again a
martingale, it suffices to show that $(M_t^\sse{\gamma_k}(\phi
))_{k\in\N
}$ and $(N_t^\sse{\gamma_k}(\phi))_{k\in\N}$ are uniformly integrable
for every fixed $t$; see, e.g., \cite{MZ84}, Theorem~11).
Using the Burkholder--Davis--Gundy and Jensen inequalities as well as
Lemma~\ref{lemma uniform first moments dual process}, we obtain for
every $1<p\le2$ that
\begin{eqnarray*}
\sup_{k\in\N}\E \bigl[\bigl|M_t^\sse{
\gamma_k}(\phi)\bigr|^p \bigr]&\le& C_p \sup
_{k\in\N}\E \bigl[ \bigl(\bigl[ M^\sse{
\gamma_k}(\phi), M^\sse{\gamma _k}(\phi )
\bigr]_t \bigr)^{p/2} \bigr]
\\
&\le& C_p \sup_{k\in\N} \bigl(\E \bigl[\bigl[
M^\sse{\gamma_k}(\phi), M^\sse {
\gamma_k}(\phi)\bigr]_t \bigr] \bigr)^{p/2}
\\
&=& C_p \sup_{k\in\N} \bigl(\E \bigl[ \bigl\langle
\Lambda^\sse{\gamma _k}_t,\phi ^2
\bigr\rangle \bigr] \bigr)^{p/2} <\infty.
\end{eqnarray*}
An analogous assertion holds for $N^\sse{\gamma_k}(\phi)$. Hence the
weak limit $(M(\phi),N(\phi))$ is again a martingale.

The quadratic (co)variation converges along with the sequence of
martingales to the quadratic (co)variation of the limit martingales
(see, e.g., \cite{MZ84}, Theorem~12).
Thus identity \eqref{Cov1} on the covariation structure of the limit
martingales follows directly from the corresponding identity for the
finite rate model, which completes our proof.
\end{pf}

The fact that limit points w.r.t. the Skorokhod topology satisfy the
martingale problem $(\mathbf{MP})_{\mu_0,\nu_0}^\rho$ has important
consequences: Namely, they also
satisfy the (weaker) martingale problem $(\mathbf{MP'})_{\mu_0,\nu
_0}^\rho$,
which will be of crucial importance in the uniqueness proof in
Section~\ref{ssn:martconv} below. Also, they admit a similar Green function
representation as for the finite rate symbiotic branching model.
Since these properties are true of \emph{any} solution to the
martingale problem $(\mathbf{MP})_{\mu_0,\nu_0}^\rho$, not just
limit points of our family of
processes, and since the methods to prove them are standard, we have
decided put the corresponding proofs into Appendix~\ref{appendix1}; see
Lemma~\ref{extended martingale problem} and Corollary~\ref{cor MPinf}.
At this point, we only prove the absolute continuity of the limit
measures which is part of our main result Theorem~\ref{thmm:main2}.
This is in fact also true for any solution to the martingale problem
$(\mathbf{MP})_{\mu_0,\nu_0}^\rho$ and is a simple consequence of a
general criterion for absolute
continuity due to \cite{Dawsonetal2002a}:

\begin{prop}[(Absolute continuity)]\label{prop:AC}
Let $\varrho\in(-1,0]$, and
suppose $(\mu_t,\nu_t)_{t\ge0}$
is any solution to the martingale problem $(\mathbf{MP})_{\mu_0,\nu
_0}^\rho$. Then for each fixed
$t>0$, the measures $\mu_t$ and $\nu_t$ are absolutely continuous
w.r.t. Lebesgue measure, $\p_{\mu_0,\nu_0}$-a.s.
\end{prop}
\begin{pf}
Fix $T>0$.
Using the same transformation as in \cite{DM12}, page 24, we define
%
\begin{equation}
\label{transformation}\tilde\mu_t:=\mu_t,\qquad \tilde \nu
_t:=\frac{1}{\sqrt{1-\rho^2}}(\nu_t-\rho\mu_t).
\end{equation}
Then $(\tilde\mu_t,\tilde\nu_t)_{t\in[0,T]}$ is a continuous
$\calM^2$-valued process, where $\calM$ denotes the space of Radon
measures on $\R$ (note that $\rho\le0$). Using the Green function
representation of $(\mathbf{MP})_{\mu_0,\nu_0}^\rho$ from Corollary~\ref{cor:Green function
representation}, it is easily checked that for all nonnegative test
functions $0\le\phi\in\calC_c^\infty$, the processes
\[
\tilde M(\phi)_t:=\langle\tilde\mu_t,S_{T-t}
\phi\rangle,\qquad \tilde N(\phi )_t:=\langle\tilde\nu_t,S_{T-t}
\phi\rangle,\qquad t\in[0,T]
\]
are martingales
with covariance structure
\[
\bigl[\tilde M(\phi), \tilde M(\phi)\bigr]_t = \bigl[ \tilde N(
\phi), \tilde N(\phi )\bigr]_t,\qquad \bigl[\tilde M(\phi), \tilde N(\phi)
\bigr]_t=0,\qquad t\in[0,T].
\]
Applying Theorem~57 in \cite{Dawsonetal2002a}, we get a.s. absolute
continuity of $\tilde\mu_T$ and $\tilde\nu_T$. Thus the same holds for
$\mu_T=\tilde\mu_T$ and $\nu_T=\rho\tilde\mu_T+\sqrt{1-\rho
^2}\tilde
\nu_T$.
\end{pf}

We remind the reader of our convention to use the same symbol for an
absolutely continuous measure and its density. Thus if $(\mu_t,\nu
_t)_{t\ge0}$ is any limit point of the family \eqref
{defn:measure-valued processes3},
we will write
\[
\mu_t(dx)=\mu_t(x) \,dx,\qquad \nu_t(dx)=
\nu_t(x) \,dx.
\]
Note, however, that although $\mu_t$ and $\nu_t$ are (as \emph
{measures}) elements of the space $\calM_{\tem}$ respectively $\calM
_\rap$, their \emph{densities} have no reason to be elements of the
function space $\calB_{\tem}$, respectively $\calB_\rap$, let alone
$\calC_\tem$, respectively $\calC_\rap$, as is the case for solutions
to the finite rate symbiotic branching model ${\operatorname{cSBM}(\varrho
,\gamma)}_{u_0, v_0}$.

We now turn to limit points with respect to the Meyer--Zheng topology.
It would be nice to prove that every Meyer--Zheng limit point satisfies
also the martingale problem $(\mathbf{MP})^\rho_{ u_0, v_0}$,
but unfortunately we have been unable to show an analogue of
Proposition~\ref{MPinf 1} for the Meyer--Zheng topology.\footnote{As a
consequence, we also cannot show a Green function representation or
absolute continuity of the limit measures $\mu_t$ and $\nu_t$ for fixed
$t$.} The reason is the following: While in that case we can still
apply \cite{MZ84}, Theorem~11, in order to show that the weak limit of
the approximating martingales is again a martingale,
it is no longer clear that the Meyer--Zheng limit $\Lambda$ of the
quadratic variation processes $\Lambda^\sse{\gamma}$ coincides with the
quadratic variation of the limit martingales. (In order to apply \cite{MZ84},
Theorem~12, we would have to know that $\Lambda$ is continuous.)

However, we can still prove that any Meyer--Zheng limit point of the
family \eqref{defn:measure-valued processes3}--\eqref
{defn:measure-valued processes4} satisfies the weaker martingale
problem $(\mathbf{MP'})_{\tilde u_0,\tilde v_0}^\rho$ of Definition~\ref
{defn:MP'}:

\begin{prop}\label{MPinf 2}
Let $\rho<0$ and $( u_0, v_0)\in(\calB_\rap^+)^2$ [resp., $(\calB
_\tem^+)^2$].
If $(\mu_t,\nu_t,\Lambda_t)_{t\ge0}\in D_{[0,\infty)}(\calM_{\rap}^3)$
[resp., $D_{[0,\infty)}(\calM_{\tem}^3)$] is any limit point with
respect to the Meyer--Zheng topology of the family $(\mu^\sse{\gamma
}_t,\nu^\sse{\gamma}_t,\Lambda^\sse{\gamma}_t)_{t\ge0}$, $\gamma>0$,
then $(\mu_t,\nu_t)_{t\ge0}$ solves the martingale problem $(\mathbf
{MP'})_{ u_0, v_0}^\rho$, with the process $(\Lambda_t)_{t\ge0}$
satisfying the requirements of Definition~\ref{defn:MP'}.
\end{prop}
\begin{pf}
We give the proof for $( u_0, v_0)\in(\calB_\rap^+)^2$.

First, we show that the limit point $(\Lambda_t)_{t\ge0}$ of the family
$(\Lambda_t^\sse{\gamma})_{t\ge0}$ has the properties required in
Definition~\ref{defn:MP'}.
It is clear that $(\Lambda_t)_{t\ge0}$ is increasing with $\Lambda
_0=0$. We check condition \eqref{finiteness Lambda 1}: By \cite
{MZ84}, Theorem~5 (see also \cite{Kurtz91}, Theorem~1.1(b)), we can find a
sequence $\gamma_k\uparrow\infty$ and a set $I\subseteq(0,\infty)$ of
full Lebesgue measure such that
the finite dimensional distributions of $(\Lambda_t^\sse{\gamma
_k})_{t\in I}$ converge weakly to those of $(\Lambda_t)_{t\in I}$ as
$k\to\infty$.
Fix $t\in I$. Then for all test functions $\phi\in\bigcup_{\lambda
>0}\calC_{-\lambda}^+$,
by estimate~\eqref{uniform first moments dual process 2} in Lemma~\ref
{lemma uniform first moments dual process} and Fatou's lemma, we have
\[
\E_{ u_0, v_0}\bigl[\langle\Lambda_t,\phi\rangle\bigr] \le
\liminf_{k\to\infty}\E_{ u_0, v_0}\bigl[\bigl\langle
\Lambda^\sse {\gamma _k}_t,\phi\bigr\rangle
\bigr] <\infty.
\]
Now use right-continuity and monotonicity of $(\Lambda_t)_{t\ge0}$ and
another application of Fatou's lemma to extend this to all $t>0$. This
shows that\break $\E_{u_0,v_0}[\Lambda_t(dx)]\in\calM_\rap$ for all $t>0$,
that is, \eqref{finiteness Lambda 1}.

It remains to check that for all
test functions $\phi,\psi\in (\calC_{\tem}^{(2)} )^+$,
the process
%
\begin{eqnarray}
\label{MP11} \tilde M_t & :=& F(\mu_t,
\nu_t,\phi,\psi) - F(\mu_0,\nu_0,\phi ,\psi)\nonumber
\\
&&{} - \frac{1}{2}\int_0^t F(
\mu_s, \nu_s,\phi,\psi) \langle\!\langle
\mu_s, \nu_s, \Delta\phi, \Delta\psi\rangle
\!\rangle_\rho \,ds
\\
&&{} - 4\bigl(1-\rho^2\bigr)\int_{[0,t]\times\R} F(
\mu_s,\nu_s,\phi,\psi) \phi(x)\psi (x) \Lambda(ds,dx),\qquad t
\ge0\nonumber
\end{eqnarray}
is a martingale. Denote by $\tilde M_t^\sse{\gamma}$ the same
expression but with $(\mu,\nu,\Lambda)$ replaced by $(\mu^\sse
{\gamma
},\nu^\sse{\gamma},\Lambda^\sse{\gamma})$. Choosing a sequence
$\gamma
_k\uparrow\infty$ such that $( \mu_t^\sse{\gamma_k}, \nu_t^\sse
{\gamma
_k},\break  \Lambda_t^\sse{\gamma_k})_{t\ge0}$ converges to $( \mu_t, \nu_t,
\Lambda_t)_{t\ge0}$ w.r.t. the Meyer--Zheng topology on $D_{[0,\infty
)}(\calM_\rap^3)$, we get that $(\tilde M_t^\sse{\gamma_k})_{t\ge0}$
converges to $(\tilde M_t)_{t\ge0}$ w.r.t. the Meyer--Zheng topology
on $D_{[0,\infty)}(\R)$ as $k\to\infty$.
Moreover, by Corollary~\ref{cor MPinf} we know that $\tilde M^\sse
{\gamma}$ are martingales for each $\gamma>0$ with quadratic variation
%
\begin{equation}
8\bigl(1-\rho^2\bigr)\int_{[0,t]\times\R}F\bigl(
\mu_s^\sse{\gamma},\nu_s^\sse {
\gamma },\phi,\psi\bigr)^2 \phi(x) \psi(x) \Lambda^\sse{
\gamma}(ds,dx).
\end{equation}
Consequently, using the Burkholder--Davis--Gundy inequality and the
fact that $|F(\cdot)|\le1$,
we have
\begin{eqnarray*}
\E_{ u_0, v_0} \bigl[\bigl|\tilde M_t^\sse{
\gamma}\bigr|^2 \bigr]&\le&\E_{ u_0,
v_0} \bigl[\bigl[\tilde
M^\sse{\gamma},\tilde M^\sse{\gamma}\bigr]_t
\bigr]
\\
&\le&8\bigl(1-\rho^2\bigr) \E_{ u_0, v_0} \biggl[\int
_{[0,t]\times\R}\phi (x)\psi(x) \Lambda^\sse{\gamma}(ds,dx)
\biggr]
\\
&=&8\bigl(1-\rho^2\bigr) \E_{ u_0, v_0} \bigl[\bigl\langle
\Lambda_t^\sse{\gamma },\phi \psi\bigr\rangle \bigr].
\end{eqnarray*}
By estimate \eqref{uniform first moments dual process 2} in Lemma~\ref
{lemma uniform first moments dual process}, for each $T>0$ the previous
display is bounded
uniformly in $\gamma>0$ and $t\in[0,T]$.
Hence we get
\[
\sup_{\gamma>0}\sup_{t\in[0,T]}\E_{ u_0, v_0}
\bigl[\bigl|\tilde M_t^\sse {\gamma}\bigr|^2 \bigr]<
\infty
\]
for all $T>0$.
Applying \cite{MZ84}, Theorem~11, we infer that the Meyer--Zheng limit
$\tilde M$ is again a martingale, which completes our argument.
\end{pf}

We now turn to proving the ``separation of types'' property, that is,
the fact that for all limit points the measures $\mu_t$ and $\nu_t$ are
mutually singular for each $t>0$. In fact, we will prove a slightly
stronger assertion, namely \eqref{singularity3} below. Its proof relies
on the colored particle moment duality of Lemma~\ref{la:mdual}
applied to mixed second moments of $(\mu_t^\sse{\gamma},\nu_t^\sse
{\gamma})$.

\begin{lemma}[(Separation of types)]\label{lemma:sep}
Let $\rho<0$ and $( u_0, v_0)\in(\calB^+_\rap)^2$ [resp., $( u_0,
v_0)\in(\calB^+_\tem)^2$].
Suppose that
$(\mu_t,\nu_t)_{t\ge0}\in D_{[0,\infty)}(\calM_{\rap}^2)$ [resp.,\break 
$D_{[0,\infty)}(\calM_{\tem}^2)$] is a limit point with respect to the
Meyer--Zheng topology of the family of measure-valued processes
$(\mu_t^\sse{\gamma},\nu_t^\sse{\gamma})_{t\ge0}$ from \eqref
{defn:measure-valued processes3}.
Then for
each $t>0$, $x\in\R$ and $\eps>0$ we have
%
\begin{equation}
\label{singularity3} S_{t+\eps} u_0(x) S_{t+\eps}
v_0(x)\ge\E_{ u_0, v_0} \bigl[S_\eps \mu
_t(x) S_\eps\nu_t(x) \bigr]\mathop{\xrightarrow}^{\eps
\downarrow0}0.
\end{equation}
\end{lemma}

\begin{pf}
We give the proof for $( u_0, v_0)\in(\calB_\rap^+)^2$, the proof for
initial conditions in $\calB_\tem^+$ being completely analogous. Note
that in either case, the left-hand side of \eqref{singularity3} is finite by Lemma~\ref{lemma estimates}(a).

Again using \cite{MZ84}, Theorem~5, choose a sequence $\gamma
_k\uparrow
\infty$ and a set $I\subseteq(0,\infty)$ of full Lebesgue measure
such that
the finite dimensional distributions of $(\mu_t^\sse{\gamma_k},\nu
_t^\sse{\gamma_k})_{t\in I}$ converge weakly to those of $(\mu_t,\nu
_t)_{t\in I}$ as $k\to\infty$.
Fix $t\in I$. Then for all test functions $\phi,\psi$
we have weak convergence
%
\begin{equation}
\label{weak convergence second mixed moments} \bigl\langle\mu_t^\sse{
\gamma_k},\phi\bigr\rangle \bigl\langle\nu_t^\sse
{\gamma _k},\psi\bigr\rangle\mathop{\xrightarrow}^{k\uparrow\infty}\langle
\mu_t,\phi \rangle \langle\nu_t,\psi\rangle
\end{equation}
in $\R$.
Thus for each $x\in\R$, letting $\phi(\cdot):=\psi(\cdot):=p_\eps
(x-\cdot)$
we obtain weak convergence
\[
S_\eps\mu_t^\sse{\gamma_k}(x)
S_\eps\nu_t^\sse{\gamma _k}(x)
\mathop{\xrightarrow}^{k\uparrow\infty}S_\eps\mu_t(x)
S_\eps\nu_t(x).
\]
Using Fatou's lemma and the colored particle moment duality in form
\eqref{eq:dual second moments} for mixed second moments, we get since
$\rho<0$,
%
\begin{eqnarray}
\label{proof separation 1a} &&\E_{ u_0, v_0} \bigl[S_\eps\mu_t(x)
S_\eps\nu_t(x) \bigr]\nonumber
\\
&&\qquad\le\liminf_{k\to\infty}\E_{ u_0, v_0} \bigl[S_\eps
\mu_t^\sse {\gamma _k}(x) S_\eps
\nu_t^\sse{\gamma_k}(x) \bigr]\nonumber
\\
&&\qquad=\liminf_{k\to\infty}\iint \,dy\,dz p_\eps(x-y)p_\eps(x-z)
\E_{ u_0,
v_0} \bigl[u_t^\sse{\gamma_k}(y)
v^\sse{\gamma_k}_t(z) \bigr]
\\
&&\qquad=\liminf_{k\to\infty}\iint \,dy\,dz p_\eps(x-y)p_\eps(x-z)
\E _{y,z} \bigl[ u_0\bigl(B_t^\ssup{1}
\bigr) v_0\bigl(B_t^\ssup{2}\bigr)
e^{\gamma_k\rho L^{1,2}_t} \bigr]\nonumber
\\
&&\qquad=\iint \,dy\,dz p_\eps(x-y)p_\eps(x-z) \E_{y,z}
\bigl[ u_0\bigl(B_t^\ssup{1}\bigr)
v_0\bigl(B_t^\ssup{2}\bigr) \1_{\{L^{1,2}_t=0\}}\nonumber
\bigr],
\end{eqnarray}
for all $x\in\R$ and $t\in I$, where $(B_t^\ssup{i})_{t\ge0}$, $i=1,2$
are independent Brownian motions started at $y$ and $z$, respectively,
and $(L_t^{1,2})_{t\ge0}$ denotes their intersection local time. It is
easy to see that the right-hand side of \eqref{proof separation 1a} is continuous
in $t$. Using the fact that $I$ has full Lebesgue measure together with
right-continuity of the paths of $(\mu_t,\nu_t)_{t\ge0}$ and Fatou's
lemma, we get estimate \eqref{proof separation 1a} for \emph{all}
$t>0$. This implies in particular that
\[
\E_{ u_0, v_0} \bigl[S_\eps\mu_t(x) S_\eps
\nu_t(x) \bigr]\le S_{t+\eps} u_0(x)
S_{t+\eps} v_0(x)<\infty,\qquad x\in\R, t>0.
\]
Moreover,
using H\"older's inequality we have
\begin{eqnarray*}
&&\E_{y,z} \bigl[ u_0\bigl(B_t^\ssup{1}
\bigr)  v_0\bigl(B_t^\ssup{2}\bigr)
\1_{\{
L^{1,2}_t=0\}
} \bigr]
\\
&&\qquad\le \bigl(\E_{y,z} \bigl[ \bigl( u_0\bigl(B_t^\ssup{1}
\bigr) v_0\bigl(B_t^\ssup {2}\bigr)
\bigr)^2 \bigr] \bigr)^{1/2} \bigl(\p_{y,z} \bigl\{
L^{1,2}_t=0 \bigr\} \bigr)^{1/2}
\\
&&\qquad= \bigl(S_t u_0^2(y) S_t
v_0^2(z) \bigr)^{1/2} \bigl(\p_{y,z}
\bigl\{ L^{1,2}_t=0 \bigr\} \bigr)^{1/2}.
\end{eqnarray*}
Observe that the right-hand side of the previous display tends to $0$ as $(y,z)\to
(x,x)$: Assume without loss of generality that $y<z$, and let $B$ be a
Brownian motion
starting at $y-z<0$
with local time $L^0$ at $0$. Using the fact that $L_t^{1,2}\stackrel
{d}{=}\frac{1}{2}L_{2t}^0$
together with Lemma~\ref{le:localtime} and the reflection principle
(see, e.g., \cite{MP10}, Theorem~2.21), we obtain for $(y,z)\to(x,x)$ that
\begin{eqnarray*}
\p_{y,z} \bigl\{L^{1,2}_t=0 \bigr\}&=&
\p_{y-z} \bigl\{L^{0}_{2t}=0 \bigr\} =\p
_{y-z} \bigl\{M^{+}_{2t}=0 \bigr\}=
\p_{y-z} \{M_{2t}\le0 \}
\\
& =&\p _{0} \{M_{2t}\le z-y \} =1-2\p_{0}
\{B_{2t}>z-y \}
\\
& \to& 1-2\p_{0} \{ B_{2t}>0 \} = 0.
\end{eqnarray*}
Since on the other hand clearly $\E_{x,x} [ u_0(B_t^\ssup{1})
v_0(B_t^\ssup{2}) \1_{\{L^{1,2}_t=0\}} ]=0$ for $t>0$, this shows
that the mapping
\[
(y,z)\mapsto\E_{y,z} \bigl[ u_0\bigl(B_t^\ssup{1}
\bigr) v_0\bigl(B_t^\ssup{2}\bigr) \1
_{\{
L^{1,2}_t=0\}} \bigr]
\]
is continuous
at all points $(x,x)$ of the diagonal in $\R^2$, where it takes the
value $0$.
As a consequence, the right-hand side of \eqref{proof separation 1a} converges to
$0$ as $\eps\downarrow0$, giving
\eqref{singularity3} for all $t>0$ and $\in\R$.
\end{pf}

Of course, Lemma~\ref{lemma:sep} holds in particular also for limit
points in the stronger Skorokhod topology. In this case,
together with absolute continuity of the limiting measures from
Proposition~\ref{prop:AC}, it implies the separation of types in the
intuitive sense \eqref{singularity1}, that is, the mutual singularity
of the limiting measures $(\mu_t,\nu_t)$ for fixed $t>0$:

\begin{corollary}[(Separation of types)]\label{cor:singularity}
Let $\rho<0$ and $( u_0, v_0)\in(\calB^+_\rap)^2$ [resp., $( u_0,
v_0)\in(\calB^+_\tem)^2$]. If $(\mu_t,\nu_t)_{t\ge0}\in\calC
_{[0,\infty
)}(\calM_{\rap}^2)$ [resp., $\calC_{[0,\infty)}(\calM_{\tem}^2)$]
is a
limit point with respect to the Skorokhod topology of the family \eqref
{defn:measure-valued processes3}, then
for each $t>0$ the measures $\mu_t$ and $\nu_t$ (which are known to be
absolutely continuous by Proposition~\ref{prop:AC}) are mutually
singular: We have
%
\begin{equation}
\label{singularity2} \E_{u_0,v_0} \biggl[\int_\R
\mu_t(x)\nu_t(x) \,dx \biggr]=0
\end{equation}
and thus also
\[
\mu_t(\cdot) \nu_t(\cdot)=0,\qquad \p_{u_0,v_0}\otimes
\ell\mbox{-a.s.}
\]
\end{corollary}
\begin{pf}
Let $0\le\varphi\in\calC_c^\infty$. By differentiation theory for
measures (see, e.g., \cite{Rudin74}, Theorem~8.6), we have $\p_{u_0,v_0}$-a.s.
\[
\bigl(S_\eps\mu_t(x),S_\eps
\nu_t(x)\bigr)\mathop{\xrightarrow}^{\eps\downarrow0}\bigl(\mu
_t(x),\nu_t(x)\bigr) \qquad\mbox{for }\ell\mbox{-a.e. }x\in\R.
\]
Using again Fatou's lemma and Fubini's theorem, we get
%
\begin{eqnarray}
&&\E_{u_0,v_0} \biggl[\int_\R \mu_t(x)
\nu_t(x)\varphi(x) \,dx \biggr]\nonumber
\\
&&\qquad=\E_{u_0,v_0} \biggl[\int_\R\lim
_{\eps\downarrow0} S_\eps\mu_t(x) S_\eps
\nu_t(x) \varphi(x) \,dx \biggr]
\\
&&\qquad\le\liminf_{\eps\downarrow0}\int_\R
\E_{u_0,v_0} \bigl[S_\eps \mu_t(x) S_\eps
\nu_t(x) \bigr] \varphi(x)\,dx.\nonumber
\end{eqnarray}
By Lemma~\ref{lemma:sep}
the integrand in the integral $\int_\R\cdots dx$ on the right-hand side of the
previous display converges to $0$ as $\eps\downarrow0$ pointwise in
$x\in\R$ and for $\eps\in[0,1]$ is dominated by the integrable function
\[
\varphi(x)\sup_{s\in[0,t+1]} \bigl\{ S_{s}u_0(x)
S_{s} v_0(x) \bigr\}
\]
[note Lemma~\ref{lemma estimates}(a)].
By dominated convergence,
$\E [\int_\R\mu_t(x)\nu_t(x)\varphi(x) \,dx ]=0$, and since
$0\le\varphi\in\calC_c^\infty$ was arbitrary, our proof is complete.
\end{pf}

\section{Self-duality and uniqueness}
\label{ssn:martconv}

In this section, we establish uniqueness for the martingale problem
$(\mathbf{MP'})_{\mu_0,\nu_0}^\rho$ [and thus also for the stronger
martingale problem $(\mathbf{MP})_{\mu_0,\nu_0}^\rho$] subject to the
restriction that the solutions have the ``separation of types'' property
\eqref{singularity3}.
Recall from the \hyperref[sec1]{Introduction} that
these martingale problems are not well posed without putting some
restrictions on the solutions, and that
for the finite rate symbiotic branching model $\operatorname{cSBM}(\varrho
,\gamma)_{u_0, v_0}$ uniqueness
is established by prescribing the structure of the quadratic variation
process $(\Lambda)_{t\ge0}$. In~\cite{EF04}, Proposition~5, this is proved
via an exponential self-duality. Our first goal in this section is to
extend this self-duality to solutions of the martingale problem
$(\mathbf{MP'})_{\mu_0,\nu_0}^\rho$
satisfying the said condition,
circumventing an explicit specification of the quadratic variation.
We have the following result:

\begin{prop}\label{prop:self-dual}
Let $\rho\in(-1,1)$. Fix
(possibly random) initial conditions $(\mu_0,\nu_0)\in\calM_{\tem}^2$
and (deterministic) initial conditions $(\tilde\mu_0,\tilde\nu
_0)\in
(\calB_{\rap}^+)^2$.
Suppose that $(\mu_t,\nu_t,)_{t\ge0}\in D_{[0,\infty)}(\calM_\tem^2)$
respectively $(\tilde\mu_t,\tilde\nu_t)_{t\ge0}\in\break  D_{[0,\infty
)}(\calM
_\rap^2)$
are solutions to the martingale problem $(\mathbf{MP'})^\rho_{\mu
_0,\nu
_0}$ respectively $(\mathbf{MP'})^\rho_{\tilde\mu_0,\tilde\nu_0}$.
Further, assume that the solutions satisfy the ``separation of types''
property in the sense
that for
Lebesgue-a.e. $t\in(0,\infty)$ and
all $x\in\R$,
we have
%
\begin{equation}
\label{singularity4} \E_{\mu_0,\nu_0} \bigl[S_\eps\mu_t(x)
S_\eps\nu_t(x) \bigr]\mathop{\xrightarrow}^ {\eps\downarrow0}0
\quad\mbox{and}\quad \E_{\tilde\mu_0,\tilde\nu_0} \bigl[S_\eps\tilde\mu_t(x)
S_\eps \tilde\nu _t(x) \bigr]\mathop{\xrightarrow}^{\eps
\downarrow0}0.
\end{equation}
Moreover, assume that for each $T>0$ we have
%
\begin{eqnarray}
\label{upper bound singularity} \sup_{t\in[0,T], \eps\in[0,1]}\E_{\mu_0,\nu_0}
\bigl[S_\eps\mu _t(\cdot) S_\eps
\nu_t(\cdot) \bigr]&\in&\calB_\tem^+,
\nonumber
\\[-8pt]
\\[-8pt]
\nonumber
\sup_{t\in[0,T], \eps\in[0,1]}\E_{\tilde\mu_0,\tilde\nu
_0} \bigl[S_\eps \tilde
\mu_t(\cdot) S_\eps\tilde\nu_t(\cdot) \bigr]&\in&
\calB _\rap^+.
\end{eqnarray}
Then
the following approximate
self-duality holds for the processes $(\mu_t,\nu_t)_{t\ge0}$ and
$(\tilde\mu_t,\tilde\nu_t)_{t\ge0}$, involving the function $F$ as
in~\eqref{self-duality function}: for $T > 0$,
%
\begin{equation}
\label{self-duality} \int_0^T\E \bigl[F(
\mu_t,\nu_t,\tilde\mu_0,\tilde
\nu_0) \bigr]\,dt =\lim_{\eps\downarrow0}\int
_0^T \E \bigl[F(S_\eps
\mu_0,S_\eps\nu _0,\tilde\mu _t,
\tilde\nu_t) \bigr]\,dt.
\end{equation}
Moreover, for $(\mu_0,\nu_0)\in(\calB_{\tem}^+)^2$, we have the
self-duality
%
\begin{equation}
\label{self-duality 2} \E_{\mu_0,\nu_0} \bigl[F(\mu_t,\nu_t,
\tilde\mu_0,\tilde\nu _0) \bigr]=\E_{\tilde\mu_0,\tilde\nu_0}
\bigl[F(\mu_0,\nu_0,\tilde\mu _t,\tilde
\nu_t) \bigr],\qquad t\ge0.
\end{equation}
\end{prop}

The general strategy of the proof is similar to that of the results in
\cite{EK86}, Section~4.4; however none of those results is directly
applicable in our case. Also, we employ the same spatial smoothing
procedure using the heat kernel as in the proof of \cite{AT00},
Proposition~1.
\begin{pf*}{Proof of Proposition \ref{prop:self-dual}}
By Definition~\ref{defn:MP'}, there exist increasing processes
$(\Lambda
_t)_{t\ge0}\in D_{[0,\infty)}(\calM_\tem)$ and $(\tilde\Lambda
_t)_{t\ge
0}\in D_{[0,\infty)}(\calM_\rap)$, with $\Lambda_0=\tilde\Lambda_0=0$
and satisfying \eqref{finiteness Lambda 1}, such that for all test
functions, expression \eqref{MP9} is a martingale. For the purposes of
the proof, we may assume that $(\mu,\nu,\Lambda)$ and $(\tilde\mu
,\tilde
\nu,\tilde\Lambda)$ are defined on
a common sample space $\Omega$ and are independent of each other.
The corresponding probability, respectively expectation on $\Omega$,
will be denoted by $\p$, respectively $\E$.

Observe that by the definition of $\langle\!\langle\cdot\rangle\!
\rangle
_\rho$ [recall \eqref{self-duality product}] and the symmetry of the
heat kernel, we have for each $\eps>0$ and $\phi,\psi\in (\calC
_\rap
^{(2)} )^+$,
\begin{eqnarray*}
\langle\!\langle S_\eps\mu_t,S_\eps
\nu_t,\phi,\psi\rangle\!\rangle _\rho& =&\langle\!
\langle\mu_t,\nu_t,S_\eps\phi,S_\eps
\psi\rangle\!\rangle _\rho,
\\
\langle\!\langle S_\eps\mu_t,S_\eps
\nu_t,\Delta\phi,\Delta\psi \rangle \!\rangle_\rho&=&
\langle\!\langle\mu_t,\nu_t,\Delta S_\eps\phi,
\Delta S_\eps \psi\rangle\!\rangle_\rho.
\end{eqnarray*}
Thus by taking expectations in \eqref{MP9} with $(S_\eps\phi,S_\eps
\psi
)$ in place of $(\phi,\psi)$, we get
%
\begin{eqnarray}
\label{MP10} && \E \bigl[ F({S_\eps\mu_t},
{S_\eps\nu_t},\phi,\psi) - F({S_\eps \mu
_0},{S_\eps\nu_0},\phi,\psi) \bigr]\nonumber
\\
&&\qquad = \frac{1}{2} \E \biggl[\int_0^t
F({S_\eps\mu_s}, {S_\eps\nu _s},
\phi ,\psi) \langle\!\langle{S_\eps\mu_s},
{S_\eps\nu_s}, \Delta\phi, \Delta \psi\rangle\!\rangle_\rho \,ds \biggr]
\\
&&\qquad\quad{} + 4\bigl(1-\rho^2\bigr) \E \biggl[\int_{[0,t]\times\R}
F({S_\eps\mu _s},{S_\eps \nu_s},
\phi,\psi) {S_\eps\phi}(x){S_\eps\psi}(x) \Lambda (ds,dx)
\biggr]\nonumber
\end{eqnarray}
for all $\eps>0$ and $\phi,\psi\in (\calC_\rap^{(2)} )^+$.
An analogous assertion holds for $(\tilde\mu,\tilde\nu,\tilde
\Lambda)$
if $\phi,\psi\in (\calC_\tem^{(2)} )^+$.

Now fix $T>0$, and for $t,s\in[0,T]$, $\eps>0$, let
\[
f_\eps(t,s):=\E \bigl[F({S_\eps\mu}_t,{S_\eps
\nu}_t,{S_\eps \tilde\mu }_s,{S_\eps
\tilde\nu}_s) \bigr].
\]
Observe that this function is well defined since $S_\eps\mu_t$ and
$S_\eps\nu_t$, respectively $S_\eps\tilde\mu_t$ and $S_\eps\tilde
\nu_t
$, are in $ (\calC_\tem^{(2)} )^+$, respectively $ (\calC
_\rap
^{(2)} )^+$; see, for example, Corollary~\ref{cor estimates}(b).
Then
\begin{eqnarray*}
&&\int_0^T \bigl(f_\eps(r,0)-f_\eps(0,r)
\bigr) \,dr
\\
&&\qquad=\int_0^T \bigl(f_\eps(T-r,r)-f_\eps(0,r)
\bigr) \,dr - \int_0^T \bigl(f_\eps
(r,T-r)-f_\eps(r,0) \bigr) \,dr
\\
&&\qquad=\int_0^T \bigl(\E \bigl[F({S_\eps
\mu}_{T-r},{S_\eps\nu }_{T-r},{S_\eps
\tilde\mu}_r,{S_\eps\tilde\nu}_r)\\
&&\hspace*{75pt}{} -
F({S_\eps\mu}_{0},{S_\eps \nu
}_{0},{S_\eps\tilde\mu}_r,{S_\eps
\tilde\nu}_r) \bigr] \bigr) \,dr
\\
&&\qquad\quad{} - \int_0^T \bigl(\E \bigl[F({S_\eps
\mu}_r,{S_\eps\nu}_r,{S_\eps
\tilde\mu }_{T-r},{S_\eps\tilde\nu}_{T-r})\\
&&\hspace*{89pt}{} -
F({S_\eps\mu}_{r},{S_\eps\nu
}_{r},{S_\eps\tilde\mu}_0,{S_\eps
\tilde\nu}_0) \bigr] \bigr) \,dr.
\end{eqnarray*}
Now we use \eqref{MP10} [resp., the analogous identity for $(\tilde
\mu
,\tilde\nu,\tilde\Lambda)$] with $t$ replaced by $T-r$ for each
$r\in
[0,T]$ and $(\phi,\psi):=(S_\eps\tilde\mu_r,S_\eps\tilde\nu_r)$ [resp.,
$(\phi,\psi):=(S_\eps\mu_r,S_\eps\nu_r)$] to see that the previous
display is equal to\vspace*{-1pt}
\begin{eqnarray*}
&&\frac{1}{2}\int_0^T \E \biggl[\int
_0^{T-r} F({S_\eps\mu
}_{s},{S_\eps\nu }_{s},{S_\eps
\tilde\mu}_r,{S_\eps\tilde\nu}_r)
\langle\!\langle {S_\eps\mu }_s,{S_\eps
\nu}_s,\Delta{S_\eps\tilde\mu}_r,
\Delta{S_\eps\tilde \nu }_r\rangle\!\rangle_\rho \,ds \biggr] \,dr
\\[-1pt]
&&\qquad{} + 4\bigl(1-\rho^2\bigr)\int_0^T
\E \biggl[\int_{[0,T-r]\times\R}F({S_\eps \mu
}_{s},{S_\eps\nu}_{s},{S_\eps\tilde
\mu}_r,{S_{\eps}\tilde\nu}_r)\\
&&\hspace*{158pt}{}\times {S_{2\eps}\tilde\mu}_r(x){S_{2\eps}\tilde
\nu}_r(x) \Lambda (ds,dx) \biggr] \,dr
\\[-1pt]
&&\qquad{} - \frac{1}{2}\int_0^T \E \biggl[
\int_0^{T-r}F({S_\eps\mu
}_{r},{S_\eps\nu }_{r},{S_\eps
\tilde\mu}_s,{S_\eps\tilde\nu}_s)\\[-1pt]
&&\hspace*{99pt}{}\times
 \langle\!\langle\Delta{S_\eps\mu}_r,
\Delta{S_\eps\nu}_r,{S_\eps \tilde\mu
}_s,{S_\eps\tilde\nu}_s\rangle
\!\rangle_\rho \,ds \biggr] \,dr
\\[-1pt]
&&\qquad{} - 4\bigl(1-\rho^2\bigr)\int_0^T
\E \biggl[\int_{[0,T-r]\times\R}F({S_\eps \mu
}_{r},{S_\eps\nu}_{r},{S_\eps\tilde
\mu}_s,{S_\eps\tilde\nu}_s)\\[-1pt]
&&\hspace*{158pt}{}\times
 {S_{2\eps}\mu}_r(x){S_{2\eps}
\nu}_r(x) \tilde\Lambda(ds,dx) \biggr] \,dr.\vspace*{-1pt}
\end{eqnarray*}
Observe that due to symmetry of the Laplacian and Fubini's theorem, the
first and third term
of the last display cancel. Thus we have shown that\vspace*{-1pt}
\begin{eqnarray*}
&&\int_0^T \bigl(f_\eps(r,0)-f_\eps(0,r)
\bigr) \,dr
\\[-1pt]
&&\qquad = 4\bigl(1-\rho^2\bigr) \biggl(\int_0^T
\E \biggl[\int_{[0,T-r]\times\R
}F({S_\eps\mu
}_{s},{S_\eps\nu}_{s},{S_\eps\tilde
\mu}_r,{S_{\eps}\tilde\nu}_r)\\[-1pt]
&&\hspace*{160pt}{}\times {S_{2\eps}\tilde\mu}_r(x){S_{2\eps}\tilde
\nu}_r(x) \Lambda (ds,dx) \biggr] \,dr
\\[-1pt]
&&\quad\qquad{} - \int_0^T\E \biggl[\int
_{[0,T-r]\times\R}F({S_\eps\mu }_{r},{S_\eps
\nu }_{r},{S_\eps\tilde\mu}_s,{S_\eps
\tilde\nu}_s)\\
&&\hspace*{123pt}{}\times {S_{2\eps}\mu}_r(x){S_{2\eps}
\nu}_r(x) \tilde\Lambda(ds,dx) \biggr] \,dr \biggr).\vspace*{-1pt}
\end{eqnarray*}
We will show that each term in the difference on the right-hand side of the
previous display converges to $0$ as $\eps\downarrow0$. Consider the
first term: Since $|F(\cdot)|\le1$, it is bounded in absolute value up
to a constant by\vspace*{-1pt}
\begin{eqnarray*}
&&\int_0^T\E \biggl[\int_{[0,T-r]\times\R}{S_{2\eps}
\tilde\mu }_r(x){S_{2\eps}\tilde\nu}_r(x)
\Lambda(ds,dx) \biggr] \,dr
\\
&&\qquad=\int_0^T\E_{\mu_0,\nu_0} \biggl[\int
_{\R} \E_{\tilde\mu
_0,\tilde\nu
_0} \bigl[{S_{2\eps}\tilde
\mu}_r(x){S_{2\eps}\tilde\nu }_r(x) \bigr]
\Lambda_{T-r}(dx) \biggr] \,dr
\\
&&\qquad\le\int_0^T\E_{\mu_0,\nu_0} \biggl[\int
_{\R} \E_{\tilde\mu
_0,\tilde\nu
_0} \bigl[{S_{2\eps}\tilde
\mu}_r(x){S_{2\eps}\tilde\nu }_r(x) \bigr]
\Lambda_{T}(dx) \biggr] \,dr.
\end{eqnarray*}
By assumption \eqref{singularity4}, the integrand in the above display
converges to $0$
for all $x\in\R$ and almost all $r\in[0,T]$ as $\eps\downarrow0$.
Hence using conditions \eqref{finiteness Lambda 1} and \eqref{upper
bound singularity} together with dominated convergence, we are done.
The argument for the second term in the difference is completely
analogous. Thus in view of the definition of $f_\eps$, we have shown that
%
\begin{equation}
\label{approximate duality difference} \lim_{\eps\downarrow0}\int_0^T
\bigl(\E \bigl[F(\mu_t,\nu _t,S_\eps\tilde\mu
_0,S_\eps\tilde\nu_0) \bigr] - \E
\bigl[F(S_\eps\mu_0,S_\eps\nu _0,
\tilde \mu_t,\tilde\nu_t) \bigr] \bigr) \,dt=0.
\end{equation}

Since $\tilde\mu_0$ and $\tilde\nu_0$ are assumed to be in $\calB
_\rap
^+$, using estimate \eqref{estimate 1} in Lem\-ma~\ref{lemma estimates}(a) and dominated convergence, it is easy to see that
\[
\int_0^T\E \bigl[F(\mu_t,
\nu_t,S_\eps\tilde\mu_0,S_\eps
\tilde \nu_0) \bigr]\,dt \rightarrow\int_0^T
\E \bigl[F(\mu_t,\nu_t,\tilde\mu_0,\tilde \nu
_0) \bigr]\,dt,
\]
as $\eps\downarrow0$.
[Note that the same argument cannot in general be employed for the
second term in the difference in \eqref{approximate duality
difference}: Since $\mu_0$ and $\nu_0$ are only assumed to be in
$\calM
_\tem$ and not in $\calB_\tem^+$, we do not have \eqref{estimate 1} but
only the weaker estimate \eqref{estimate 2} in Lemma~\ref{lemma
estimates}(b), which, however, is not sufficient for dominated
convergence here.] Thus \eqref{self-duality} is proved.

If $(\mu_0,\nu_0)\in(\calB_\tem^+)^2$, we can again use estimate
\eqref
{estimate 1} and dominated convergence to conclude that also
\[
\int_0^T\E \bigl[F(S_\eps
\mu_0,S_\eps\nu_0,\tilde\mu_t,
\tilde \nu_t) \bigr] \,dt\to\int_0^T
\E \bigl[F(\mu_0,\nu_0,\tilde\mu_t,\tilde\nu
_t) \bigr] \,dt
\]
as $\eps\downarrow0$.
Thus in this case we get from \eqref{approximate duality difference} that
\begin{eqnarray*}
&&\int_0^T \bigl(\E \bigl[F(
\mu_t,\nu_t,\tilde\mu_0,\tilde\nu
_0) \bigr] - \E \bigl[F(\mu_0,\nu_0,\tilde
\mu_t,\tilde\nu_t) \bigr] \bigr) \,dt
\\
&&\qquad=\lim_{\eps\downarrow0}\int_0^T
\bigl(\E \bigl[F(\mu_t,\nu _t,S_\eps\tilde
\mu_0,S_\eps\tilde\nu_0) \bigr] - \E
\bigl[F(S_\eps\mu_0,S_\eps \nu _0,
\tilde\mu_t,\tilde\nu_t) \bigr] \bigr) \,dt = 0
\end{eqnarray*}
for each $T>0$. Since the processes $(\mu_t,\nu_t)_{t\ge0}$ and
$(\tilde
\mu_t,\tilde\nu_t)_{t\ge0}$ are assumed c\`adl\`ag, it is readily
checked that the same is true of the integrand in the last display.
Differentiating, we obtain the self-duality \eqref{self-duality 2} for
all $t\ge0$.
\end{pf*}

\begin{prop}[(Uniqueness)] \label{prop:unique}
Fix $\rho\in(-1,0)$ and (possibly random) initial conditions $(\mu
_0,\nu
_0)\in\calM_\tem^2$ or $\calM_\rap^2$. Then there is at most one
solution $(\mu,\nu,\Lambda)$ to the martingale problem $(\mathbf
{MP'})_{\mu_0,\nu_0}^\rho$ satisfying the ``separation of types''
property \eqref{singularity 1}.
\end{prop}
\begin{pf}
Let $(\mu,\nu,\Lambda)$ and $(\mu',\nu',\Lambda')$ be any two solutions
to $(\mathbf{MP'})_{\mu_0,\nu_0}^\rho$, with (possibly random) initial
conditions $(\mu_0,\nu_0)\in\calM_\tem^2$, which satisfy condition~\eqref{singularity 1}. By Propositions \ref{tightness_MZ dual} and
\ref
{MPinf 2}, we know that for any $(\tilde u_0,\tilde v_0)\in(\calB
_\rap
^+)^2$, there exists a solution $(\tilde\mu_t,\tilde\nu_t)_t\in
D_{[0,\infty)}(\calM_{\rap}^2)$
of the martingale problem $(\mathbf{MP'})_{\tilde u_0,\tilde v_0}^\rho
$, which by Lemma~\ref{lemma:sep} satisfies also the ``separation of
types'' condition~\eqref{singularity 1}.
Note that \eqref{singularity 1} ensures that both assumptions \eqref
{singularity4} and \eqref{upper bound singularity} of Proposition~\ref
{prop:self-dual} hold [for \eqref{upper bound singularity}, use
Lemma~\ref{lemma estimates}(a) in the \hyperref[app]{Appendix}].
Consequently, we can apply the self-duality of Proposition~\ref
{prop:self-dual} to conclude that for all $(\tilde u_0,\tilde v_0)\in
(\calB_\rap^+)^2$, we have
\begin{eqnarray*}
\int_0^T \E \bigl[F(\mu_t,
\nu_t,\tilde u_0,\tilde v_0) \bigr]\,dt&=&\lim
_{\eps\downarrow0}\int_0^T\E
\bigl[F(S_\eps\mu_0,S_\eps\nu _0,
\tilde\mu _t,\tilde\nu_t) \bigr]\,dt
\\
&=& \int_0^T\E \bigl[F\bigl(
\mu'_t,\nu'_t,\tilde
u_0,\tilde v_0\bigr) \bigr]\,dt,\qquad  T\ge0.
\end{eqnarray*}
Differentiating, we get
%
\begin{equation}
\E \bigl[F(\mu_t,\nu_t,\tilde u_0,\tilde
v_0) \bigr] = \E \bigl[F\bigl(\mu '_t,
\nu'_t,\tilde u_0,\tilde v_0
\bigr) \bigr]
\end{equation}
first for Lebesgue-a.e. $t>0$ and then, by right-continuity, for all $t>0$.
Since for $\rho\in(-1,1)$, the family of functions $ \{F(\cdot
,\cdot;\tilde u_0,\tilde v_0)\dvtx(\tilde u_0,\tilde v_0)\in(\calB
_\rap
^+)^2 \}$
is measure-determining for $\calM_\tem^2$ (see, e.g., \cite
{Dawsonetal2003}, proof of Lemma~3.1), it follows that the
one-dimensional distributions of $(\mu,\nu)$ and $(\mu',\nu')$
coincide. Arguing as in \cite{Bass98}, proof of Theorem VI.3.2, this
can be easily extended to the finite-dimensional distributions; thus
$(\mu,\nu)$ and $(\mu',\nu')$ have the same law on $D_{[0,\infty
)}(\calM
_\tem^2)$.

The proof for initial conditions in $\calM_\rap$ is completely analogous.
\end{pf}

\section{Bounds on the width of the interface}\label{ssn:width}

In this section, we will prove the $p$th moment estimate on
the approximate width of the interface $(R_t(\eps) - L_t(\eps))$
of Theorem~\ref{thmm:width}
using the fourth moment estimates established in
Proposition~\ref{mixed_moments}.
Since we are interested in the dependence of the constants on $\gamma$,
we write as above $(u^\sse{\gamma}_t, v^\sse{\gamma}_t)$ for a solution
of $\operatorname{cSBM}(\rho,\gamma)$ and moreover
define
\[
L_t^\sse{\gamma}(\eps) = \inf \biggl\{ x \dvtx\int
_{-\infty}^x u^\sse{\gamma
}_t(y) v^\sse{\gamma}_t(y) \,dy\geq\eps \biggr\}
\wedge R\bigl(u^\sse {\gamma }_t,v^\sse{
\gamma}_t\bigr)
\]
and
\[
R_t^\sse{\gamma}(\eps) = \sup \biggl\{ x \dvtx\int
_x^\infty u^\sse {\gamma
}_t(y) v^\sse{\gamma}_t(y) \,dy\geq\eps \biggr\}
\vee L\bigl(u^\sse{\gamma }_t,v^\sse{
\gamma}_t\bigr).
\]

\begin{pf*}{Proof of Theorem~\ref{thmm:width}}
First, we prove the statement for the case $\gamma= 1$, and
at the end we will deduce the statement for general $\gamma$ using a
scaling argument.
Therefore, we write $(u_t, v_t) := (u^\sse{1}_t, v^\sse{1}_t)$ and
$(R_t, L_t) := (R_t^\sse{1}, L_t^\sse{1})$.
We recall from~\eqref{eq:reforml} and (\ref{eq:0202-2}) in the proof of
Proposition~\ref{mixed_moments}
(for the system with branching rate $1$)
that since $\rho< - \frac{1}{\sqrt{2}}$,\vadjust{\goodbreak} for any $\tilde\eps\in
(0,\frac{1}{2})$
there exists a constant $ C(\rho, \tilde\eps) > 0$ such that for all $z
> 0$ and $t \geq0$,
\begin{eqnarray*}
\label{eq:0403-1}&&\E_{\1_{\R^-},\1_{\R^+}} \biggl[ \int_\R
u_t(x) v_t(x) u_t(x + z) v_t(x+z)
\,dx \biggr]
\\
&&\qquad =\E_{\1_{\R^-},\1_{\R^+}} \biggl[ \int_\R u_t(x)
v_t(x) u_t(x - z) v_t(x -z) \,dx \biggr]
\\
&&\qquad \leq C(\rho, \tilde\eps) \bigl(1 \wedge z^{-2(1-\tilde\eps)}\bigr).
\end{eqnarray*}
Defining for $q \in(0,1)$
\[
I_q(t) := \int_\R\int_\R|x-y|^q
u_t(x) v_t(x) u_t(y) v_t(y) \,dx \,dy
\]
and choosing $\tilde{\eps} =\frac{1}{4}(1-q) $, the estimate
in~(\ref
{eq:0403-1}) shows that for all $t \geq0$,
\begin{eqnarray*}
\E_{\1_{\R^-},\1_{\R^+}} \bigl[ I_q(t) \bigr]& =& 2 \int
_0^\infty|z|^q \E_{\1
_{\R^-},\1_{\R^+}}
\biggl[ \int_\R u_t(x) v_t(x)
u_t(x + z) v_t(x+z) \,dx \biggr] \,dz
\\
& \leq& C\biggl(\rho,\frac{1}4 (1-q)\biggr) \int_0^\infty
z^q \bigl(1 \wedge z^{-2(1-\tilde{\eps})}\bigr) \,dz
\\
& \leq& C\biggl(\rho,\frac{1}4(1-q)\biggr) \biggl( 1 + \int
_1^\infty z^{-2 +
2\tilde
{\eps} + q} \,dz \biggr)
\\
& =& C\biggl(\rho,\frac{1}4 (1-q)\biggr) \biggl( 1+ \frac{2}{1-q}
\biggr) < \infty
\end{eqnarray*}
since by our choice of $\tilde{\eps}$ we have $2\tilde{\eps}+q =
\frac{1}{2}
+\frac{1}{2} q < 1$.
Fix $z > 0$. Then on the event that $R_t(\eps) - L_t(\eps) > z$, we
can estimate
using the definition of $L_t(\eps), R_t(\eps)$ that
\[
I_q(t) \geq z^q \int_{-\infty}^{L_t(\eps)}
u_t(x) v_t(x) \,dx \int_{R_t(\eps)}^\infty
u_t(y) v_t(y) \,dy \geq\eps^2 z^q.
\]
Hence we can conclude that
\begin{eqnarray*}
\p_{\1_{\R^-},\1_{\R^+}} \bigl\{ R_t(\eps) - L_t(\eps ) > z
\bigr\}& \leq&\eps ^{-2} z^{-q} \E_{\1_{\R^-},\1_{\R^+}}\bigl[
I_q(t)\1_{ \{ R_t(\eps) -
L_t(\eps) > z\}} \bigr]
\\
& \leq&\eps^{-2} z^{-q} \E_{\1_{\R^-},\1_{\R^+}}\bigl[
I_q(t)\bigr] \leq \tilde C(\rho,q) \eps^{-2}
z^{-q},
\end{eqnarray*}
where we define $\tilde C(\rho,q) := C(\rho, \frac{1}4 (1-q)) (1 +
\frac
{2}{1-q})$.
Thus, we have by Fubini that for any $0< p< q < 1$,
\begin{eqnarray*}
\E_{\1_{\R^-},\1_{\R^+}} \bigl[ \bigl(\bigl(R_t(\eps) - L_t(
\eps)\bigr)^+\bigr)^p \bigr] & =& p\int_0^\infty
z^{p-1} \p_{\1_{\R^-},\1_{\R^+}} \bigl\{ R_t(\eps) -
L_t(\eps) > z \bigr\} \,dz
\\
& \leq& p \int_0^\infty z^{p-1} \bigl(
1 \wedge\tilde C(\rho, q) \eps^{-2} z^{-q } \bigr) \,dz
\\
& =& p \bigl(\tilde C(\rho,q) \eps^{-2}\bigr)^{{p}/{q}} \int
_0^\infty z^{p-1}\bigl(1 \wedge
z^{-q}\bigr) \,dz.
\end{eqnarray*}
Therefore, for any $\delta\in(0,2(1-p))$, by choosing $q = \frac
{2p}{2-\delta} \in(p,1)$ we can find a constant $C(\rho, p, \delta)$
such that for all $t \geq0$,
%
\begin{equation}
\label{eq:2507-1} \E_{\1_{\R^-},\1_{\R^+}} \bigl[ \bigl(\bigl(R_t(\eps) -
L_t(\eps)\bigr)^+\bigr)^p \bigr] \leq C(\rho, p, \delta)
\eps^{-2 + \delta}.
\end{equation}
Finally, we return to the case of a general branching rate $\gamma$.
Then, by the scaling property~\eqref{scaling property}, we have that
\begin{eqnarray*}
L_t^\sse{\gamma}(\eps) & =& \inf \biggl\{ x \dvtx \int
_{-\infty}^x u^\sse {\gamma}_t(y)
v^\sse{\gamma}_t(y) \,dy\geq\eps \biggr\} \wedge R
\bigl(u^\sse {\gamma}_t,v^\sse{
\gamma}_t\bigr)
\\
& \stackrel{d} {=}& \inf \biggl\{ x \dvtx\int_{-\infty}^x
u^\sse {1}_{\gamma^2
t}(\gamma y) v^\sse{1}_{\gamma^2 t}(
\gamma y) \,dy\geq\eps \biggr\} \wedge \frac{1} \gamma R
\bigl(u^\sse{1}_{\gamma^2 t},v^\sse{1}_{\gamma^2 t}
\bigr) \\
&=& \frac{1}{\gamma} L^\sse{1}_{\gamma^2 t} (\gamma\eps).
\end{eqnarray*}
Similarly, $R_t^\sse{\gamma}(\eps) \stackrel{d}{=} \frac{1} \gamma
R_{\gamma^2t}^\sse{1}(\gamma\eps)$.
Hence, by~\eqref{eq:2507-1} (which holds for branching rate~$1$), we
can deduce that
\begin{eqnarray*}
&&\E_{\1_{\R^-},\1_{\R^+}} \bigl[ \bigl(\bigl(R^\sse{\gamma }_t(
\eps) - L^\sse {\gamma}_t(\eps)\bigr)^+\bigr)^p
\bigr] \\
&&\qquad = \gamma^{-p} \E_{\1_{\R^-},\1_{\R^+}} \bigl[ \bigl(
\bigl(R^\sse{1}_{\gamma
^2t}(\gamma\eps) - L^\sse{1}_{\gamma^2 t}(
\gamma\eps)\bigr)^+\bigr)^p \bigr]
\\
&&\qquad \leq C(\rho, p, \delta) \eps^{-2 + \delta} \gamma^{- ( 2 + p -
\delta
) }.
\end{eqnarray*}
\upqed\end{pf*}

\begin{appendix}\label{app}
\section*{Appendix}
\subsection{Notation and spaces of functions and measures}\label{appendix0}

In this appendix, for the convenience of the reader, we have collected
our notation, and we recall some well-known facts concerning the spaces
of functions and measures employed throughout the paper. Most of the
material in this subsection can be found, for example, in \cite
{Dawsonetal2002,Dawsonetal2003} or \cite{EF04}.

For $\lambda\in\mathbb{R}$,
let
\[
\phi_{\lambda}(x) := e^{- \lambda|x|},\qquad  x \in\mathbb{R},
\]
and for $f\dvtx\mathbb{R}\rightarrow\mathbb{R}$, define
\[
|f|_{\lambda} := \Vert f / \phi_{\lambda}\Vert_{\infty},
\]
where $\Vert\cdot\Vert_{\infty}$ is the supremum norm.
Let $\mathcal{B}_{\lambda}$ denote the space of all measurable
functions $f\dvtx\mathbb{R} \rightarrow\mathbb{R}$ such that
$|f|_{\lambda}
<\infty$ and with the property that $f(x) / \phi_{\lambda}(x)$ has a
finite limit
as $|x| \rightarrow\infty$. Next, introduce the spaces
%
\begin{equation}
\mathcal{B}_{\rap} := \bigcap_{\lambda> 0}
\mathcal{B}_{\lambda} \quad\mbox{and} \quad\mathcal{B}_{\tem} := \bigcap
_{\lambda> 0} \mathcal{B}_{- \lambda}
\end{equation}
of
\textit{rapidly decreasing} and \textit{tempered} measurable
functions,
respectively.

We write $\mathcal{C}_{\lambda}, \mathcal{C}_{\rap}, \mathcal
{C}_{\tem}$
for the subspaces of continuous functions in $\calB_\lambda$, $\calB
_\rap$, $\calB_\tem$, respectively.
If we additionally require that all partial derivatives up to order
$k\in\N$ exist and belong to $\calC_\lambda,\calC_\rap,\calC
_\tem$, we write
$\calC_\lambda^{(k)},\mathcal{C}_{\rap}^{(k)}, \mathcal{C}_{\tem}^{(k)}$.
We will also use the space $\calC_c^\infty$ of infinitely
differentiable functions with compact support. If $\calF$ is any of the
above spaces of functions, the notation $\calF^+$ will refer to the
subset of nonnegative elements of $\calF$.

For each $\lambda\in\mathbb{R}$, the linear space $\mathcal
{C}_{\lambda}$ endowed
with the norm $| \cdot|_{\lambda}$ is a separable Banach space, and
the space $\mathcal{C}_{\rap}$ is topologized by the metric
%
\begin{equation}
d_{\rap}^{\mathcal{C}}(f,g) := \sum_{n = 1}^{\infty}
2^{-n}\bigl(|f-g|_{n} \wedge1\bigr),\qquad f,g \in\mathcal{C}_{\rap},
\end{equation}
which turns it into a Polish space. Analogously, $\mathcal{C}_{\tem}$
is Polish if we topologize it with the metric
%
\begin{equation}
d_{\tem}^{\mathcal{C}}(f,g) := \sum_{n=1}^{\infty}
2^{-n}\bigl(|f-g|_{-1/n} \wedge1\bigr), \qquad f,g \in\mathcal{C}_{\tem}.
\end{equation}

Let $\mathcal{M} $ denote the space of
(nonnegative) Radon
measures on $\mathbb{R}$.
For $\mu\in\calM$ and a measurable function $f$, we will use any of the
following notation:
\[
\langle\mu, f \rangle, \qquad \int_\R\mu(dx) f(x),\qquad \int
_\R f(x) \mu(dx)
\]
to denote the integral of $f$ with
respect to the measure $\mu$ (if it exists). For integrals with respect
to the Lebesgue measure $\ell$ on $\R$,
we will simply write $dx$ in place of $\ell(dx)$. If $\mu\in\calM$ is
absolutely continuous w.r.t. $\ell$, we will identify $\mu$ with its
density, writing
\[
\mu(dx)=\mu(x) \,dx.
\]

For $\lambda\in\R$, define
\[
\calM_\lambda:= \bigl\{\mu\in\calM\dvtx\langle\mu,\phi_\lambda
\rangle< \infty \bigr\},
\]
and introduce the spaces
\[
\mathcal{M}_{\tem} := \bigcap_{\lambda> 0}
\calM_\lambda,\qquad \calM _\rap :=\bigcap
_{\lambda>0}\calM_{-\lambda}
\]
of \textit{tempered} and \textit{rapidly decreasing measures}, respectively.
These spaces of measures are topologized as follows: Let $d_0$ be a
complete metric on $\mathcal{M}$ inducing the vague topology, and define
%
\begin{equation}
\label{defn:top_tem} d_{\tem}^{\mathcal{M}}(\mu,\nu) := d_0(
\mu, \nu) + \sum_{n=1}^{\infty}
2^{-n}\bigl(|\mu- \nu|_{1/n} \wedge1\bigr),\qquad \mu, \nu\in
\mathcal{M}_{\tem},
\end{equation}
where we write
\[
|\mu- \nu|_{\lambda} := \bigl|\langle\mu, \phi_{\lambda} \rangle- \langle
\nu, \phi_{\lambda} \rangle\bigr|.
\]
Note that with the above metric, $(\mathcal{M}_{\tem}, d_{\tem
}^{\mathcal{M}})$ is also Polish, and it is easily seen that $\mu
_n\to
\mu$ in $\calM_\tem$ if and only if $\langle\mu_n,\varphi\rangle
\to
\langle\mu,\varphi\rangle$ for all $\varphi\in\bigcup_{\lambda> 0}
\calC_{\lambda}$.
Denote by $\mathcal{M}_f$
the space of finite measures on
$\mathbb{R}$ endowed with the topology of weak convergence. Note that
we have $\calM_\rap\subseteq\calM_f$. The space $\calM_\rap$ is then
topologized by saying that $\mu_n\to\mu$ in $\calM_\rap$ if and
only if
$\mu_n\to\mu$ in $\calM_f$ (w.r.t. the weak topology) and $\sup_{n\in\N
}\langle\mu_n,\phi_\lambda\rangle<\infty$ for all $\lambda<0$;
see \cite
{Dawsonetal2003}, page 140. It is easy to see that this topology is
also induced by the metric
%
\begin{equation}
d_{\rap}^{\mathcal{M}}(\mu,\nu) := \widetilde{d}_0(\mu,
\nu) + \sum_{n=1}^{\infty} 2^{-n}\bigl(|
\mu- \nu|_{-n} \wedge1\bigr), \qquad\mu, \nu\in \mathcal {M}_{\rap},
\end{equation}
where $\widetilde{d}_0$ is a complete metric on $\calM_f$ inducing the
weak topology.
Again, when endowed with this metric $(\mathcal{M}_{\rap}, d_{\rap
}^{\mathcal{M}})$ becomes a Polish space.

It is clear that $\calC^+_\tem$ may be viewed as a subspace of $\calM
_\tem$ by taking a function $u\in\calC_\tem^+$ as a density w.r.t.
Lebesgue measure, that is, by identifying it with the measure $u(x)\,dx$.
It is also clear that the topology of $\calM_\tem$ restricted to
$\calC
_\tem^+$ is weaker than the topology on $\calC_\tem$ introduced above.
The same holds for the relation between $\calC_\rap^+$ and $\calM
_\rap
$. Thus we have \emph{continuous} embeddings $\calC^+_\tem
\hookrightarrow\calM_\tem$ and $\calC_\rap^+\hookrightarrow\calM
_\rap$.

Let $(p_t)_{t\ge0}$ denote the heat
kernel in $\mathbb{R}$ corresponding to $\frac{1}{2} \Delta$,
%
\begin{equation}
p_t(x) = \frac{1}{(2 \pi t)^{1/2}} \exp \biggl\{- \frac{|x|^2}{2 t} \biggr
\},\qquad t > 0, x \in\mathbb{R},
\end{equation}
and write $ (S_t)_{ t \geq0}$ for the associated heat semigroup (i.e.,
the transition semigroup of Brownian motion).
For $\mu\in\calM$ and $x\in\R$, let $S_t\mu(x):=\int_\R p_t(x-y)
\mu
(dy)$. The following estimates are well known and can be proved as in
Appendix A of~\cite{Dawsonetal2003} (see also \cite{Shiga94}, Lemma~6.2(ii)):

\begin{lemma}\label{lemma estimates}
Fix $\lambda\in\R$ and $T>0$.
\begin{longlist}[(a)]
\item[(a)]
For all $\varphi\in\calB_{\lambda}^+$, we have
%
\begin{equation}
\label{estimate 1} \sup_{t\in[0,T]}S_t\varphi(x)\le C\bigl(
\lambda,T\bigr) |\varphi|_\lambda \phi _\lambda(x),\qquad x\in\R.
\end{equation}
Moreover, there is a positive constant $C'(\lambda,T)>0$ such that we
have a lower bound
%
\begin{equation}
\label{estimate 1a} \inf_{t\in[0,T]}S_t\phi_{\lambda}(x)
\ge C'(\lambda,T) \phi _{\lambda}(x) ,\qquad x\in\R.
\end{equation}
\item[(b)]
Let $0<\eps<T$. Then for
all $\mu\in\calM_\lambda$ we have
%
\begin{equation}
\label{estimate 2} \sup_{t\in[\eps,T]}S_t\mu(x) \le C(
\lambda,T,\eps) \langle\mu,\phi_{\lambda}\rangle\phi _{-\lambda
}(x),\qquad  x
\in\R.
\end{equation}
\end{longlist}
In particular, the heat semigroup preserves the space $\calB_\lambda$
and maps $\calM_\lambda$ into $\calB_\lambda$.
\end{lemma}

For $T > 0$ and $\lambda\in\mathbb{R}$, let $\mathcal{C}_{T,
\lambda
}^{(1,2)}$ denote
the space of real-valued functions $\psi$ defined on $[0,T] \times
\mathbb{R}$ such that
$t \mapsto\psi_t(\cdot)$, $t \mapsto\partial_t \psi_t(\cdot)$
and $t
\mapsto\Delta\psi_t(\cdot)$
are continuous $\mathcal{C}_{\lambda}$-valued functions, and define
\[
\mathcal{C}^{(1,2)}_{T, \rap} := \bigcap
_{\lambda> 0} \mathcal {C}^{(1,2)}_{T, \lambda},\qquad
\mathcal{C}^{(1,2)}_{T, \tem} := \bigcap
_{\lambda> 0} \mathcal {C}^{(1,2)}_{T, -\lambda}.
\]

The following is a simple corollary of Lemma~\ref{lemma estimates}:

\begin{corollary}\label{cor estimates}
Fix $\lambda\in\R$ and $T>0$.
\begin{longlist}[(a)]
\item[(a)] For all $\varphi\in\calC_\lambda^{(2)}$, the function
\[
\psi_t(x):=S_{T-t}\varphi(x),\qquad t\in[0,T], x\in\R
\]
is in $\calC^{(1,2)}_{T,\lambda}$.
\item[(b)] For all $\mu\in\calM_\lambda$ and $\eps>0$, the function
\[
\psi_t(x):=S_{T-t}\mu(x),\qquad t\in[0,T-\eps], x\in\R
\]
is in $\calC^{(1,2)}_{T-\eps,\lambda}$.
\end{longlist}
\end{corollary}

For a Polish space $E$ and $I\subseteq\R$, we denote by $D_I(E)$,
respectively $\calC_I(E)$, the space of c\`adl\`ag, respectively
continuous, $E$-valued paths $t\mapsto f_t$, $t\in I$. (In our case, we
will always have $I=[0,\infty)$ or $I=(0,\infty)$ and $E\in\{(\calC
^+_\tem)^m,(\calC^+_\rap)^m, \calM_\tem^m,\calM_\rap^m\}$ for
some power
$m\in\N$.) Endowed with the usual Skorokhod ($J_1$)-topology,
$D_{I}(E)$ is then also Polish. In this paper, we will use the
Skorokhod topology only in restriction to $\calC_{I}(E)$ where it
coincides with the usual topology of locally uniform convergence.

For processes which are c\`adl\`ag but not continuous, we will instead
use the weaker \emph{Meyer--Zheng ``pseudo-path'' topology} on
$D_{[0,\infty)}(E)$. To describe the Meyer--Zheng topology, introduced
in \cite{MZ84}, let $\lambda(dt):= \exp(-t) \,dt$, and let $w(t), t
\in
[0,\infty)$ be an $E$-valued Borel function. Then, a ``pseudo-path''
corresponding to $w$ is the probability law $\psi_w$ on $[0, \infty)
\times E$ given as the image measure of $\lambda$ under
the mapping $ t \mapsto(t, w(t))$. Note that two functions which are
equal Lebesgue-a.e. give rise to the same pseudo-path. Further $w
\mapsto\psi_w$ is one-to-one on the space of c\`adl\`ag paths
$D_{[0,\infty)}(E)$, and thus yields an embedding of $D_{[0,\infty
)}(E)$ into the space of probability measures on $[0, \infty) \times
E$. The induced topology on $D_{[0,\infty)}$ is then called the
pseudo-path topology. Very conveniently, convergence in this topology
is equivalent to convergence in Lebesgue measure; see \cite{MZ84}, Lemma~1.

For $E=\R$,
\cite{MZ84}, Theorem~4, provides a rather convenient sufficient
condition for relative compactness of a sequence of stochastic
processes on $D_{[0,\infty)}(E)$ equipped with this topology.
The condition can be stated as follows: If $(X^\ssup{n}_t)_{t\ge0}$,
$n\in\N$ is a sequence of c\`adl\`ag real-valued stochastic processes,
with $(X^\ssup{n}_t)_{t\ge0}$ adapted to a filtration $(\calF^\ssup
{n})_{t\ge0}$, then Meyer and Zheng require that
%
\begin{equation}
\label{MZ condition} \sup_{n\in\N} \Bigl(V_T
\bigl(X^\ssup{n}\bigr)+\sup_{t\le T}\E
\bigl[\bigl|X_t^\ssup {n}\bigr|\bigr] \Bigr)<\infty
\end{equation}
for all $T>0$. Here $V_T(X^\ssup{n}):=\sup\E [\sum_i\llvert \E
[X^\ssup
{n}_{t_{i+1}}-X_{t_i}^\ssup{n} | \calF^\ssup{n}_{t_i}]\rrvert  ]$,
where the $\sup$ is taken over all partitions of the interval $[0,T]$,
denotes the conditional variation of $X^\ssup{n}$ up to time $T$.
In \cite{Kurtz91}, this tightness criterion was extended to processes
taking values in general separable metric spaces $E$, which is the
version we need for our measure-valued processes. In fact, by \cite{Kurtz91},
Corollary~1.4, we only have to check condition \eqref{MZ
condition} for the coordinate processes and in addition a compact
containment condition in order to obtain tightness of our
measure-valued processes in the pseudopath topology (which again is
equivalent to the topology of convergence in Lebesgue
measure).\footnote
{Note, however, that the main result in \cite{Kurtz91} is much stronger
than just an extension of the Meyer--Zheng tightness criterion to a
general state space $E$. Also note that in \cite{Kurtz91}, equation
(1.7), there seems to be missing a term $\sup_{s\le t}\E[|f_i\circ
X_s^\ssup{n}|]$; cf. equation (1.2).}

\subsection{Martingale problems and Green function
representations}\label{appendix1}

We define all stochastic processes
over a sufficiently rich stochastic basis
$(\Omega, \mathcal{F},\break  (\mathcal{F}_t)_{t \geq0}, \mathbb{P})$ satisfying
the usual hypotheses. If $Y = (Y_t)_{t \geq0}$ is a stochastic
process taking values in $E$ and starting at $Y_0=y\in E$, the law of
$Y$ is denoted $\mathbb{P}_y$,
and we use $\E_y$ to denote the corresponding expectation.

Recall that solutions to the finite rate symbiotic branching model\break 
$\operatorname{cSBM}(\rho,\gamma)$
are characterized by the martingale problem given in \cite{EF04}, Definition~3. Consequently, when the solutions are interpreted as densities
w.r.t. Lebesgue measure, the corresponding measure-valued processes
solve the martingale problem $(\mathbf{MP})_{\mu_0,\nu_0}^\rho$ of
Definition~\ref{defn:MP}. In this
appendix, we have collected some properties of solutions to this
martingale problem which for the finite rate model $\operatorname{cSBM}(\rho
,\gamma)$ can already be found in \cite{EF04}; however, they are in
fact true for \emph{any} solution to $(\mathbf{MP})_{\mu_0,\nu
_0}^\rho$. These are: an extended
martingale problem for space--time functions which in turn implies a
Green function representation, and a (weaker) martingale problem
involving the self-duality function $F$ from \eqref{self-duality
function}; see Proposition~\ref{prop MPinf} below. We include a proof
only for the latter, in order to illustrate the point that the
particular form of the quadratic variation process $(\Lambda_t)_{t\ge
0}$ from Definition~\ref{defn:MP} is irrelevant in this respect.

Recall that we consider the increasing process $t\mapsto\Lambda_t(dx)$
also as a (locally finite) measure $\Lambda(ds,dx)$ on $\R^+\times\R
$, via
\[
\Lambda\bigl([0,t]\times B\bigr):=\Lambda_t(B).
\]
The following ``space--time version'' of the martingale problem
$(\mathbf{MP})_{\mu_0,\nu_0}^\rho$
can be proved by standard arguments; see, for example, \cite
{Dawsonetal2002}, Lemma~42:

\begin{lemma}\label{extended martingale problem}
Fix $\rho\in[-1,1]$ and initial conditions $(\mu_0,\nu_0)\in\calM
_\tem
^2$ (resp., $\calM_\rap^2$). Let $T>0$. If $(\mu_t,\nu_t)_{t\ge
0}\in
\calC_{[0,\infty)}(\calM_{\tem}^2)$ [resp., $\calC_{[0,\infty
)}(\calM
_{\rap}^2)$] is any solution to the martingale problem $(\mathbf
{MP})_{\mu_0,\nu_0}^\rho$, then for
all test functions $\phi,\psi\in\calC_{T,\rap}^{(1,2)}$ (resp.,
$\phi
,\psi\in\calC_{T,\tem}^{(1,2)}$)
we have that
%
\begin{eqnarray}
\label{MP2} \langle\mu_t, \phi_t\rangle&=& \langle
\mu_0,\phi_0\rangle+ \int_0^t
\biggl\langle\mu_s, \frac{1}{2}\Delta\phi_s+
\frac{\partial}{\partial
s}\phi_s\biggr\rangle \,ds
\nonumber
\\
&&{}+ \int_{[0,t]\times\R}
\phi_s(x) M\bigl(d(s,x)\bigr),
\nonumber
\\[-8pt]
\\[-8pt]
\nonumber
\langle\nu_t, \psi_t\rangle&= &\langle\nu_0,
\psi_0\rangle+ \int_0^t \biggl
\langle\nu_s, \frac{1}{2}\Delta\psi_s+
\frac{\partial}{\partial
s}\psi_s\biggr\rangle \,ds
\\
&&{}+ \int_{[0,t]\times\R}
\psi_s(x) N\bigl(d(s,x)\bigr)\nonumber
\end{eqnarray}
for $t\in[0,T]$, where $M(d(s,x))$ and $N(d(s,x))$ are zero-mean
martingale measures
with covariance structure
%
\begin{eqnarray}
\label{Cov2} && \biggl[\int_{[0,\cdot]\times\R}f_s(x) M
\bigl(d(s,x)\bigr) \biggr]_t\nonumber
\\
& &\qquad= \biggl[\int_{[0,\cdot]\times\R}f_s(x) N\bigl(d(s,x)\bigr)
\biggr]_t 
=\int_{[0,t]\times\R}f^2_s(x)
\Lambda(ds,dx),
\nonumber
\\[-8pt]
\\[-8pt]
\nonumber
&& \biggl[\int_{[0,\cdot]\times\R}f_s(x) M\bigl(d(s,x)\bigr),
\int_{[0,\cdot
]\times\R
}g_s(x) N\bigl(d(s,x)\bigr)
\biggr]_t
\\
&&\qquad =\rho\int_{[0,t]\times\R}f_s(x)g_s(x)
\Lambda(ds,dx),\nonumber
\end{eqnarray}
with $\Lambda$ from \eqref{Cov1}. Here, $f$ and $g$ are predictable
functions defined on $\Omega\times\mathbb{R}_+ \times\mathbb{R}$
such that
%
\begin{equation}
\mathbb{E}_{\mu_0,\nu_0} \biggl[ \int_{[0,t]\times\R}
f_s^2(x) \Lambda (ds,dx) \biggr] < \infty,\qquad t \in[0,T].
\end{equation}
\end{lemma}

The previous lemma immediately implies
a Green function representation for solutions to the martingale problem
$(\mathbf{MP})_{\mu_0,\nu_0}^\rho$
[recall that $(S_t)_{t\ge0}$ denotes the heat semigroup]:

\begin{corollary}[(Green function representation)]\label{cor:Green
function representation}
Under the assumptions of Lemma~\ref{extended martingale problem}, we
have for all $T > 0$ and test functions $\varphi\in\bigcup_{\lambda>0}
\mathcal{C}_{\lambda}$ (resp., $\bigcup_{\lambda>0}\mathcal
{C}_{-\lambda
}$) that
%
\begin{eqnarray}
\label{MP7} \langle\mu_t, S_{T-t}\varphi\rangle&=&
\langle\mu_0,S_T\varphi \rangle + \int
_{[0,t]\times\R}S_{T-s}\varphi(x) M(ds,dx),
\nonumber
\\[-8pt]
\\[-8pt]
\nonumber
\langle\nu_t, S_{T-t}\varphi\rangle&=& \langle
\nu_0,S_T\varphi \rangle + \int_{[0,t]\times\R}S_{T-s}
\varphi(x) N(ds,dx)
\end{eqnarray}
for $t\in[0,T]$,
where $M(d(s,x))$, $N(d(s,x))$ are the martingale measures from Lemma~\ref{extended martingale problem}. In particular, $(\langle\mu_t,
S_{T-t}\varphi\rangle)_{t\in[0,T]}$ and $(\langle\nu_t,
S_{T-t}\varphi
\rangle)_{t\in[0,T]} $ are martingales with covariance structure given
by \eqref{Cov2} with $f_s(x)=g_s(x)=S_{T-s}\varphi(x)$.
\end{corollary}

\begin{pf}
For $\varphi\in\calC_\rap^{(2)}$ (resp., $\calC_\tem^{(2)}$), this
follows at once from
the extended martingale problem of Lemma~\ref{extended martingale
problem} by putting $\phi_t:=\psi_t:=S_{T-t}\varphi$ for $t\in[0,T]$,
observing that the latter function is in $\calC_{T,\rap}^{(1,2)}$ for
$\varphi\in\calC_\rap^{(2)}$ (resp., in $\calC_{T,\tem}^{(1,2)}$ for
$\varphi\in\calC_\tem^{(2)}$) by Corollary~\ref{cor estimates},
and that $(\frac{1}{2}\Delta+\frac{\partial}{\partial
_s})S_{T-s}\varphi
\equiv0$. In order to extend \eqref{MP7} to more general $\varphi$, one
uses simple approximation arguments involving monotone, respectively
dominated, convergence.
\end{pf}

\begin{prop}\label{prop MPinf}
Fix $\varrho\in(-1,1)$ and $(\mu_0,\nu_0)\in\calM_\tem^2$
(resp., $\calM
_\rap^2$). Let $(\mu_t,\nu_t)_{t\ge0}\in\calC_{[0,\infty)}(\calM
_{\tem
}^2)$ [resp., $\calC_{[0,\infty)}(\calM_{\rap}^2)$] be any solution to
the martingale problem $(\mathbf{MP})_{\mu_0,\nu_0}^\rho$. Then the
process $(\Lambda_t)_{t\ge
0}\in
\calC_{[0,\infty)}(\calM_\tem)$ [resp., $\calC_{[0,\infty)}(\calM
_\rap
)$] from Definition~\ref{defn:MP}, governing the correlations of the
martingales as in \eqref{Cov1}, is increasing with $\Lambda_0=0$ and
satisfies condition \eqref{finiteness Lambda 1}. Moreover, for all
$T>0$ and (nonnegative) test functions $0\le\phi,\psi\in\calC
_{T,\rap
}^{(1,2)}$ (resp., $\in\calC_{T,\tem}^{(1,2)}$), the process
%
\begin{eqnarray}
\label{MP8} && F(\mu_t, \nu_t,\phi_t,
\psi_t) - F(\mu_0,\nu_0,\phi_0,
\psi_0)\nonumber
\\
& &\qquad{}- \int_0^t F(\mu_s,
\nu_s,\phi_s,\psi_s) \biggl\langle\biggl
\langle\mu_s, \nu_s, \biggl(\frac{1}{2}\Delta+
\frac{\partial}{\partial s} \biggr)\phi_s, \biggl(\frac{1}{2}\Delta+
\frac{\partial}{\partial s} \biggr)\psi _s\biggr\rangle\biggr
\rangle_\rho\, ds
\\
& &\qquad{}- 4\bigl(1-\rho^2\bigr)\int_{[0,t]\times\R} F(
\mu_s,\nu_s,\phi_s,\psi _s)
\phi _s(x)\psi_s(x) \Lambda(ds,dx),\nonumber
\end{eqnarray}
$ t\in[0,T]$, is a martingale
with quadratic variation given by
%
\begin{equation}
\label{qv2} 8\bigl(1-\rho^2\bigr)\int_{[0,t]\times\R}
F(\mu_s,\nu_s,\phi_s,\psi_s)^2
\phi _s(x)\psi_s(x) \Lambda(ds,dx).
\end{equation}
\end{prop}
\begin{pf}
In view of \eqref{Cov1}, it is clear that $\Lambda$ is increasing and
$\Lambda_0=0$. Moreover, since the martingales in Definition~\ref
{defn:MP} are assumed square integrable, we have $\E_{\mu_0,\nu
_0}[\langle\Lambda_t,\phi^2\rangle]=\E_{\mu_0,\nu_0}[M_t(\phi
)^2]<\infty
$ for all test functions $\phi\in\calC_{\rap}^{(2)}$ (resp., $\calC
_{\tem
}^{(2)}$). Thus \eqref{finiteness Lambda 1} is satisfied.

The proof of \eqref{MP8} is basically a straightforward application of
It\^{o}'s formula; cf. the proof of Proposition~5 in \cite{EF04}. We
sketch it here for the convenience of the reader and to make clear that
the arguments in \cite{EF04} do not rely on properties of the finite
rate model, but actually work for any solution to the martingale
problem $(\mathbf{MP})_{\mu_0,\nu_0}^\rho$.
Define
\begin{eqnarray*}
Y_t:=\langle\mu_t+\nu_t,\phi_t+
\psi_t\rangle& =&\langle\mu_0+\nu _0,\phi
_0+\psi_0\rangle
\\
&&{} +\int_0^t\biggl\langle\mu_s+
\nu_s, \biggl(\frac{1}{2}\Delta+\frac
{\partial
}{\partial s} \biggr) (
\phi_s+\psi_s)\biggr\rangle \,ds
\\
&&{} +\int_{[0,t]\times\R}\bigl(\phi_s(x)+
\psi_s(x)\bigr) (M+N) (ds,dx),
\\
Z_t:=\langle\mu_t-\nu_t,\phi_t-
\psi_t\rangle &=&\langle\mu_0-\nu _0,\phi
_0-\psi_0\rangle
\\
&&{} +\int_0^t\biggl\langle\mu_s-
\nu_s, \biggl(\frac{1}{2}\Delta+\frac
{\partial
}{\partial s} \biggr) (
\phi_s-\psi_s)\biggr\rangle \,ds
\\
&&{} +\int_{[0,t]\times\R}\bigl(\phi_s(x)-
\psi_s(x)\bigr) (M-N) (ds,dx),
\end{eqnarray*}
where $M$ and $N$ are the martingale measures from Lemma~\ref{extended
martingale problem}. We observe that $Y$ and $Z$ are continuous
real-valued semimartingales with covariance structure easily calculated as
\begin{eqnarray*}
[Y,Y]_t&=&2(1+\rho)\int_{[0,t]\times\R} \bigl(
\phi_s(x)+\psi_s(x)\bigr)^2 \Lambda(ds,dx),
\\
{}[Z,Z]_t&=&2(1-\rho)\int_{[0,t]\times\R} \bigl(
\phi_s(x)-\psi_s(x)\bigr)^2 \Lambda(ds,dx),
\\
{}[Y,Z]_t&=&0.
\end{eqnarray*}
Now define $H(y,z):=\exp(-\sqrt{1-\rho} y+i\sqrt{1+\rho} z)$, apply
It\^{o}'s formula to the process $(H(Y_t,Z_t))_{t\ge0}$ and use the
trivial identity $(\phi+\psi)^2-(\phi-\psi)^2=4\phi\psi$ to
obtain by a
straightforward calculation that
\begin{eqnarray*}
&&F(\mu_t,\nu_t,\phi_t,\psi_t)\\
&&\qquad=F(
\mu_0,\nu_0,\phi_0,\psi_0)
\\
&&\qquad\quad{}+\int_0^tF(\mu_s,
\nu_s,\phi_s,\psi_s)\cdot\biggl\langle
\biggl\langle\mu _s, \nu_s, \biggl(\frac{\Delta}{2}+
\frac{\partial}{\partial s} \biggr)\phi _s, \biggl(\frac{\Delta}{2}+
\frac{\partial}{\partial s} \biggr)\psi _s\biggr\rangle \biggr
\rangle_\rho \,ds
\\
&&\qquad\quad{}+4\bigl(1-\rho^2\bigr)\int_0^t
\int_\R F(u_s,v_s,
\phi_s,\psi_s) \phi _s(x)\psi
_s(x) \Lambda(ds,dx)
\\
&&\qquad\quad{}+\int_{[0,t]\times\R}F(u_s,v_s,
\phi_s,\psi_s) \\
&&\hspace*{44pt}\qquad\quad{}\times\bigl(-\sqrt{1-\rho } \bigl(\phi
_s(x)+\psi_s(x)\bigr)
 \\
 &&\hspace*{93pt}{}+i\sqrt{1+\rho} \bigl(\phi_s(x)-\psi_s(x)\bigr)
\bigr)M(ds,dx)
\\
&&\qquad\quad{}+\int_{[0,t]\times\R}F(u_s,v_s,
\phi_s,\psi_s)\\
&&\hspace*{44pt}\qquad\quad{}\times \bigl(-\sqrt{1-\rho } \bigl(\phi
_s(x)+\psi_s(x)\bigr)
 \\
 &&\hspace*{90pt}{}-i\sqrt{1+\rho} \bigl(\phi_s(x)-\psi_s(x)\bigr)
\bigr)N(ds,dx).
\end{eqnarray*}
This gives \eqref{MP8}, and computing the quadratic variation of the
martingale term in the above display, we obtain \eqref{qv2}.
\end{pf}

\begin{corollary}\label{cor MPinf}
Fix $\rho\in(-1,1)$ and $(\mu_0,\nu_0)\in\calM_\tem^2$ (resp.,
$\calM
_\rap^2$).
Then any solution $(\mu_t,\nu_t)_{t\ge0}\in\calC_{[0,\infty
)}(\calM
_{\tem}^2)$ [resp., $\calC_{[0,\infty)}(\calM_{\rap}^2)$] to the
martingale problem $(\mathbf{MP})_{\mu_0,\nu_0}^\rho$ is also
a solution
to the martingale problem $(\mathbf{MP'})_{\mu_0,\nu_0}^\rho$.
\end{corollary}

\subsection{Some facts on Brownian motion and its local time}\label{appendix2}

In this subsection, we recall some of the
standard facts (and their variations) on Brownian motion in a formulation
adapted to our needs.
In the following we will denote for any suitable process $(X_t)_{t \geq
0}$ its local time in $x$ by $L_t^x := L_t^{x,X}$.

\begin{lemma}\label{le:localtime} If $(B_t)_{t \geq0}$ is a Brownian
motion started at $x \in\R$ with
local time $(L_t^0)_{t \geq0}$ in~$0$, then
\[
\bigl(L_t^0\bigr)_{t \geq0} \stackrel{d} {=}
\bigl(M_t^+\bigr)_{t \geq0},
\]
where $(M_t)_{t \geq0}$ is the maximum process of a Brownian motion
started at $-|x|$.
\end{lemma}

\begin{pf} We adapt the proof of Theorem~7.38 in~\cite{MP10}.
By Tanaka's formula~\cite{MP10}, Theorem~7.33, we find that
\[
|B_t| - |x| = \int_0^t
\operatorname{sign}(B_s) \,\dd B_s + L^0_t.
\]
By~\cite{MP10}, Lemma~7.40, the stochastic integral is equal in
distribution to
a standard Brownian motion,
so if we set
\[
W_t = - \biggl(|x| + \int_0^t
\operatorname{sign}(B_s) \,\dd B_s \biggr) ,
\]
then $W$ is a linear Brownian motion started at $-|x|$, and
we have that
%
\begin{equation}
\label{eq:3101-1} |B_t| = - W_t + L_t^0.
\end{equation}
Let
$(M_t)_{t \geq0}$ denote the maximum process of $(W_t)_{t \geq0}$.
We want to show that for all $t \geq0$, we have $M_t^+ = L_t^0$.
It follows immediately
from~(\ref{eq:3101-1}) that
$W_s \leq L_s^0 \leq L_t^0$ for all $s \leq t$, so that by taking
the maximum we obtain $M_t^+=0\vee M_t \leq L_t^0$.

Now suppose there exists a time $t$ such that $M_t^+ < L^0_t$.
Let $u := \inf\{ r \le t \dvtx L^0_r = L^0_t \}$.
Since $L^0$ only increases on the set $\{ s \dvtx B_s = 0\}$, by continuity
and since $L^0_t > 0$, we must have $B_u = 0$.
In particular, from \eqref{eq:3101-1} we get
$W_u = L^0_u$ with $u \le t$. Thus, we can deduce that
\[
M_u \geq W_u = L^0_u =
L^0_t > M_t,
\]
which yields a contradiction since $u \le t$ and $M$
is obviously increasing. Hence, $M^+_t = L^0_t$ as claimed.
\end{pf}

\begin{lemma}\label{le:local_estimate}
Let $(B_t)_{t\ge0}$ be a Brownian motion started at $z \in\R$ with
local time $(L_t^0)_{t\ge0}$ in~$0$. Then for all $\alpha> 0$ and
$t\ge1$,
\[
\p_z \bigl\{ L_t^0 \leq\alpha\log t \bigr
\} \leq\sqrt{\frac{2}{\pi}} \frac{
\alpha\log t + |z|}{t^{{1}/{2}}}.
\]
\end{lemma}

\begin{pf} Using Lemma~\ref{le:localtime}, we find that if
$(M_t)_{t\ge0}$
denotes the maximum process of a Brownian motion started at $-|z|$, we
can estimate
\begin{eqnarray*}
\p_z \bigl\{ L_t^0 \leq\alpha\log t \bigr
\} & =& \p_{-|z|} \bigl\{ M_t^+ \leq\alpha\log t \bigr\} =
\p_0 \bigl\{ M_t \leq \alpha \log t + |z| \bigr\}
\\
& = &\p_0\bigl \{ |B_t| \leq\alpha\log t + |z| \bigr\} \leq\sqrt{
\frac
{2}{\pi}} \frac{\alpha\log t+ |z|}{t^{{1}/{2}}} ,
\end{eqnarray*}
where we used the reflection principle; see, for example,~\cite{MP10}, Theorem~2.21, in the second-to-last step.
\end{pf}

\begin{corollary}\label{le:asymp_collision} Suppose that $(B_t^\ssup
{1})_{t\geq0}$ and $(B_t^\ssup{2})_{t \geq0}$ are independent
Brownian motions
started at $x < y$, respectively, and denote their collision local time
as $(L_t^{1,2})_{t \geq0}$.
Then for all $\alpha> 0$ and $t\ge1$,
\[
\p_{x,y} \bigl\{ L_t^{1,2} \leq\alpha\log t \bigr
\} \leq\frac{1}{\sqrt{\pi}} \frac{2 \alpha\log t + y-x }{t^{{1}/{2}}}.
\]
\end{corollary}

\begin{pf} This follows immediately from Lemma~\ref{le:local_estimate}.
Note that $W_t := B_t^\ssup{2} - B_t^\ssup{1}$, $t \geq0$ is by
definition a Brownian motion
(with quadratic variation $2t$ and started at $y-x$), and thus $B_t :=
W_{t/2}- (y-x)$, $t\geq0$ is a standard
Brownian motion. Moreover
$L^{1,2}_t = L_t^{0,B^\ssup{2} - B^\ssup{1}} = L_t^{0,W}$.
Now observe that
\begin{eqnarray*}
L_t^{0,W} & =& \lim_{\eps\da0 }
\frac{1}{2 \eps} \int_0^t \1_{\{
|W_s| \leq\eps\}}
\,ds =\lim_{\eps\da0} \frac{1}{2\eps} \int_0^t
\1_{\{ |B_{2s}+y-x|
\leq
\eps\}} \,ds
\\
& \stackrel{d} {=} &\lim_{\eps\da0} \frac{1}{2\eps}\int
_0^t \1_{\{
|\sqrt{2} B_{s}+y-x| \leq\eps\}} \,ds =\frac{1}{\sqrt{2}}
L_t^{{(x-y)}/{\sqrt{2}},B}.
\end{eqnarray*}
Hence by Lemma~\ref{le:local_estimate},
\begin{eqnarray*}
\p_{x,y} \bigl\{ L_t^{1,2} \leq\alpha\log t \bigr
\} & =& \p_0 \bigl\{ L_t^{{(x-y)}/{\sqrt{2}},B} \leq\sqrt{2}
\alpha \log t \bigr\}\\
& =& \p_{{(y-x)}/{\sqrt{2}}} \bigl\{ L_t^{0,B}
\leq\sqrt{2} \alpha \log t \bigr\}
\\
& \leq&\sqrt{ \frac{2}{\pi}} \frac{ \sqrt{2} \alpha\log t +
({1}/{\sqrt{2}})(y - x)}{
t^{{1}/{2}}} ,
\end{eqnarray*}
which proves the corollary.
\end{pf}

The following is a slightly generalized version of Lemma~2 in \cite
{AT00}. It follows easily from the occupation times formula for
Brownian local time.

\begin{lemma}\label{lem:occ_times}
Let $B^\ssup{1}$, $B^\ssup{2}$ be independent Brownian motions defined
on $(\Omega,\calF,\p)$. Then for every $h\dvtx\R\times\R^+\times
\Omega\to\R$
measurable and bounded or nonnegative, we have
%
\begin{equation}\quad
\label{eq:occ_times}\int_0^t h\bigl(B^\ssup{2}_s-B^\ssup
{1}_s,s,\cdot\bigr) \,ds=\int_\R\int
_0^t h(z,s,\cdot) \,dL_s^{z,B^\ssup
{2}-B^\ssup{1}}\,dz,\qquad
\p\mbox{-a.s.}
\end{equation}
\end{lemma}

In \cite{AT00}, Lemma~2, this is stated for functions of the form
$h(z,s,\omega)=f(z)Y_s(\omega)$, where it is assumed that $f$ is
continuous and $(Y_s)$ is predictable. Neither of the two assumptions
is really needed. Also note that the factor $2$ in the statement of
Lemma~2 in \cite{AT00} seems to be incorrect.
\begin{pf*}{Proof of Lemma \ref{lem:occ_times}}
Let $X_s:=B^\ssup{2}_s-B^\ssup{1}_s$. For $h(z,s,\omega)= \break f(z)\1
_{(a,b]}(s)g(\omega)$, with $f$ and $g$ measurable bounded and $0\le
a<b<\infty$, \eqref{eq:occ_times} holds by the occupation times
formula~\cite{RevuzYor}, Corollary~VI.1.6,
since for $\p$-almost all $\omega\in\Omega$,
\begin{eqnarray*}
\int_0^t h\bigl(X_s(\omega),s,
\omega\bigr) \,ds &=& \biggl(\int_{0}^{b\wedge t} f
\bigl(X_s(\omega)\bigr) \,ds - \int_0^{a\wedge
t}f
\bigl(X_s(\omega)\bigr) \,ds \biggr)g(\omega)
\\
&=& \biggl(\int_\R f(z)L_{b\wedge t}^{z,X}(
\omega) \,dz - \int_\R f(z)L_{a\wedge t}^{z,X}(
\omega) \,dz \biggr)g(\omega)
\\
&=&\int_\R f(z)\int_{a\wedge t}^{b\wedge t}\,dL_s^{z,X}(
\omega) \,dz g(\omega )
\\
&=&\int_\R\int_0^t
h(z,s,\omega) \,dL_s^{z,X}(\omega)\,dz.
\end{eqnarray*}

Let $\calC$ denote
the class of all functions $h$ of the above form. Clearly, $\mathcal
{C}$ is closed under multiplication and generates the product $\sigma
$-algebra $\calB_\R\otimes\calB_{\R^+}\otimes\calF$ on $\R
\times\R
^+\times\Omega$.
Moreover, let $\calH$ denote the space of all bounded measurable
functions $h\dvtx\R\times\R^+\times\Omega\to\R$ for which
\eqref
{eq:occ_times} holds. Since \eqref{eq:occ_times} is stable under linear
combinations and under monotone convergence, $\calH$ is a monotone
vector space of bounded measurable functions which contains $\calC$.
Hence by the monotone class theorem~\cite{RevuzYor}, Theorem~0.2.2,
$\calH$ contains all bounded measurable functions $h\dvtx\R\times\R
^+\times
\Omega\to\R$, which proves the assertion.
\end{pf*}
\end{appendix}

\section*{Acknowledgments} We would like to thank the referees for
reading our paper very carefully and suggesting several improvements
to the presentation.

\bibliographystyle{imsart-number}

%
%





\printaddresses
\end{document}